\documentclass[11pt]{article}
\usepackage{geometry}
\usepackage{graphicx}
\usepackage{amsmath}
\usepackage{amssymb}
\usepackage{amsthm}
\usepackage{amsfonts}
\usepackage{subfig}
\usepackage{bbm}
\usepackage{algorithm}
\usepackage{algorithmic}
\usepackage{epsfig}
\usepackage{color}
\usepackage{enumitem}
\usepackage{setspace}
\usepackage{sidecap}
\usepackage{subfig}

\numberwithin{equation}{section}

\newtheorem{theorem}{Theorem}[section]
\newtheorem{proposition}[theorem]{Proposition}

\newtheorem{assumption}[theorem]{Assumption}

\newcommand{\bthm}{\begin{theorem}}
\newcommand{\ethm}{\end{theorem}}

\newcommand{\bprop}{\begin{proposition}}
\newcommand{\eprop}{\end{proposition}}

\def\XXint#1#2#3{{\setbox0=\hbox{$#1{#2#3}{\int}$}
     \vcenter{\hbox{$#2#3$}}\kern-.5\wd0}}

\newcommand{\du}{\, \mathrm{d}}

\newtheorem{remark}{Remark}

\begin{document}

\title{Approximate global minimizers to pairwise interaction problems via convex relaxation}

\author{Mahdi Bandegi
\thanks{Dept. of Mathematical Sciences, NJIT, Newark, NJ, USA, mb495@njit.edu} 
\hspace{1.5cm} 
David Shirokoff
\thanks{Corresponding author. \newline
Dept. of Mathematical Sciences, NJIT, Newark, NJ, USA, david.g.shirokoff@njit.edu}}

    
\maketitle

\begin{abstract}
We present a new approach for computing approximate global 
minimizers to a large class of non-local pairwise interaction problems defined over 
probability distributions.  The approach predicts candidate global 
minimizers, with a recovery guarantee, that are sometimes exact, and often 
within a few percent of the optimum energy (under appropriate normalization 
of the energy).  The procedure relies on a convex relaxation of the pairwise 
energy that exploits translational symmetry, followed by a recovery procedure 
that minimizes a relative entropy.  Numerical discretizations of the 
convex relaxation yield a linear programming problem over convex cones that 
can be solved using well-known methods.  
One advantage of the approach is that it provides sufficient conditions
for global minimizers to a non-convex quadratic variational problem, in the 
form of a linear, convex, optimization problem for the auto-correlation of the 
probability density.
We demonstrate the 
approach in a periodic domain for 
examples arising from models in materials, social phenomena and
flocking.
The approach also exactly recovers 
the global minimizer when a lattice of Dirac masses solves the convex 
relaxation. 
An important by-product of the relaxation is a decomposition of 
the pairwise energy functional into the sum of a convex functional and 
non-convex functional.  We observe that in some cases, the non-convex
component of the decomposition can be used to characterize the support of 
the recovered minimizers.

\end{abstract}

\medskip

\noindent
{\bf Keywords:} Global minimizers, non-convex energy, pairwise interactions, 
convex relaxations, conic programming, semi-definite programming, flocking,
self-assembly.\\

\noindent
{\bf AMS Subject Classifications:}
 49M30, 49S05.


\section{Introduction}

In this paper, we present a new approach for computing candidate global 
minimizers to a class of energy functionals that arise as continuum 
approximations to a large collection of interacting particles.  
Although there are many models for how a collection of particles 
may interact, we consider here functionals corresponding to systems 
where particles interact only in pairs with each other. 
The resulting pairwise energy functionals are in 
general non-convex and quadratic, and may have multiple local minimizers, making the 
global optimization a potentially difficult problem.  
One could in principle numerically discretize the quadratic functional we 
consider, using
$n \gg 1$ spatial grid points, and arrive at a finite dimensional, quadratic optimization problem.
Unfortunately, minimizing such a discrete problem 
through currently known methods, see Sections 4--5 in \cite{HiriartUrrutySeeger2010} 
(and references within), is computationally prohibitive, and requires 
$\mathcal{O}(2^n)$ floating point operations.  
We stress that these computational costs are a currently known upper bound,
and future algorithms may improve upon them.

The idea in the new approach is to avoid minimizing the non-convex quadratic
functional, and instead minimize a linear convex functional that bounds the 
non-convex functional from below.  The solution to this lower bound 
problem then results in a new sufficient condition for global minimizers.
The advantage of this approach is that 
numerical discretizations of the linear, lower bound problem,
using $n$ spatial grid points,  require $\mathcal{O}(n)$ linear
constraints, and hence may be solved using $\mathcal{O}(n)$ 
floating point operators.  We then obtain candidate 
minimizers, not by minimizing the 
original energy, but rather choosing ones that try to satisfy this new 
sufficient condition.  If a candidate satisfies the sufficient condition
exactly, then we are guaranteed that it is a global minimum.  If 
a candidate minimizer does not satisfy the condition exactly, then
by virtue of the fact that the sufficient condition
provides a lower bound to the energy functional, we can
quantify a worst case estimate on the energy difference
between the candidate and global minimizer.

Although parts of the approach are numerical in nature, a
by-product of the analytic formulation is an optimal 
decomposition of the energy functional into the sum of a 
non-negative, non-convex functional, and a convex functional. 
The resulting convex functional in the decomposition is 
then highly reminiscent of a convex envelope. 
This decomposition will also help to explain the 
emergence of new length scales that characterize the patterns
of many interacting particles.

Pair interaction problems are ubiquitous throughout the sciences, appearing 
in problems ranging from electromagnetics, the weak interaction of nuclear 
matter \cite{SchneiderHorowitzHughtoBerry2013}, biological 
swarming \cite{BernoffTopaz2013, CraigTopaloglu2015, 
MogilnerKeshetBentSpiros2003, NouralishahiWuVandenberghe2008, 
BrechtUminskyKolokolnikovBertozzi2012, WaldspurgerAspremontMallat2015, 
WangMillerLizierProkopenkoRossi2012}, colloids, polymers 
\cite{Cotter1977, Onsager1949}, consensus \cite{MotschTadmor2014}, 
mathematical physics \cite{CarrilloSlepcevWu2016,LuOtto2015} and 
self-assembly \cite{HolmesGortlerBrenner2013, MarcotteStillingerTorquato2011, 
RechtsmanStillingerTorquato2005} to name a few.  In these systems, each 
particle exhibits and experiences a force from every other particle 
in the system. The resulting sum of the pairwise energies then promote the 
collective organization of matter into the formation of structures such as 
solids or crystalline lattices 
\cite{CohnKumarSchurmann2010, Suto2005, Suto2006, Suto2011}. 


Global minimizers or ground states for many particle systems play a key 
physical role as they often describe the most likely observed state at low 
temperatures, influence the structure of matter at high temperatures, 
and are also important for computing phase diagrams 
\cite{EmelianenkoLiuDu2006, HarveyErikssonOrbanChartrand2013}.  Dynamically, 
global minimizers appear as steady states to gradient flows, or as critical 
points to Hamiltonian systems, and therefore may play a role in characterizing
the long time behavior in some dynamical systems.


We consider problems motivated by a large number, $N \gg 1$,
of interacting particles, where a probability measure 
$\rho(\mathbf{x})\du\mathbf{x}$ is used to represent the distribution of 
particles. Here $\mathbf{x} \in \mathbbm{R}^d$ denotes the 
spatial coordinates in a dimension $d \geq 1$. For problems on a domain 
$\Omega \subseteq \mathbbm{R}^d$, we consider energy functionals that take the form
\begin{align} \label{Energy}
	\mathcal{E}(\rho) := 	\frac{1}{2} \int_{\Omega} \int_{\Omega}   
	\rho(\mathbf{x}) \rho(\mathbf{y}) W(\mathbf{x} - \mathbf{y}) \du \mathbf{x} \du \mathbf{y}.
\end{align}

In equation (\ref{Energy}), $\rho(\mathbf{x}) \du\mathbf{x}$ 
(resp.~$\rho(\mathbf{y}) \du\mathbf{y}$) is
interpreted as the fraction of particles in the vicinity of a point 
$\mathbf{x}$ (resp.~$\mathbf{y}$) in the domain $\Omega$. Hence, 
the energy (\ref{Energy}) is the double integral over all possible pairs 
of particles at locations $\mathbf{x}$ and $\mathbf{y}$, weighted by the 
\emph{interaction potential} $W(\mathbf{x}-\mathbf{y})$.  Physically, $W(\mathbf{r})$
typically represents the energy cost of having two particles separated by 
the vector $\mathbf{r}$.  
Due to the double integral in (\ref{Energy}) over all possible pairs of 
locations $\mathbf{x}$ and $\mathbf{y}$, we refer to 
$\mathcal{E}(\rho)$ as the \emph{pairwise energy}.

Formally $\rho(\mathbf{x}) \du\mathbf{x}$ will be taken as a probability measure, 
however for brevity we will suppress the $\du\mathbf{x}$ throughout the 
written text and write $\rho(\mathbf{x})$ with the understanding that 
$\rho(\mathbf{x})$ is a measure and includes $L^1(\Omega)$ probability densities
and non-classical functions such as a Dirac mass.  Without a loss 
of generality, the total mass $m$ of $\rho(\mathbf{x})$ is taken to be $1$:
\begin{align}\label{Mass}
	m := \int_{\Omega}\rho(\mathbf{x}) \du \mathbf{x} = 1.
\end{align}

If, $\rho(\mathbf{x})$ was normalized to $m \neq 1$ in equation (\ref{Mass}), 
i.e., as the total number of particles in the system $m = N$,  
then a re-scaling of $\tilde{\rho}(\mathbf{x}) = m^{-1}\rho(\mathbf{x})$ 
re-scales $\mathcal{E}(\rho) = m^2 \mathcal{E}(\tilde{\rho})$ 
by $m^{2}$. As a result, minimizing 
$\mathcal{E}(\tilde{\rho})$ over 
$\tilde{\rho}(\mathbf{x})$ with mass $1$ is equivalent to minimizing 
a re-scaled $\mathcal{E}(\rho)$ over $\rho(\mathbf{x})$ with mass $m$.  In general, the 
assumption of (\ref{Mass}), as opposed to a different value of $m$, does 
not alter the approach in this paper. 

\begin{remark}\label{SymmetricW}
	For the purposes of minimizing the energy (\ref{Energy}) on 
	$\Omega = \mathbbm{R}^d$, the interaction potential $W(\mathbf{x})$ 
	may be assumed to be mirror symmetric, i.e., even with respect 
	to the simultaneous negation of the coordinates, 
	for all $\mathbf{x} \in \mathbbm{R}^d$,
	$W(-\mathbf{x}) = W(\mathbf{x})$ where 
	$W(-\mathbf{x}) := W(-x_1, \ldots, -x_d)$. 
	If, for instance, $W(\mathbf{x})$ is not mirror symmetric, one may write	
	$W(\mathbf{x} ) = W_E(\mathbf{x}) + W_O(\mathbf{x})$ where
	$W_E(\mathbf{x})$ and $W_O(\mathbf{x})$ are the following even and odd components of 	$W(\mathbf{x})$:
	\begin{align}\nonumber
		W_E(\mathbf{x}) := \frac{1}{2}\Big( W(\mathbf{x} )  + W(-\mathbf{x} ) \Big), \quad
		W_O(\mathbf{x}) := \frac{1}{2}\Big( W(\mathbf{x} )  - W(-\mathbf{x} ) \Big). 
	\end{align}
	The function $W_O(\mathbf{x})$, when inserted into the energy integral 
	(\ref{Energy}), then integrates to zero by a change of variables:
	\begin{align} 
		\int_{\mathbb{R}^d} \int_{\mathbb{R}^d}   
		\rho(\mathbf{y}) W_{O}(\mathbf{x} - \mathbf{y}) \rho(\mathbf{x}) \du \mathbf{x} \du \mathbf{y} 
		= \frac{1}{2}\int_{\mathbb{R}^d} \int_{\mathbb{R}^d}   
		\rho(\mathbf{y})\Big( W(\mathbf{x} - \mathbf{y}) - W(\mathbf{y} - \mathbf{x})\Big) \rho(\mathbf{x}) \du \mathbf{x} \du \mathbf{y} 		
		= 0.
	\end{align}
	Hence, the energy $\mathcal{E}(\rho)$ in (\ref{Energy}) is the same 
	for all $\rho(\mathbf{x})$ regardless of whether $W(\mathbf{x})$ 
	or $W_E(\mathbf{x})$ is used.
	Therefore, one may assume that $W(\mathbf{x}) = W_E(\mathbf{x})$ is 
	the symmetric component of $W(\mathbf{x})$, even when $W(\mathbf{x})$
	is not mirror symmetric. Note that mirror symmetry does not constrain 
	$W(\mathbf{x})$ to be even symmetric in each individual component,
	i.e., in general one could have $W(x_1, -x_2) \neq W(x_1, x_2)$ 
	and still satisfy $W(-x_1, -x_2) = W(x_1, x_2)$. 
	The same results regarding mirror symmetry hold for the periodic 
	domain $\Omega = [0,1]^d$.
\end{remark}

The approach in this paper will focus on energies of the form (\ref{Energy}),
however we now briefly discuss how the energy $\mathcal{E}(\rho)$, which is defined for 
probability measures $\rho(\mathbf{x})$, can be related 
to the energy of a discrete particle system.  For example, 
restricting $\rho(\mathbf{x})$ in the energy $\mathcal{E}(\rho)$ to a sum
of $N$ Dirac masses can be interpreted as the energy of an $N$ particle 
system.  Specifically, substituting an ansatz of Dirac masses 
into the energy (\ref{Energy}) yields: 
\begin{align}
	\mathcal{E}_N(\mathbf{x}_1, \mathbf{x}_2, \ldots, \mathbf{x}_N) 
	:= \mathcal{E}(\rho_N), \hspace{5mm} \textrm{where} \hspace{5mm}
	\rho_N(\mathbf{x}) = \frac{1}{N} \sum_{j = 1}^{N} \delta(\mathbf{x} - \mathbf{x}_j).
\end{align}
By direct calculation, and assuming that $W(\mathbf{x})$ is continuous
so that the integration against Dirac masses is well defined, one has
\begin{align} \label{DiscretePairwiseEnergy}
	\mathcal{E}_N(\mathbf{x}_1, \mathbf{x}_2, \ldots, \mathbf{x}_N) 
	= \frac{1}{2 N^2}\sum^N_{i=1}
	\sum_{j = 1}^N W(\mathbf{x}_i - \mathbf{x}_j).
\end{align}
Within the double-sum (\ref{DiscretePairwiseEnergy}) are $N$ terms where $i = j$
that result in a total contribution of $(2N)^{-1} W(\mathbf{0})$ to the overall energy
$\mathcal{E}_N$.  Provided $W(\mathbf{0}) < \infty$ is bounded\footnote{
Many interaction potentials are not bounded at $\mathbf{x} = \mathbf{0}$, 
see for instance the divergent power law potentials in 
\cite{ChoksiFetecauTopaloglu2013, SimioneSlepcevTopaloglu2015}.} at the origin, 
$\mathcal{E}_N$ can be identified as the energy of $N$ discrete 
interacting particles--interacting with the same interaction potential $W(\mathbf{x})$
as in (\ref{Energy}).  The calculation also shows that minimizing 
$\mathcal{E}(\rho)$ over probability measures $\rho(\mathbf{x})$ includes the 
energies $\mathcal{E}_N$ of all possible arrangements of $N$ particles, for any $N\geq 1$.


\begin{remark} (Numerical example: a particle gradient flow for a periodic 
Morse potential)
Arrangements of particles that minimize 
the collective energy $\mathcal{E}_N$ may form patterns on length scales 
that are not readily identifiable from the interaction energy $W(\mathbf{x})$. 
Figure \ref{DiscreteParticles} shows the time evolution for a collection 
of randomly distributed particles undergoing a one dimensional gradient flow
governed by the system of ordinary differential equations:
\begin{align}\label{DiscreteGradFlow}
 \frac{dx_j}{dt} = -\nabla_{x_j} \mathcal{E}_N, \quad 1 \leq j \leq N. 
\end{align}
Here the periodic Morse potential (\ref{PeriodicMorse}) (with parameters 
$\sigma = 0.1$, $(L, G) = (1.2, 0.9)$) was used for $\mathcal{E}_N$, 
while the initial particle 
positions, i.e., $x_j$ at $t = 0$ for $1 \leq j \leq N$, was
taken to be randomly distributed in the domain $[0,1]$, sampled from a uniform 
probability distribution. A total of $N = 400$ particles was used in the 
simulation, however the same histogram shape in Figure \ref{DiscreteParticles} was
observed in repeated trials, for different values of $N = 200$ and $600$, and also
for (slightly perturbed) uniformly distributed initial data.  
Figure \ref{DiscreteParticles} also shows the histogram of particle positions as 
$t \rightarrow \infty$, demonstrating that the particles coalesce into a region with a 
width of $\sim 0.159$ units.  
\begin{figure}[htb!] 
	\centering
	\subfloat[(a) Gradient flow]{\includegraphics[width = 0.36\textwidth]
	{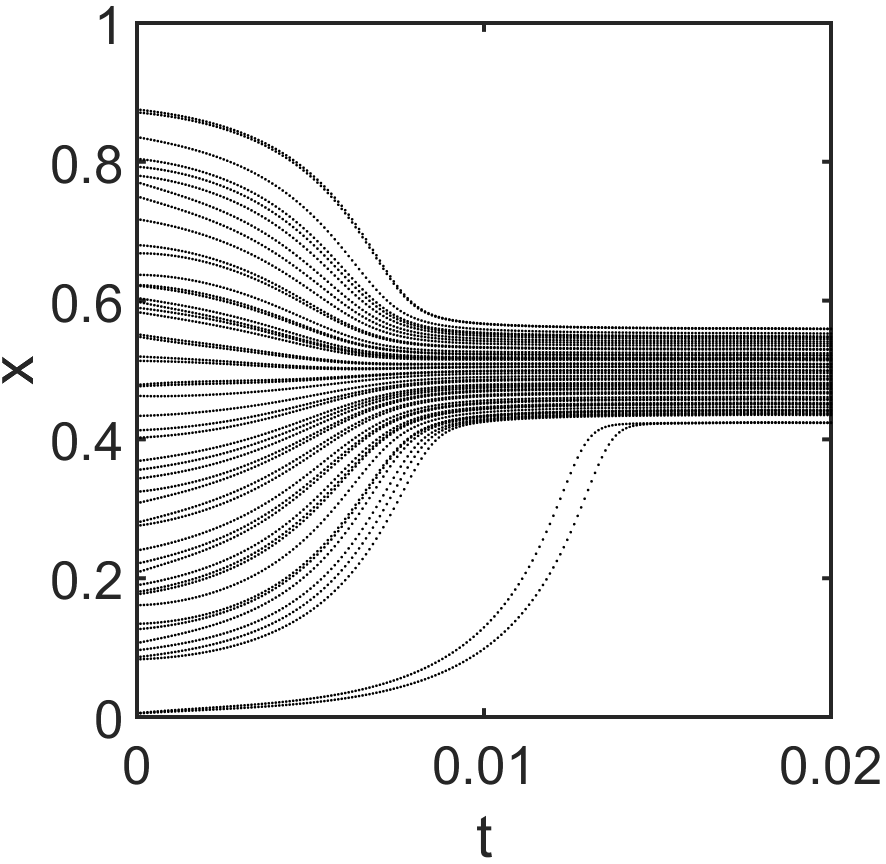}}
    \subfloat[(b) Particle density]{\includegraphics[width = 0.33\textwidth]
    {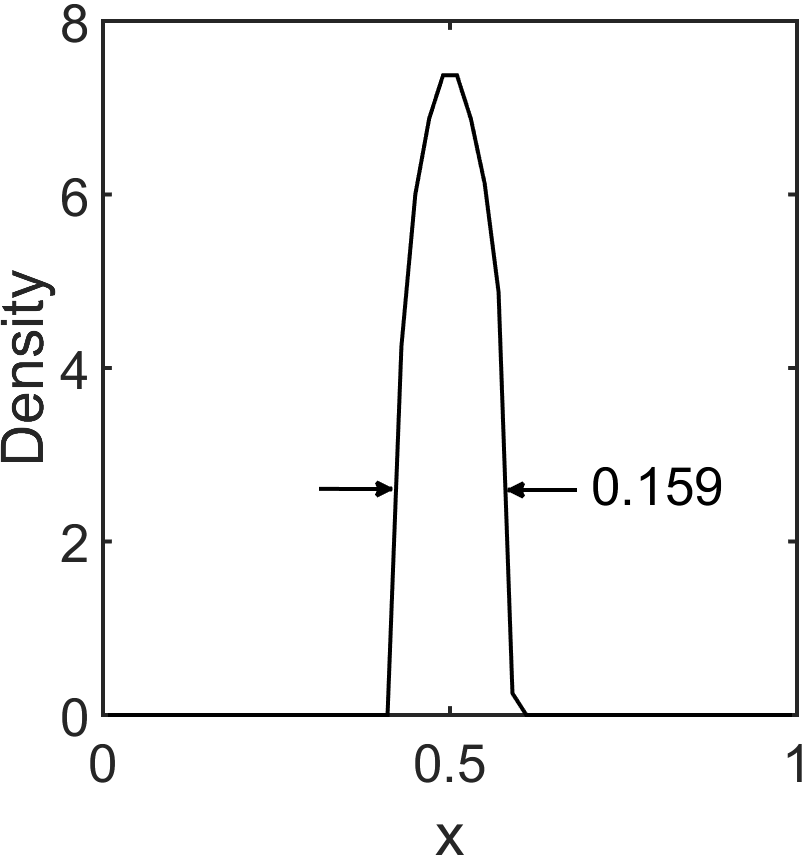} }
    \caption{(a) Time evolution of equation (\ref{DiscreteGradFlow}) for
    the interaction potential (\ref{PeriodicMorse}) and $N = 400$ discrete particles 
    (only 50 shown), towards a critical point of $\mathcal{E}_N$.  (b) 
    The density distribution (using 50 bins) of the discrete particles at 
    steady state  (i.e., as $t \rightarrow \infty$) from part (a).  
    The density is normalized to have area one.  
    The support of the density has an approximate 
    width of $\sim 0.159$ (computed as the difference between the maximum and 
    minimum particle locations) 
    which is not immediately related to the interaction 
    potential $W(x)$. The width, however, will emerge as the length scale in 
    the optimal dual decomposition for the energy presented in Section 
    \ref{sec_dualdecomposition} (see also Figure \ref{ResultsPeriodicMorse}a).}
 \label{DiscreteParticles}
 \end{figure}
\end{remark}

Recently, computational methods based on convex relaxations or lower bounds 
have been used to estimate low energy states and phase diagrams in material 
science.  For instance, \cite{ShirokoffChoksiNave2014} computed convex lower 
bounds to estimate the order-disorder phase transition in energy functionals 
containing double-wells.  Meanwhile, \cite{LiuLaiOsher2015} used relaxations 
to compute approximate density matrices for quantum systems at zero 
temperature, while 
\cite{LocatelliSchoen2003, LocatelliSchoen2013, MaranasFloudas1994} 
have computed molecular structures.


The approach we present for computing approximate global minimizers is 
similar in spirit to other state of the art algorithms currently used in 
optimization theory and integer programming that exploit matrix semi-definite 
programming (SDP) (see also 
\cite{BoydVandenberghe2004, Floudas1999, NesterovWolkowiczYe2000} for a 
discussion on SDPs and relaxations).  For example, semi-definite based convex 
relaxations represent some of the best known polynomial time algorithms for 
computing approximate solutions to the graph partitioning problem 
\cite{GoemansWilliamson1995} and matrix completion problem 
(\emph{Netflix prize}) \cite{CandesTao2010}. They have also been used in data 
science to approximately solve the k-disjoint clique problem 
\cite{AmesVavasis2014} and blind deconvolution \cite{AhmedRechtRomberg2013}, 
while other relaxations have been used to characterize the sparsest element 
in a discrete set \cite{DemanetHand2014}.


Our paper is presented as follows: Section 2 introduces the general problem 
and definition of the recovery guarantee. In Section 3 we formulate the 
convex relaxation, while in Section 4 we outline the recovery procedure. 
Section 5 contains a detailed description of the dual problem and resulting 
optimal decomposition of the pairwise energy. Sections 
7 and 8 present numerical examples in dimensions 
one and two respectively. Finally, Appendix A contains information 
on known cases where the convex relaxation is exact, Appendix B contains 
numerical details, while 
Appendix C characterizes solutions to the relaxed problem
that take the form of three Dirac masses.

\bigskip

\section{Problem formulation and preliminaries}\label{Sec:preliminaries}

Consider a periodic domain $\Omega = [0,1]^d$ in dimension $d$ 
(eventually taken to be $d = 1, 2$ in 
Sections \ref{sec:NumericalResults_d1}--\ref{sec:NumericalResults_d2}),
and interaction energy $W(\mathbf{x})$. We are interested in the problem of 
finding global minimizers to the pairwise energy (\ref{Energy}):
\begin{align} \nonumber
 	 (P) \quad \quad &\textrm{Minimize } \;\; 
 	 \frac{1}{2} \int_{\Omega} \int_{\Omega} 
 	 \rho(\mathbf{x}) \rho(\mathbf{y}) W(\mathbf{x} - \mathbf{y})  
 	 \du \mathbf{x} \du \mathbf{y}, 
\end{align}
	\[		
	\textrm{over probability measures $\rho(\mathbf{x}) \in \mathcal{C}_1$ with } 
	\int_{\Omega} \rho(\mathbf{x}) \du \mathbf{x} = 1.
	\]
	Here we have introduced $\mathcal{C}_1$ as a convex cone to characterize 
	non-negative measures\footnote{In the definition of $\mathcal{C}_1$, 
	$C^0(\Omega)$ is the space of periodic continuous functions on $\Omega$ endowed with 
	the sup norm.  Since $\Omega = [0,1]^d$ is compact, the functions 
	$u \in C^0(\Omega)$ are bounded and also form a Banach space.  
	The notation here, $\langle f, u\rangle$, 
	represents the pairing of elements $f$ in the dual space
	$C^{0}(\Omega)'$ with continuous functions $u$ that are elements of
	$C^0(\Omega)$ (See Chapter 5 in \cite{Folland1999} for a general 
	discussion on Banach spaces).  In addition, the Riesz representation
	theorem for bounded continuous functions (see Chapter 7 in \cite{Folland1999},
	or Chapter 2 in \cite{Rudin87})
	shows that elements in 
	$C^{0}(\Omega)'$ can be identified as non-negative \emph{Borel measures}, 
	which justifies the integral representation of $\langle f, u\rangle$.}:
	\[ 
	\mathcal{C}_1 := \Big\{ f \in C^{0}(\Omega)' : \langle f, u \rangle \geq 0 
	\textrm{ for all } u \in C^{0}(\Omega) \textrm{ with } u(\mathbf{x}) \geq 0 \Big\}, 
	\]
	where 
	\[
	\langle f, u \rangle = \int_{\Omega} u(\mathbf{x}) f(\mathbf{x}) \du \mathbf{x}
	\]
	is the integral of the continuous function $u(\mathbf{x})$ against the measure
	$f(\mathbf{x}) \du \mathbf{x}$. 	
	In the case when $\rho(\mathbf{x})$ is a classical function, we may equivalently replace 
$\mathcal{C}_1$ with $\rho(\mathbf{x}) \geq 0$ for all $\mathbf{x} \in \Omega$. 
In the problem (P), we further assume that $W(\mathbf{x})$ satisfies the 
following properties
\begin{enumerate}	
	\item[(W1)] Mirror symmetric: $W(\mathbf{x}) = W(-\mathbf{x})$, holds for all 
	$\mathbf{x}$ (See Remark \ref{SymmetricW} for justification).
	\item[(W2)] Continuous on $\Omega$. 
	\item[(W3)] Periodic with period $1$: $W(\mathbf{x} + \mathbf{k}) = W(\mathbf{x})$, 
	for all $\mathbf{x} \in \Omega$ and integer vectors $\mathbf{k} \in \mathbbm{Z}^d$.
	\item[(W4)] Normalized with mean zero: $\int_{\Omega} W(\mathbf{x}) \du \mathbf{x} = 0$. 
	In such a case, the minimum to (P) is at most zero since $\mathcal{E}(1) = 0$.
	Note that one can always add, without loss of generality, a constant to $W(\mathbf{x})$. 
\end{enumerate}


\begin{remark}
	For numerical simplicity we have intentionally limited the problem (P) to 
	continuous interaction potentials $W(\mathbf{x})$ on periodic domains 
	$\Omega$. Many of the results presented here apply to other 
	domains as well, including the sphere or $\mathbb{R}^d$.  
	For instance, when $\Omega = \mathbb{R}^d$, one may still define 
	a convex relaxation for problem (P). In this case, the countable
	wavenumbers (i.e. Fourier series) used to define the relaxation in 
	Section~\ref{Sec_convex_r} for the periodic domain $[0,1]^d$ 
	will be replaced 
	with a continuous set of wavenumbers (i.e. Fourier transform). 
	Additional difficulties, not encountered here, may arise in the numerical 
	solution when
	the domain is unbounded.
\end{remark}

For the problems we consider where $\Omega$ is compact, (P) 
admits a global minimum $\mathcal{E}_0 := \mathcal{E}(\rho_0)$, 
achieved
by some probability measure $\rho_0(\mathbf{x})$. Note that $\rho_0(\mathbf{x})$ 
is not unique since $\mathcal{E}(\rho)$ is invariant under translations so 
that $\rho_0(\mathbf{x} + \mathbf{s})$ is also a global minimum for any 
$\mathbf{s}$.  For 
non-compact domains, the existence and uniqueness 
\cite{CarrilloChipotHuang2014, ChoksiFetecauTopaloglu2013, SimioneSlepcevTopaloglu2015} 
(up to translations and rotations) of global minimizers is more subtle since 
mass can be spread arbitrarily far apart (see also
\cite{BalagueCarrilloLaurentRaoul2013a, BalagueCarrilloYao2014, CanizoCarrilloPatacchini2015} 
for results on the structure of minimizers).


We now review several necessary conditions imposed by the 
first and second variation of $\mathcal{E}(\rho)$ that a global minimizer $\rho_0(\mathbf{x})$ 
must satisfy (see \cite{BernoffTopaz2013} for a discussion and
\cite{CarrilloFigalliPatacchini2016} for a rigorous treatment).
Firstly, a candidate global minimizer $\rho^*(\mathbf{x})$ satisfies 
the first order necessary conditions 
if the first variation of $\mathcal{E}(\rho)$:
\begin{align}
	\Lambda(\mathbf{x}) := 
	\int_{\Omega} W(\mathbf{x} - \mathbf{y}) \rho^*(\mathbf{y}) \du \mathbf{y}, 
\end{align}
satisfies
\begin{align} \label{KKT_for_P} 
	\Lambda(\mathbf{x}) &= 2\mu , \quad \hspace{0mm} \textrm{for all } 
	\mathbf{x} \in S_* := \textrm{supp}(\rho^*). \\
	\label{ExtendedFirstOrder_condition}
	\Lambda(\mathbf{x}) &\geq 2\mu, \quad \textrm{for all } 
	\mathbf{x} \in \Omega,  \; (\textrm{Including } \mathbf{x} \notin S_*).	
\end{align}
Here $\textrm{supp}(f)$ is the support of $f(\mathbf{x})$, i.e., the set 
where $f(\mathbf{x})$ does not vanish\footnote{If $f(\mathbf{x})$ is a 
continuous function on $\Omega$, then 
$\textrm{supp}(f) = \textrm{cl}\{ \mathbf{x} : f(\mathbf{x}) \neq 0\}$,
where $\textrm{cl}$ denotes the closure. See Chapter 7 in \cite{Folland1999}
for the definition when $f(\mathbf{x})$ is a measure.}, while 
$\mu \in \mathbbm{R}$ is a Lagrange multiplier constant.
Note that multiplying (\ref{KKT_for_P}) through by 
$\rho^*(\mathbf{x})$ and integrating over $\Omega$ shows 
that $\mathcal{E}(\rho^*) = \mu$.  Hence, if $\rho^*(\mathbf{x})$
satisfies the first order condition (\ref{KKT_for_P}), then $\rho^*(\mathbf{x})$ has
energy $\mu$.  As a result, setting $\mu = \mathcal{E}_0$ in 
equation (\ref{KKT_for_P}) shows that the global minimizer 
$\rho_0(\mathbf{x})$ satisfies: 
\begin{align}\label{Sufficient_Cond}
	\int_{\Omega} W(\mathbf{x} - \mathbf{y}) \rho_0(\mathbf{y}) \du \mathbf{y} = 2\mathcal{E}_0, 
	\quad \hspace{0mm} \textrm{for all } 
	\mathbf{x} \in S_0 := \textrm{supp}(\rho_0). 
\end{align}
One difficulty with using the condition (\ref{Sufficient_Cond}) to solve 
for $\rho_0(\mathbf{x})$ is that both $\mathcal{E}_0$ and $S_0$ are not known a priori.  
As implied by the integral equation in \cite{BernoffTopaz2013} (see Remark 2.5 in \cite{CarrilloFigalliPatacchini2016} for a rigorous treatment),
consideration of the second variation of $\mathcal{E}(\rho)$ will show 
that knowledge of $S_0$ alone will be sufficient to compute $\rho_0(\mathbf{x})$ 
through a convex optimization problem. 
Specifically, a candidate $\rho^*(\mathbf{x})$ satisfies the second
order necessary conditions for a global minimum if the second variation
is non-negative (within the class of perturbations that make the first 
variation vanish):
\begin{align}
\label{SecondOrder_condition}
	\mathcal{E}(f) &\geq 0, \textrm{ for finite measures } f(\mathbf{x}),  
	\textrm{ with } \int_{\Omega} f(\mathbf{x}) \du \mathbf{x} = 0, 
	\textrm{ and } \textrm{supp}(f) \subseteq S_*. 
\end{align}
Here the class of $f(\mathbf{x})$'s in (\ref{SecondOrder_condition}) are 
exactly the measures that when integrated against $\Lambda(\mathbf{x})$ 
vanish. Equation (\ref{SecondOrder_condition}) also implies the following
remark regarding the convexity of $\mathcal{E}(\rho)$ when restricted to 
probabilities having supports in $S_*$.   
%

\begin{remark} (The importance of $S_0$) \label{E_Convex_on_Z}
	Examining the necessary condition in (\ref{SecondOrder_condition})
	when $\rho_0(\mathbf{x})$ is a global minimum, one has the following
	observations:
	\begin{enumerate}
		\item[(i)] Condition (\ref{SecondOrder_condition}) implies that: 
		\[
			\mathcal{E}(\rho) \textrm{ is convex when restricted to } 
			\mathcal{B} := \Big\{ \rho(\mathbf{x}) \in \mathcal{C}_1, 
	\int_{\Omega} \rho(\mathbf{x}) \du \mathbf{x} = 1, \; \mathrm{supp}(\rho) \subseteq S_0\Big\}.		
		\]						
		\item[(ii)] Knowledge of the support of $\rho_0(\mathbf{x})$, i.e., the set $S_0$, 
		implies that (P) may be formulated as a convex optimization problem. 
	\end{enumerate}
	Note that $\mathcal{B}$ is a convex set.  To show (i), take
	any $\rho_1(\mathbf{x}), \rho_2(\mathbf{x}) \in \mathcal{B}$ and 
	set $f(\mathbf{x}) := \rho_1(\mathbf{x}) - \rho_2(\mathbf{x})$.  
	Therefore, $f(\mathbf{x})$ has support in $S_0$ and satisfies the criteria in 
	(\ref{SecondOrder_condition}).  Then,
	by direct calculation using the fact that $\mathcal{E}(\rho)$ is quadratic, one
	has for any $0 \leq \lambda \leq 1$,
	\begin{align} \label{ConvexSets}
		0 \leq  (1-\lambda)\lambda \; \mathcal{E}(f) = \lambda \mathcal{E}(\rho_1) + (1-\lambda) \mathcal{E}(\rho_2) 
	-\mathcal{E}(\lambda \rho_1 + (1-\lambda) \rho_2 ), \\ \nonumber
	\Rightarrow \mathcal{E}(\lambda \rho_1 + (1-\lambda) \rho_2 ) \leq \lambda \mathcal{E}(\rho_1) + (1-\lambda) \mathcal{E}(\rho_2). 
	\end{align}			
%
	The inequality (\ref{ConvexSets}) shows that $\mathcal{E}(\rho)$ is convex when
	restricted to probabilities in $\mathcal{B}$.  For (ii), note that 
	$\rho_0(\mathbf{x}) \in \mathcal{B}$, so that restricting the optimization of 
	$\mathcal{E}(\rho)$ in (P) to the 
	space $\mathcal{B}$ produces the same minimum $\mathcal{E}_0$.  Moreover,
	(P) then becomes the following convex problem:
	\begin{align}\nonumber
		&\min \mathcal{E}(\rho), \;\;\mathrm{ subject}\;\mathrm{to}\; 
		\rho(\mathbf{x}) \in \mathcal{B}.
	\end{align}	
\end{remark}
Remark \ref{E_Convex_on_Z} highlights the importance of finding sets $S_*$
where $\mathcal{E}(\rho)$ is convex. In our approach, we do not have a 
proof that the recovered candidate
minimizers satisfy the first and second order necessary conditions, however
in Section \ref{sec_dualdecomposition} we provide new sufficient conditions 
for $\mathcal{E}(\rho)$ to be convex when $\textrm{supp}(\rho) \subseteq S_*$.  Sections \ref{sec:NumericalResults_d1}--\ref{sec:NumericalResults_d2} 
 then demonstrate that our recovered minimizers often satisfy this
 new sufficient condition.


A common practice in optimization theory is to guarantee that a candidate 
minimizer (or maximizer) is within a factor $\alpha$ of the optimal value.  
Here we say that an approximate minimizer $\rho^*(\mathbf{x})$ to problem (P) 
has an $(\alpha, \nu)$ guarantee, where 
$0 \leq \alpha \leq 1$, $\nu \geq \mathcal{E}_0$, if the shifted energy 
$\mathcal{E}(\rho^*) - \nu$ is optimal to within a factor of $\alpha$: 
\[ 
(\mathcal{E}_0 - \nu) \leq \mathcal{E}(\rho^*) - \nu \leq \alpha (\mathcal{E}_0 - \nu ). 
\]

In the context of gradient flows on $\mathcal{E}(\rho)$, one may always add an 
arbitrary constant to the underlying potential $W(\mathbf{x})$, and hence 
$\mathcal{E}(\rho)$, without effecting the dynamics of $\rho(\mathbf{x})$.  
To eliminate the ambiguity of adding such an arbitrary constant, we introduce 
the shift $\nu = \mathcal{E}(\rho_{ref})$ as a reference energy with respect
 to a base probability $\rho_{ref}(\mathbf{x})$. 


Clearly, if $\alpha = 1$ with any $\nu$ then 
$\mathcal{E}(\rho^*) = \mathcal{E}_0$, and hence $\rho^*(\mathbf{x})$ is a global 
minimizer.  In this case, we drop the notation $\nu$ and simply say the 
solution $\rho^*(\mathbf{x})$ is optimal with a guarantee $\alpha = 1$. In the numerical 
section of this paper we always report an $\alpha$ guarantee with $\nu = 0$. 
Due to the normalization (W4), $\nu = \mathcal{E}(1) = 0$ corresponds to the 
constant state $\rho_{ref}(\mathbf{x}) = 1$.


In general, problem (P) is difficult to solve since the energy $\mathcal{E}(\rho)$ 
is a non-convex functional of $\rho(\mathbf{x})$.  In the next section we will 
show how to replace (P) with a convex relaxation (R) that is more amenable to analysis. 
We will:
\begin{enumerate}
	\item Formulate a convex relaxation (R) of (P).
	\item Solve the relaxation (R) using efficient linear programming (LP) algorithms.
	\item Recover a candidate minimizer from (R) using minimal points of the 
	Kullback-Leibler divergence, and report an $(\alpha, \nu)$ guarantee for the 
	candidate minimizer (with $\nu = 0$).
\end{enumerate}


\section{The Convex relaxation}\label{Sec_convex_r}
The purpose of this section is to formulate a convex relaxation of (P) 
that takes the form of a constrained linear optimization problem.  The linear
optimization problem may then be numerically approximated and solved using 
linear programming techniques. 

To obtain the relaxation, we first rewrite $\mathcal{E}(\rho)$ by performing a 
coordinate change of variables in the integral. Letting 
$\mathbf{s} = \mathbf{x} - \mathbf{y}$, 
\begin{align} \nonumber
	\mathcal{E}(\rho) &= \frac{1}{2} \int_{\Omega} \int_{\Omega} 
	\rho(\mathbf{x}) \rho(\mathbf{x} + \mathbf{s})  W(\mathbf{s}) 
	\du \mathbf{x} \du \mathbf{s} 
	= \frac{1}{2} \langle F, W\rangle, \\ \nonumber 
	F(\mathbf{s}) &:= \int_{\Omega} \rho(\mathbf{x}) \rho(\mathbf{x} 
	+ \mathbf{s}) \du \mathbf{x} = \rho \circ \rho.
\end{align}
Here we have introduced $F(\mathbf{x})$ as the auto-correlation of 
$\rho(\mathbf{x})$, along with a 
shorthand binary operator notation $\circ$.  In addition, we assume that
$\rho(\mathbf{x})$ is defined periodically on $\Omega$ so that 
$F(\mathbf{x})$ is also periodic on $\Omega$. 

The original problem (P) can then be understood as minimizing a linear 
functional $\langle F, W \rangle$, over the space of elements 
$F(\mathbf{x}) \in \mathcal{A}$ that arise as the auto-correlations of probabilities:
	\[ 
	\mathcal{A} := \Big\{ F : F(\mathbf{s}) = \int_{\Omega} 
	\rho(\mathbf{x}) \rho(\mathbf{x} + \mathbf{s}) \du \mathbf{x}, \; 
	\textrm{such that }\rho \in \mathcal{C}_1, 
	\int_{\Omega} \rho(\mathbf{x}) \du \mathbf{x} = 1 \Big\}.
	\]
	
  We will show below that $\mathcal{A}$ is not a convex space 
  (See Remark~\ref{Rmk:NonconvexEx}).  Hence, we 
  have reformulated the original problem (P) of minimizing a non-convex 
  objective functional over a convex set, to the minimization of a linear, 
  convex functional over a non-convex set. Our goal is now to relax the admissible 
  space of functions $\mathcal{A}$ to a convex set. Ideally, one would like 
  to use the smallest convex relaxation, i.e., the convex hull of $\mathcal{A}$, 
  however we use a space of convex cones that may be exploited in 
  subsequent numerical computations.  Specifically, since $F(\mathbf{x})$ 
  is defined in the periodic domain $\Omega$, it is natural to consider 
  representations as a Fourier series.   
  The following proposition, which characterizes several well-known properties 
  of auto-correlations, will play an important role in defining the relaxation.

\begin{proposition}\label{Prop_properties_of_A}
	(Properties of $\mathcal{A}$)  
	Given any $F(\mathbf{x}) \in \mathcal{A}$, the following properties 
	hold:
	\begin{enumerate}
	\item[(A1)]	$F(\mathbf{x})$ is non-negative, i.e. for any non-negative continuous 
	function $u(\mathbf{x}) \geq 0$, $\langle F, u \rangle \geq 0$. 
	\item[(A2)] $F(\mathbf{x})$ integrates to one: $\langle F, 1 \rangle = 1$.
	\item[(A3)] $F(\mathbf{x})$ is mirror symmetric about the origin, i.e., 
	$F(-\mathbf{x}) = F(\mathbf{x})$, corresponding to zero sine modes. For every 
	$\mathbf{k} \in \mathbbm{Z}^d$, $\mathbf{k} \neq \mathbf{0}$: 
	$\langle F, \sin(2\pi \mathbf{k} \cdot \mathbf{x} ) \rangle = 0$. 
	\item[(A4)] $F(\mathbf{x})$ has non-negative cosine modes. For every 
	$\mathbf{k} \in \mathbbm{Z}^d$, $\mathbf{k} \neq \mathbf{0}$: 
	$\langle F, \cos(2\pi \mathbf{k} \cdot \mathbf{x} ) \rangle \geq 0$. 
	\end{enumerate}
	Here $\mathbbm{Z}^d$ is the set of integers defined
	by 
	\[
	\mathbbm{Z}^d = \big\{ (n_1,\ldots, n_d) : \textrm{ for integers } 
	n_j, \; 0\leq j \leq d \big\}.\]
	Note that values of $-\mathbf{k}$ in properties (A3)--(A4)
	characterize the same constraints as $\mathbf{k}$, and are therefore redundant.
	We will, however retain all $\mathbf{k} \in \mathbb{Z}^d$ to simplify subsequent
	notation. 
\end{proposition} 
\begin{proof}  
	The proof of (A1)--(A4) is straightforward and done by a direct calculation 
	of the appropriate integrals $\langle \cdot, \cdot \rangle$. 
	If $F(\mathbf{x}) \in \mathcal{A}$, then 
	$F(\mathbf{s}) = \int_{\Omega} \rho(\mathbf{x}) \rho(\mathbf{x}+\mathbf{s})\du \mathbf{x}$ 
	for some $\rho(\mathbf{x}) \in \mathcal{C}_1$. 
	The integral $\langle F, u\rangle$ can then be written as:
	\[
	\langle F, u \rangle = \int_{\Omega}\int_{\Omega}\rho(\mathbf{x})\rho(\mathbf{y})
	u(\mathbf{x}-\mathbf{y}) \du \mathbf{x} \du \mathbf{y} =	
	\langle \rho, U \rangle, \textrm{ where } 
	U(\mathbf{x}) := 
	\int_{\Omega} \rho(\mathbf{y})u(\mathbf{x}-\mathbf{y}) \du \mathbf{y}.
	\]
	For (A1), take any continuous, non-negative function $u(\mathbf{x}) \geq 0$
	to integrate against $F(\mathbf{x})$. 
	Then, since $\rho(\mathbf{x}) \in \mathcal{C}_1$, the function
	$U(\mathbf{x}) \geq 0$ 
	is non-negative, and also continuous since it is a convolution. 
	Hence, integrating $U(\mathbf{x})$ against $\rho(\mathbf{x})$ is also 
	non-negative, implying:
	$\langle F, u \rangle~=~\langle \rho, U\rangle \geq 0$.\\
	For (A2), taking $u(\mathbf{x}) = 1$ in the definition for $U(\mathbf{x})$ 
	implies that $U(\mathbf{x}) = 1$. It then follows that
	$\langle F, 1 \rangle = \langle \rho, 1 \rangle = 1$. \\
	For (A3), integrating $F(\mathbf{x})$ against any sine mode, 
	$\sin(2\pi \mathbf{k}\cdot \mathbf{x})$, yields:
	\begin{align*}
	\langle F, \sin(2\pi \mathbf{k} \cdot \mathbf{x} ) \rangle 
	&= \int_{\Omega} \int_{\Omega} \rho(\mathbf{x}) \rho(\mathbf{y}) 
	\Big(\sin(2\pi\mathbf{k}\cdot \mathbf{x})\cos(2\pi\mathbf{k}\cdot \mathbf{y}) 
	- \sin(2\pi\mathbf{k}\cdot \mathbf{y})\cos(2\pi\mathbf{k}\cdot \mathbf{x})\Big)
	\du \mathbf{x} \du \mathbf{y} \\
	&= 0.
	\end{align*}
	Note also that a similar calculation shows that $F(\mathbf{x})$ is mirror 
	symmetric, i.e., for any continuous function $u(\mathbf{x})$, one has 
	$\langle F(-\mathbf{x}), u(\mathbf{x})\rangle 
	= \langle F(\mathbf{x}), u(\mathbf{x})\rangle$.  Here $(-\mathbf{x})$,
	denotes the simultaneous negation of all coordinates (see also
	Remark \ref{SymmetricW}). \\
	Finally, for (A4), integrating $F(\mathbf{x})$ against any cosine mode, 
	$\cos(2\pi \mathbf{k}\cdot \mathbf{x})$, yields:
	\begin{align*}
	\langle F, \cos(2\pi \mathbf{k} \cdot \mathbf{x} ) \rangle 
	&= \int_{\Omega} \int_{\Omega} \rho(\mathbf{x}) \rho(\mathbf{y}) 
	\cos(2\pi \mathbf{k} \cdot (\mathbf{x}-\mathbf{y}) ) 
	\du \mathbf{x} \du \mathbf{y}, \\ \nonumber
	&= \left|\langle \rho, \cos(2\pi \mathbf{k} \cdot \mathbf{x})\rangle\right|^2 
	+  \left|\langle \rho, \sin(2\pi \mathbf{k} \cdot \mathbf{x})\rangle\right|^2\geq 0.
	\end{align*}
\end{proof}

\begin{remark} \label{Rmk:NonconvexEx}
	(The set $\mathcal{A}$ is not convex)
	To show that $\mathcal{A}$ is not convex take 
	$f_1(x) = 1 + \cos(2\pi x)$ and $f_2(x) = 1 + \cos(2n\pi x)$ on 
	$\Omega = [0, 1]$, where $n \gg 1$ is a large integer. The convex 
	combination of 
	\[ 
	\lambda (f_1 \circ f_1) + (1-\lambda) (f_2 \circ f_2) 
	= 1 + \frac{1}{4}\cos(2\pi x) + \frac{1}{4}\cos(2n\pi x), 
	\] 
	when $\lambda = \frac{1}{2}$, must come from an auto-correlation
	of a function taking the form (with arbitrary phases $\varphi_{1}, \varphi_2$) 
	\[ 
	f_3(x) = 1 + \frac{1}{\sqrt{2}}\cos(2\pi x - \varphi_1) 
	+ \frac{1}{\sqrt{2}}\cos(2n \pi x - \varphi_2).
	\]
	Choosing $n$ large enough, the minimum value of $f_3(x)$, regardless 
	of the values $\varphi_1$, $\varphi_2$, can be made arbitrarily 
	close to $1-\sqrt{2} < 0$. Hence, for sufficiently large $n$,
	there is no non-negative function $f_3(x)$ with auto-correlation 
	$(\lambda f_1 \circ f_1 + (1-\lambda) f_2 \circ f_2)$. 
\end{remark}

Properties (A1)--(A2) characterize $F(\mathbf{x}) \in \mathcal{A}$ as a 
probability measure,
and therefore show that the set $\mathcal{A}$ is a subset of the convex cone 
$\mathcal{C}_1$, i.e., $\mathcal{A} \subseteq \mathcal{C}_1$. 
Properties (A3)--(A4), are related to a standard result in signal 
processing--that the Fourier series of an auto-correlation is a 
\emph{power spectrum}.  In the case
at hand, (A3)--(A4) motivate the definition of a second convex cone, 
$\mathcal{C}_2$, defined by measures having non-negative cosine 
modes and zero sine modes:
\begin{align}
	\mathcal{C}_2 &:= \Big\{  f \in C^0(\Omega)' : 
	\langle f, \cos(2\pi \mathbf{k} \cdot \mathbf{x} )\; \rangle \geq 0, \; 
	\langle f, \; \sin(2\pi \mathbf{k} \cdot \mathbf{x}) \; \rangle = 0, \; 
	\forall \; \mathbf{k} \in \mathbbm{Z}^d \setminus \mathbf{0} \; \Big\}.
\end{align}
Hence, properties (A3)--(A4) show that the set $\mathcal{A}$ is also a 
subset of $\mathcal{C}_2$, i.e., $\mathcal{A}\subseteq \mathcal{C}_2$.
Finally, taking the properties (A1)--(A4) together, we define the set
\begin{align*}
	\mathcal{C} := \Big\{  f \in C^0(\Omega)' &: 
	\textrm{ for all continuous } u(\mathbf{x}) \geq 0, \textrm{ and }
	\; \mathbf{k} \in \mathbbm{Z}^d \setminus \mathbf{0}, \\ 
	&\langle f, \cos(2\pi \mathbf{k} \cdot \mathbf{x} )\; \rangle \geq 0, 
	\quad \langle f, u(\mathbf{x}) \rangle \geq 0,\\
	&\langle f, \; \sin(2\pi \mathbf{k} \cdot \mathbf{x}) \; \rangle = 0, \; 
	\quad \langle f, 1 \rangle = 1, \quad \; \Big\}.
\end{align*}
Proposition~\ref{Prop_properties_of_A} may then be alternatively
stated as: $\mathcal{A}$ is a subset of $\mathcal{C}$, i.e., 
$\mathcal{A} \subseteq \mathcal{C}$.  
Our goal now is to extend 
the non-convex set $\mathcal{A}$, in the optimization of (P), 
to a relaxed set $\mathcal{C}$.
The purpose of introducing $\mathcal{C}_1$ and $\mathcal{C}_2$ is to
identify the set $\mathcal{C}$ as a convex subset of a convex cone.
To this end, we make the following remarks characterizing $\mathcal{C}$:
\begin{remark}\label{Rmk:characterizationC}
(The set $\mathcal{C}$ is a convex cone with an affine constraint) 
The set $\mathcal{C}$ is defined through linear constraints and 
inequalities, and hence
is a convex set. However, $\mathcal{C}$ may also alternatively
be written as
\[
	\mathcal{C} = \Big\{ f : 
	f \in \mathcal{C}_1 \cap \mathcal{C}_2, \textrm{ and } \langle f, 1 \rangle 
	= 1\Big\}.
\]
Here we have used the cones $\mathcal{C}_1$ and $\mathcal{C}_2$ to 
represent the properties (A1), (A3) and (A4) in the definition of 
$\mathcal{C}$. Since both $\mathcal{C}_1$ and $\mathcal{C}_2$ are 
convex cones, by the intersection properties of convex cones, it 
follows that $\mathcal{C}_1\cap \mathcal{C}_2$ is also a convex cone. 
As a result, the set $\mathcal{C}$ may be interpreted 
as the convex cone $\mathcal{C}_1\cap \mathcal{C}_2$, whose elements
satisfy the additional affine constraint $\langle f, 1 \rangle = 1$.  
\end{remark} 

\begin{remark} \label{Rmk:strictsubset}
	(The set $\mathcal{C}$ contains elements that are not in $\mathcal{A}$).  
	Consider
	$\Omega = [0,1]$ and $F(x) = 1 + \cos(2\pi x) \in \mathcal{C}$. Then, 
	only functions of the form $f(x) = 1 + \sqrt{2} \cos(2\pi x - \varphi)$, 
	for any $\varphi$, have auto-correlations equal to $F(x)$.   
	Since $f(x)$ contains negative values, then $f(x) \notin \mathcal{C}_1$ 
	showing that $F(x) \notin \mathcal{A}$.  Moreover, a similar calculation
	shows that $F(x)$ cannot be written as the convex combination of 
	two, or even a finite number of, elements in $\mathcal{A}$, i.e., 
	$F(x) \neq \lambda f_1 \circ f_1 + (1-\lambda) f_2 \circ f_2$, for
	probabilities $f_1(x)$ and $f_2(x)$ and $0 \leq \lambda \leq 1$.
	This suggests that $F(x)$ cannot be approximated by convex combinations
	of elements in $\mathcal{A}$. 
\end{remark}


We now define the relaxed problem by extending the set $\mathcal{A}$ 
to the convex set $\mathcal{C}$:

\begin{align} \nonumber
	(R) \quad \quad &\textrm{Minimize } \quad  \frac{1}{2} \langle F, W \rangle \\ \nonumber
	&\textrm{subject to} \quad  \langle F, \cos(2\pi \mathbf{k}\cdot\mathbf{x})  \rangle \geq 0, \quad \quad \langle F, u(\mathbf{x})  \rangle \geq 0, \\ \nonumber
	&\phantom{\textrm{subject to}}\quad	\langle F, \sin(2\pi \mathbf{k}\cdot\mathbf{x})  \rangle = 0, \quad \quad \langle F, 1 \rangle = 1,  \\ \nonumber
 	&\textrm{for all integers } \mathbf{k} \in \mathbbm{Z}^d \setminus \mathbf{0} \textrm{ and non-negative continuous functions } u(\mathbf{x}) \geq 0.  	 			  
\end{align}	
We denote any solution to (R) as $F_R(\mathbf{x})$, and set 
$\mathcal{E}_R = \frac{1}{2}\langle F_{R}, W\rangle$. Moreover, we note 
that $\mathcal{E}_R$ is a lower bound to $\mathcal{E}_0$, 
i.e., $\mathcal{E}_0 \geq \mathcal{E}_R$, since (R) can be understood 
as optimizing (P) over a feasible set $\mathcal{C}$ that contains
$\mathcal{A}$.

	As discussed in Remark~\ref{Rmk:characterizationC}, the constraints 
	in (R) are both (i) linear in 
	$F(\mathbf{x})$, and (ii) (up to the affine constraint 
	$\langle F, 1\rangle = 1$) restricted to lie in the convex cone 
	$\mathcal{C}_1 \cap \mathcal{C}_2$. This will lead to numerical 
	discretizations of (R) that take the form of a \emph{conic linear 
	programming} problem.  Note that in practice when solving (R), 
	it is often better to enforce the mirror symmetry of 
	$F(\mathbf{x})$ directly, and remove redundant 
	$\mathbf{k}$ values (i.e. $\mathbf{k}$ and $-\mathbf{k}$ yield the
	same constraint)
	for the cosine constraints. 
	This will also allow for the removal 
	of the sine constraints in (R), and reduce the 
	size of the domain $\Omega$, and hence the optimization 
	problem.
	A few remarks are now in order:




\begin{remark} \label{Rmk:verification}
	(Sufficient conditions for a global minimizer)
	The relaxation (R) may in some cases verify that a candidate minimizer 
	$\rho^*(\mathbf{x})$ solves (P). Suppose $\rho^*(\mathbf{x})$ is a probability 
	distribution with auto-correlation $F_R(\mathbf{x})$.  Then, since any probability
	distribution is by definition larger than the minimizer 
	$\mathcal{E}(\rho^*) = \mathcal{E}_R \geq \mathcal{E}_0 \geq \mathcal{E}_R$. 
	Therefore, one has $\mathcal{E}_R = \mathcal{E}_0$, which implies that
	$\rho_0(\mathbf{x}) = \rho^*(\mathbf{x})$ is a, possibly non-unique, global
	minimum.  
\end{remark}

\begin{remark} \label{Rmk:lattic_exact} 
	(Lattices are exact) If the solution $F_R(\mathbf{x})$ forms a 
	periodic lattice pattern\footnote{
	A lattice $X$ is the infinite array of discrete points
	defined by a set of primitive vectors $\mathbf{v}_j$:
	$X = \{ \sum_{j = 1}^d n_j \mathbf{v}_j: 
	n_j \in \mathbb{Z}, \textrm{for } 1 \leq j \leq d \} $.  
	Take $\chi = X \cap \Omega$ as the points in $X$ restricted
	to the computational domain. Hence $\chi$ may be defined for any $X$.
	We refer here to $\chi$ as a \emph{lattice pattern} if $X$ can 
	be written as translated copies of $\chi$:
	$X = \cup_{\vec{n} \in \mathbb{Z}^d} (\chi + \vec{n})$.  Note that
	$\chi$ may be a set that is larger than one containing the primitive 
	lattice vectors, and that $X$ cannot always be written as the collection
	of translated copies of $\chi$ (in which case $\chi$ would not be a lattice
	pattern). }
	$\chi \subset \Omega$
	\begin{align} \label{DiscreteLattice}
		F_R(\mathbf{x}) = \frac{1}{|\chi|}\sum_{\mathbf{s}\in \chi} 
		\delta(\mathbf{x} - \mathbf{s}), 
	\end{align}
	where $|\chi|$ is the number of points in the lattice pattern, 
	then the relaxation is exact. For solutions of the form (\ref{DiscreteLattice}), 
	$F_R \circ F_R = F_R(\mathbf{x})$.  Hence, taking 
	$\rho^*(\mathbf{x}) = F_R(\mathbf{x})$, 
	satisfies $F_R(\mathbf{x}) = \rho^* \circ \rho^*$ thereby 
	implying that a lattice is 
	the global minimizer.
	
	Minimizers that take the form of a lattice are of great physical interest, 
	as they explain why matter may form crystal structures.  Proofs that 
	particle models, in the large particle number limit, exhibit lattice 
	minimizer have been	done for sticky disk models \cite{HeitmannRadin1980, Radin1981},
	Lennard-Jones type interaction potentials \cite{Theil2006}, and energies 
	which include the sum of Lennard-Jones interaction potentials and three particle
	interactions \cite{ELi2009, FarmerEsedogluSmereka2016}.
\end{remark}


In Appendix B, we discuss how to numerically discretize
and solve (R).  In general, we observe that numerical solutions convergence to 
either (i) classical functions $F_R(\mathbf{x})$ that are continuous, 
i.e., $F_R(\mathbf{x}) \in \mathcal{C}^0(\Omega)$, or (ii) non-classical
functions $F_R(\mathbf{x})$ that consist of a finite collection of Dirac point
masses. Motivated by Remark~\ref{Rmk:verification}, the following section 
presents one approach for recovering a candidate minimizer $\rho^*(\mathbf{x})$ 
using $F_R(\mathbf{x})$. 

In the case when $F_R(\mathbf{x})$ is a collection of Dirac masses, we will 
expect recovered candidates $\rho^*(\mathbf{x})$ to also be a collection 
of Dirac masses.  This is because the auto-correlation of a discrete set of Dirac masses 
is a discrete set of Dirac masses. In contrast, when $F_R(\mathbf{x})$ is a 
continuous function, we will expect $\rho^*(\mathbf{x})$ to be in $L^2(\Omega)$, 
and typically take the form of a piece-wise continuous function, i.e., since the 
auto-correlation of a piece-wise continuous function is continuous.

\section{Recovering $\rho^*(x)$ from $F_R(x)$ by minimizing a relative entropy}\label{Sec_KLrecovery}
In this section we outline a procedure for recovering a candidate global 
minimizer $\rho^*(\mathbf{x})$ from knowledge of the solution to (R), i.e., $F_R(\mathbf{x})$.
In general, the relaxed space $\mathcal{C}$, and therefore solutions to 
problem (R) may include measures that are not auto-correlations of probabilities 
(see Remark~\ref{Rmk:strictsubset}).  Hence, $\mathcal{A} \subset \mathcal{C}$ 
is only a proper subset of $\mathcal{C}$ and as a result, the solution $F_R(\mathbf{x})$ 
may not come from an auto-correlation of a probability distribution.

The problem of recovering $\rho^*(\mathbf{x})$ from $F_R(\mathbf{x})$ is equivalent to 
deauto-correlating a function $F(\mathbf{x}) \in \mathcal{C}$, with the caveat that the 
source function $\rho^*(\mathbf{x})$ is also a probability distribution.  
The additional non-negativity restriction, i.e., $\rho^*(\mathbf{x}) \in \mathcal{C}_1$,
distinguishes the phase recovery problem at hand from other phase recovery 
problems recently studied in the context of signal processing 
\cite{CandesLi2012, CandesLiSoltanolkotabi2014, CandesTao2010, JaganathanOymakHassibi2012, Torquato2006}.  
In our recovery process we follow a procedure introduced by Schulz and Snyder 
\cite{SchulzSnyder1992} (see also \cite{SchulzVoelz2005}) which chooses 
$\rho^*(\mathbf{x})$ as a minimizer of the Kullback-Leibler divergence functional (also 
known as the information divergence) between $F_R(\mathbf{x})$ and the auto-correlation 
$F_{\rho}(\mathbf{x}) = \rho \circ \rho$. As discussed in 
\cite{SchulzSnyder1992} (and references within) the information divergence 
functional has many nice properties for the recovery of non-negative signals 
making it a natural choice for the recovery of $\rho^*(\mathbf{x})$.  


	In this discussion we assume that 
	$F_R(\mathbf{x}) \in \mathcal{C} \cap C^{0}(\Omega)$ is a continuous 
	function, however the approach here can also be extended
	to handle cases where $F_R(\mathbf{x})$ is a collection of discrete Delta masses.
	The Kullback-Leibler divergence is defined as
	\begin{align} \label{KullbackLeiblerCts}
		\mathcal{F}(\rho) &:= \int_{\Omega} F_R(\mathbf{x}) 
		\ln\Big( \frac{F_R(\mathbf{x})}{F_{\rho}(\mathbf{x}) } \Big) \; 
		\du \mathbf{x} = \int_{\Omega} F_R(\mathbf{x}) 
		\ln\Big( \frac{F_R(\mathbf{x})}{ \rho\circ\rho } \Big) \; \du \mathbf{x},
	\end{align}
	where we assume that $\rho(\mathbf{x}) \in \mathcal{C} \cap L^2(\Omega)$ 
	with $\int_{\Omega}\rho(\mathbf{x}) \du \mathbf{x} = 1$. 
	In the definition of $\mathcal{F}(\rho)$, one adopts the conventions:
	\begin{align} \label{KL_Conventions}
		0 \ln \frac{0}{a} := 0, \quad \quad	0 \ln \frac{0}{0} := 0, 
		\quad \quad a \ln \frac{a}{0} := \infty,   
	\end{align}
to allow for both $F_R(\mathbf{x})$ and $\rho\circ \rho$ to vanish on some set.  

	Viewing both $F_R(\mathbf{x})$ and $F_{\rho}(\mathbf{x})$ as probability 
	distributions, the Kullback-Leibler divergence, defined by $\mathcal{F}(\rho)$, 
	measures the mismatch between probabilities $F_{R}(\mathbf{x})$ and 
	$F_{\rho}(\mathbf{x})$.  Although $\mathcal{F}(\rho)$ does 
	not define a metric between $F_R(\mathbf{x})$ and $F_{\rho}(\mathbf{x})$, 
	for instance since it is not symmetric, it is always non-negative 
	$\mathcal{F}(\rho) \geq 0$, and may still be used to guarantee an exact 
	match between $F_{R}(\mathbf{x})$ and $F_{\rho}(\mathbf{x})$.  
	Specifically, $\mathcal{F}(\rho^*) = 0$,  only when 
	$F_{R}(\mathbf{x}) = F_{\rho}(\mathbf{x})$ (See, for instance, Pinsker's inequality 
	in Chapter 2 of \cite{Massart2003}). 
	Hence, in light of Remark~\ref{Rmk:verification}, we have the following alternative
	sufficient condition for a global minimizer--which motivates the minimization
	of $\mathcal{F}(\rho)$: 
	\begin{remark}\label{Rmk:eqv_suff_cond}
	(Equivalent sufficient condition for a global minimizer)
	Let $F_R(\mathbf{x})$ solve (R).  Then if $\mathcal{F}(\rho^*) = 0$, it 
	follows that 
	$\rho^*(\mathbf{x})$ solves (P).  For instance, if $\mathcal{F}(\rho^*) = 0$,
	then the two auto-correlations are equal:
	$F_R(\mathbf{x}) = F_{\rho^*}(\mathbf{x}) = \rho^*\circ\rho^*$,
	so that the conditions in Remark~\ref{Rmk:verification} are satisfied.
	\end{remark}
	
To compute minimizers of $\mathcal{F}(\rho)$ we use the Schulz-Snyder iterative 
algorithm\footnote{In the original paper \cite{SchulzSnyder1992}, the 
functional $\mathcal{F}(\rho)+\int_{\Omega} F_R(\mathbf{x}) - F_{\rho}(\mathbf{x}) \du \mathbf{x} $ 
was used instead of $\mathcal{F}(\rho)$. Due to the fact that the Schulz-Snyder 
iterative algorithm conserves the constraint 
$\int_{\Omega} F_{\rho}(\mathbf{x}) \du \mathbf{x} = 1$, the discrepancy in 
functional definition has no effect on the algorithm or results.}, which is 
an iterative method on the space of non-negative 
probabilities. 
	The advantage of the Schulz-Snyder algorithm is not only that it 
	minimizes the functional $\mathcal{F}(\rho)$, but the 
	iterations naturally enforce the probability constraints.    
	As a result, the algorithm is very easy to implement.

	We now briefly summarize the derivation, and properties of the Schulz-Snyder 
	algorithm. The idea is to iterate the Euler-Lagrange equation that one
	obtains by taking the first variation of $\mathcal{F}(\rho)$. 
	 Namely, the Euler-Lagrange equation of 
	(\ref{KullbackLeiblerCts}) is given as follows: For any mean zero 
	perturbation $g(\mathbf{x})$ whose support is contained in the support 
	of $\rho^*(\mathbf{x})$, the first variation vanishes
	\begin{align}	\label{Perturbation}
	 \int_{\Omega} g(\mathbf{x}) 
	 \frac{\delta \mathcal{F}}{\delta \rho}(\rho^*) = 0, 
	 \quad \Longrightarrow \quad  
	 \frac{\delta \mathcal{F}} {\delta \rho}(\rho^*) 
	 = \textrm{const. } \textrm{ for any } \mathbf{x} 
		\in \textrm{supp}(\rho^*). 
	 \end{align}
	By direct calculation, the (unconstrained) $L^2$ variation of 
	$\mathcal{F}(\rho)$ is:
	\begin{align} \label{L2_Variation}
		\frac{\delta \mathcal{F}}{\delta \rho}(\rho) 
		= -2 \int_{\Omega}\rho(\mathbf{x} + \mathbf{y})
		\frac{F_R(\mathbf{y})}{F_{\rho}(\mathbf{y})} \du \mathbf{y}. 
	\end{align}
	Hence, multiplying (\ref{L2_Variation}) by $\rho^*(\mathbf{x})$ and 
	integrating over space yields the constant in (\ref{Perturbation}). 
	Critical points of $\mathcal{F}(\rho)$ satisfying the first variation
	conditions are then concisely described by
\begin{align} \label{ConciseStableCP}
	\frac{\delta \mathcal{F}}{\delta \rho}(\rho^*) \left\{
\begin{array}{ll}
= -2 & \text{if } \rho^*(\mathbf{x}) > 0,\\
> -2 & \text{if } \rho^*(\mathbf{x}) = 0,
\end{array}\right. 
\quad \Longrightarrow \quad \rho^*(\mathbf{x}) 
= \rho^*(\mathbf{x}) \int_{\Omega} \rho^*(\mathbf{x}+\mathbf{y}) 
\frac{F_R(\mathbf{y})}{F_{\rho^*}(\mathbf{y})} \du \mathbf{y}. 
\end{align}	
	Equation (\ref{ConciseStableCP}) may now be used to devise an iterative 
	fixed-point algorithm:
	
	\begin{algorithm}[htb!]
		\begin{flushleft}
		{\bf Recovering $\rho^*(\mathbf{x})$ from $F_R(\mathbf{x})$ (Schulz-Snyder) }
		\begin{enumerate}
			\item Initialize $\rho_{0}(\mathbf{x}) > 0 $ to be strictly positive 
			with $\int_{\Omega}\rho_0(\mathbf{x})\du\mathbf{x} = 1$.  Ensure that 
			$\rho_{0}(\mathbf{x})$ has no planes of symmetry: for any fixed vector 
			$\mathbf{a}$, the shifted $\rho_0(\mathbf{x})$ is not even symmetric 
			$\rho_{0}(\mathbf{a}-\mathbf{x} ) \neq \rho_{0}(\mathbf{x}-\mathbf{a})$.
			\item Iterate the discrete mapping:
		\begin{align} \nonumber
			&\rho_{n+1}(\mathbf{x}) = - \frac{1}{2}\rho_n(\mathbf{x}) 
			\frac{\delta \mathcal{F}}{\delta \rho}(\rho_n), \quad \quad \quad \quad 
			\textrm{ for } \mathbf{x} \in \Omega, \textrm{and } n = 1, 2,3 \ldots \\ \nonumber
			&\phantom{\rho_{n+1}(\mathbf{x})} = \rho_{n}(\mathbf{x}) \int_{\Omega} 
			\rho_{n}(\mathbf{x}+\mathbf{y}) \frac{F_R(\mathbf{y})}{F_{\rho_n}(\mathbf{y})}  
			\du\mathbf{y}, \\ \nonumber
			&\textrm{where}\quad	F_{\rho_n}(\mathbf{x}) = \rho_n \circ \rho_n 
			\textrm{ and we have used the fact that } F_R(\mathbf{y}) = F_R(-\mathbf{y}).
		\end{align}
			\item Take $\rho^*(\mathbf{x}) = \rho_{\infty}(\mathbf{x})$ as the 
			candidate global minimizer to (P).
		\end{enumerate}
		\end{flushleft}
	\end{algorithm}

	The algorithm also 
	ensures the following properties, which we state without proof\footnote{Note: properties
	1, 2 and 4 are straight-forward to prove. See \cite{SchulzSnyder1992} for a proof of 
	a discrete version of the monotonicity property 3.}
	\begin{enumerate}
		\item (Positivity preserving) $\rho_{n}(\mathbf{x}) \geq 0$ for all 
		$\mathbf{x}$ and $n \geq 0$.
		\item (Mass preserving) 
		$\int_{\Omega} \rho_n(\mathbf{x})\du \mathbf{x} = 1$ for all $n \geq 0$. 		
		\item (Monotonicity) $\mathcal{F}(\rho_{n+1}) \leq \mathcal{F}(\rho_n)$ 
		for all $n \geq 0$.
		\item (Fixed points) If $\rho^*(\mathbf{x})$ is a fixed point in the Schulz-Snyder 
		algorithm, then $\rho^*(\mathbf{x})$ satisfies the first variation
		conditions (\ref{ConciseStableCP}).
	\end{enumerate}
	Finally, as prescribed in Step 1 of the Schulz-Snyder algorithm, it is 
	important to avoid initializing the data $\rho_0(\mathbf{x})$ to lie in any invariant 
	set of the iterative map from Step 2. Initializing the data $\rho_0(\mathbf{x})$ 
	to lie in an invariant set can potentially constrain the resulting 
	fixed point minimizer $\rho^*(\mathbf{x})$ to have the same symmetry as 
	$\rho_0(\mathbf{x})$. The Schulz-Snyder algorithm has invariant sets that 
	include the following subspaces:
	\begin{itemize}
		\item If $\rho_n(\mathbf{x}_p) = 0$ for some point 
		$\mathbf{x}_p \in \Omega$ then $\rho_{n+1}(\mathbf{x}_p) = 0$.
		\item If for a fixed vector $\mathbf{a}$, 
		$\rho_{n}(\mathbf{a}-\mathbf{x} ) = \rho_{n}(\mathbf{x}-\mathbf{a})$ 
		then $\rho_{n+1}(\mathbf{a}-\mathbf{x} ) = \rho_{n+1}(\mathbf{x}-\mathbf{a})$.
	\end{itemize} 
	The first symmetry regarding $\rho_{n+1}(\mathbf{x}_p) = 0$ follows 
	directly from testing both sides of Step 2 in the iterative scheme at a 
	point $\mathbf{x}_p \in \Omega$. The 
	second property, regarding planes of symmetry, can be shown as well since 
	both $F_R(\mathbf{x})$ and $F_{\rho}(\mathbf{x})$ are mirror symmetric about $\mathbf{0}$:
	\begin{align} \nonumber
		\rho_{n+1}(\mathbf{a} - \mathbf{x}) &= \rho_n(\mathbf{a} - \mathbf{x}) 
		\int_{\Omega} \rho_n( \mathbf{a} - \mathbf{x} + \mathbf{y}) 
		\frac{F_R(\mathbf{y})}{F_{\rho_n}(\mathbf{y})} \du \mathbf{y}, \\ \nonumber
		&= \rho_n(\mathbf{a} - \mathbf{x}) \int_{\Omega} 
		\rho_n( \mathbf{a} - \mathbf{x} - \mathbf{y}) 
		\frac{F_R(\mathbf{y})}{F_{\rho_n}(\mathbf{y})} \du \mathbf{y}, \\ \nonumber
		&= \rho_n(\mathbf{x}-\mathbf{a} ) \int_{\Omega} 
		\rho_n( \mathbf{x} - \mathbf{a} + \mathbf{y}) 
		\frac{F_R(\mathbf{y})}{F_{\rho_n}(\mathbf{y})} \du \mathbf{y} 
		= \rho_{n+1}(\mathbf{x} + \mathbf{a}).
	\end{align}


We now briefly discuss several numerical details of the Schulz-Snyder
algorithm.  One advantage with minimizing the Kullback-Liebler divergence over 
other norms or metrics is that the Schulz-Snyder algorithm may be numerically computed 
using integral quadrature rules, without enforcing non-negativity and mass 
constraints.  Moreover,
up to a negative sign in $\mathbf{x}$, the integral in Step 2 of the algorithm has the
form of a convolution--which may also be computed in an efficient manner using
the fast Fourier transform.  Finally, regarding the convergence rate of the scheme, 
one might heuristically expect it to behave in a fashion similar to other
iterative methods with an exponential convergence at large $n$, i.e., 
$|\mathcal{F}(\rho_n) - \mathcal{F}(\rho_{\infty})| \sim \gamma^n$, for a
value of $0 < \gamma < 1$.  Together these properties make using the Kullback-Liebler
divergence an attractive approach for practitioners.

In Section~\ref{Sec:preliminaries}, necessary conditions for a candidate 
minimizer to solve (P) were given by equations (\ref{KKT_for_P}), (\ref{ExtendedFirstOrder_condition}) and (\ref{SecondOrder_condition}).
Although numerical examples in Sections~\ref{sec:NumericalResults_d1} and 
\ref{sec:NumericalResults_d2} provide supporting evidence that solutions to 
equation (\ref{ConciseStableCP}) may (at least in some cases) satisfy
 (\ref{KKT_for_P}), (\ref{ExtendedFirstOrder_condition}) and (\ref{SecondOrder_condition}), we have no formal 
proof of such a result.  The minimization via the Schultz-Snyder 
algorithm does however often recover candidates $\rho^*(\mathbf{x})$ 
with $F_{\rho^*}(\mathbf{x})$ having the same support as 
$F_R(\mathbf{x})$--which, as we will show through 
the introduction of the dual formuation to (R), guarantees the
necessary condition related to (\ref{SecondOrder_condition}) in 
Remark \ref{E_Convex_on_Z}.

\section{The Dual decomposition} \label{sec_dualdecomposition}

The purpose of this section is to formulate the dual 
optimization problem to the convex 
relation (R), and show how it may be used, in some cases, to 
explain why the supports of the recovered minimizers $\rho^*(\mathbf{x})$ 
satisfy the necessary conditions in Remark \ref{E_Convex_on_Z}. 
This will be done in two steps.  First, the dual formulation will provide a 
decomposition of the pairwise energy $\mathcal{E}(\rho)$ 
that takes the form of a non-convex/convex splitting:
\begin{align}\label{dual_decomp_functional}
	\mathcal{E}(\rho) = \mathcal{E}^+(\rho) + \mathcal{K}(\rho),
\end{align}
where 
\begin{enumerate}
	\item $\mathcal{E}^+(\rho) \geq 0$, is a non-negative functional for 
	all non-negative measures $\rho(\mathbf{x}) \in \mathcal{C}_1$, and, in 
	general, is non-convex. 
	\item $\mathcal{K}(f)$ is convex for all finite measures $f(\mathbf{x})$. 
	Namely, for all $0 \leq \lambda \leq 1$ and 
	$f_1(\mathbf{x}), f_2(\mathbf{x})$ (which may be negative), 
	one has:
		\[
		\mathcal{K}(\lambda f_1 + (1-\lambda) f_2) \leq \lambda \mathcal{K}(f_1) 
		+ (1-\lambda) \mathcal{K}(f_2).
		\]
\end{enumerate}
Second, the non-negative part of the decomposition 
(\ref{dual_decomp_functional}), $\mathcal{E}^+(\rho)$, will be used 
to provide a sufficient condition to satisfy the necessary conditions in Remark \ref{E_Convex_on_Z}. 

Decompositions of the form given by (\ref{dual_decomp_functional}) are in
general not unique. However, the dual formulation to (R) will provide such a 
decomposition that also maximizes the minimum value of the convex functional 
$\mathcal{K}(\rho)$ over probabilities $\rho(\mathbf{x})$.  In other words, 
we will seek $\mathcal{K}(\rho)$ to be, in some sense, the largest convex 
functional that underestimates $\mathcal{E}(\rho)$. As a result, the optimal 
functional $\mathcal{K}(\rho)$ that we compute has a strong resemblance to 
the convex envelope of $\mathcal{E}(\rho)$.

We will show below that an optimal decomposition of the form
(\ref{dual_decomp_functional}) may be formulated as the dual problem
to (R) -- and therefore computed with the same computational 
cost as solving (R).  Here the construction of the optimal decomposition of 
the form (\ref{dual_decomp_functional}) will arise by decomposing the 
interaction energy $W(\mathbf{x})$ into the sum of a non-negative function, 
and a function with non-negative cosine modes. 

To motivate the dual formulation to (R), first consider any decomposition for 
$W(\mathbf{x})$ that takes the form
\begin{align} \label{ConicDecomposition}
	W(\mathbf{x}) = W^+(\mathbf{x}) + K(\mathbf{x}) + 2\mathcal{E}_D,
\end{align}
where
%
\begin{enumerate}
	\item[(D1)] $0 \leq W^+(\mathbf{x}) \in \mathcal{C}^0(\Omega)$ is a 
	continuous, non-negative, mirror symmetric function 
	(See Remark~\ref{SymmetricW}).
	\item[(D2)] $K(\mathbf{x})$ is a continuous, mirror symmetric, 
	mean-zero function
	with real non-negative cosine coefficients, i.e.:
	\begin{align*}
		\hat{K}(\mathbf{k}) &:= \int_{\Omega} K(\mathbf{x}) 
		\cos(2\pi \mathbf{k} \cdot \mathbf{x}) \du \mathbf{x} \geq 0, 
		\textrm{ for all } \mathbf{k} \in \mathbbm{Z}^{d} \setminus \mathbf{0}, 
		\textrm{ and } \hat{K}(\mathbf{0}) = 0, \\
		K(\mathbf{x}) &= \sum_{\mathbf{k} \in \mathbbm{Z}^d}  
	\hat{K}(\mathbf{k})\cos(2\pi \mathbf{k}\cdot \mathbf{x}).
	\end{align*}	
	Note that the summation in the above cosine series includes all $\mathbf{k} \in \mathbb{Z}^d$,
	and	the inclusion of $\hat{K}(-\mathbf{k}) = \hat{K}(\mathbf{k})$ accounts for the 
	apparent missing factor of $2$.

	We also make the following technical assumption on the cosine coefficients of
	$K(\mathbf{x})$:
	\begin{align}\label{Assumption_Abs_cont}
		\sum_{\mathbf{k}\in \mathbbm{Z}^d} \hat{K}(\mathbf{k}) < \infty.
	\end{align}
	Assumption (\ref{Assumption_Abs_cont}) guarantees that the cosine 
	series for $K(\mathbf{x})$ converges uniformly, for instance by a Weierstrass M-test.
	Moreover (\ref{Assumption_Abs_cont}) will be sufficient to 
	use a Plancherel-type theorem when integrating $K(\mathbf{x})$ against probability 
	measures. 
	\item[(D3)] $\mathcal{E}_D$ is a constant.  Due to the normalization convention
	(W4), of $W(\mathbf{x})$, we see that 
	$\mathcal{E}_D = -\frac{1}{2}\int_{\Omega} W^+(\mathbf{x}) \du \mathbf{x} $ 
	will be negative for decompositions of the form (\ref{ConicDecomposition}).
\end{enumerate}

\begin{proposition}\label{Prop_ConicDecomp}
(Properties of the decomposition (\ref{ConicDecomposition}))
Any decomposition of the form (\ref{ConicDecomposition}) with properties
(D1)--(D3) satisfies the following:
	\begin{enumerate}
		\item The functions $W^+(\mathbf{x})$ and $K(\mathbf{x})$ are in 
		the dual cones to $\mathcal{C}_1$ and $\mathcal{C}_2$: i.e., 
		$W^+(\mathbf{x}) \in \mathcal{C}_1^*$, and $K(\mathbf{x}) \in \mathcal{C}_2^*$
		where the dual cone $X^*$ to a convex cone $X$ is given by:
			\[ 
				X^* := \{ x \in X': \langle x, y \rangle \geq 0, \forall y \in X\}.
			\]
		\item $\mathcal{E}_D \leq \mathcal{E}_R$ is a lower bound to (R).
		\item The following functional is non-negative:
			\[
				\mathcal{E}^+(\rho) := \frac{1}{2}\int_{\Omega}\int_{\Omega} 
				\rho(\mathbf{x}) \rho(\mathbf{y}) W^+(\mathbf{x}-\mathbf{y}) 
				\du \mathbf{x} \du \mathbf{y} \geq 0, \quad \textrm{for all } 
				\rho(\mathbf{x}) \in \mathcal{C}_1.
			\]
		\item The following functional is convex for $\rho(\mathbf{x}) \in \mathcal{C}_1$:
			\[
			\mathcal{K}(\rho) := \frac{1}{2}\int_{\Omega}\int_{\Omega} 
			\rho(\mathbf{x}) \rho(\mathbf{y}) K(\mathbf{x}-\mathbf{y}) 
			\du \mathbf{x} \du \mathbf{y} + \mathcal{E}_D,
			\]
	\end{enumerate}
\end{proposition}
\begin{proof}
	The proof again involves computing the appropriate integrals. \\
	For 1 we have: 
	\begin{itemize}
		\item Given $F(\mathbf{x}) \in \mathcal{C}_1$, then
		$\langle F, W^+\rangle = \int_{\Omega} W^+(\mathbf{x}) F(\mathbf{x}) \du \mathbf{x} \geq 0$, 
	since $W^+(\mathbf{x})\geq 0$, and $F(\mathbf{x}) \in \mathcal{C}_1$  is non-negative. 
		\item Given $F(\mathbf{x}) \in \mathcal{C}_2$, then 
		$\langle F, K\rangle = \sum_{\mathbf{k} \in \mathbbm{Z}^{d}} 
		\hat{K}(\mathbf{k}) \hat{F}(\mathbf{k}) \geq 0,$ since\footnote{Here we 
		provide some details justifying the series expansion for 
		$\langle F, K\rangle$. Note that 
		if $K(\mathbf{x})$ satisfies assumption (\ref{Assumption_Abs_cont}), then the
		cosine series converges uniformly.  Hence, for any $\epsilon > 0$, there 
		exists an $M > 0$, such that 
		$\max_{\mathbf{x} \in [0, 1]^d}\| D_M(\mathbf{x}) \| < \epsilon$, where 
		$D_M(\mathbf{x}) := K(\mathbf{x}) - \sum_{|\mathbf{k}| < M} 
		\hat{K}(\mathbf{k})\cos(2\pi \mathbf{k}\cdot \mathbf{x})$. Since $D_M(\mathbf{x})$ 
		is continuous with a maximum norm of $\epsilon$, this implies 
		that $\langle F, D_M\rangle \rightarrow 0$ as $M\rightarrow \infty$. 
		Hence, $\langle F, K\rangle \rightarrow  \sum_{\mathbf{k} \in \mathbbm{Z}^{d}} 
		\hat{K}(\mathbf{k}) \hat{F}(\mathbf{k})$ as $M \rightarrow \infty$. }		
		$\hat{K}(\mathbf{k}) \geq 0$, and 
		$\hat{F}(\mathbf{k}) := \langle F, \cos(2\pi \mathbf{k}\cdot\mathbf{x})\rangle \geq 0$, 
	    for all $\mathbf{k} \in \mathbbm{Z}^d$. 
	\end{itemize}
	Hence $W^+(\mathbf{x})$ and $K(\mathbf{x})$ are in the dual cones to 
	$\mathcal{C}_1$ and $\mathcal{C}_2$ respectively.\\
	For 2: A direct calculations shows:
	\begin{align} \label{DualLowerBound}
		\mathcal{E}_R &= \frac{1}{2}\langle F_R, W \rangle \\
		&=  \frac{1}{2}\Big(\langle F_R, W^+ \rangle 
		+ \langle F_R, K \rangle \Big) + \mathcal{E}_D \geq \mathcal{E}_D \nonumber,
	\end{align}
	The last inequality follows since each pairing independently is non-negative.  Namely 
	$F_R(\mathbf{x})$ is in both $\mathcal{C}_1$ and $\mathcal{C}_2$ and so the result
	in part 1 applies. Hence, both $\langle F_R, W^+\rangle \geq 0$ and 
	$\langle F_R, K\rangle \geq 0$.\\
	For 3: The proof is identical to the proof of property (A1) in Proposition~\ref{Prop_properties_of_A}.\\		
	For 4: Since $K(\mathbf{x})$ satisfies (\ref{Assumption_Abs_cont}), uniform convergence
	of the cosine series allows one to write $\mathcal{K}(\rho)$ using a Plancherel-type identity. 
	Specifically, for any $\rho(\mathbf{x}) \in \mathcal{C}_1$:
	\begin{align*}
		\mathcal{K}(\rho) &= 
		\frac{1}{2}\sum_{\mathbf{k} \in \mathbbm{Z}^d } 
		\hat{K}(\mathbf{k})	\;	\Big(
		\langle \rho, \cos(2\pi \mathbf{k}\cdot \mathbf{x}) \rangle^2 +
		\langle \rho, \sin(2\pi \mathbf{k}\cdot \mathbf{x}) \rangle^2 \Big) + \mathcal{E}_D.
	\end{align*}
	Since $\hat{K}(\mathbf{k})	\geq 0$, the functional $\mathcal{K}(\rho)$ is a positive 
	definite quadratic--and hence convex. Note that in general, numerical observations 
	later show it is often the case that $\hat{K}(\mathbf{k}) = 0$ for some subset of 
	integers $\mathbf{k}$, indicating that $\mathcal{K}(\rho)$ is not strictly convex. 
\end{proof}

	 The dual problem (D) to (R) is then formulated as optimizing 
	 (\ref{ConicDecomposition}) to find the \emph{best possible} constant 
	 $\mathcal{E}_D$ and corresponding decomposition for $W(\mathbf{x})$ into 
	 the sum of a non-negative function and a function with non-negative 
	 cosine modes:
\begin{align}
		(D)\quad \quad &\textrm{Maximize  } \mathcal{E}_D, \\ \nonumber
		&\textrm{subject to } \big( W(\mathbf{x}) - 2\mathcal{E}_D \big) 
		\in \textrm{cl}\;\big( \mathcal{C}_1^* + \mathcal{C}_2^*\big). \nonumber
\end{align}
Here $\textrm{cl}$ is the weak$^*$ closure, $\mathcal{C}_{1,2}^*$ are the 
dual cones\footnote{The formulation (D) is over the dual cone $\mathcal{C}^*$, 
which (see Lemma 3.1 in \cite{BurachikJeyakumar2004} for two intersecting 
closed convex cones) is equal to 
$\mathcal{C}^* = (\mathcal{C}_1\cap \mathcal{C}_2)^*
= \textrm{cl}\;\big( \mathcal{C}_1^* + \mathcal{C}_2^*\big)$. } 
to $\mathcal{C}_{1,2}$, where the sum 
$\mathcal{C}_1^* + \mathcal{C}_2^* 
= \{ u + v : u \in \mathcal{C}_1^*, v \in \mathcal{C}_2^*\}$.	
\begin{assumption}\label{Reg_assumption}(Regularity assumption)
	We assume there exists functions
	$W^+_R(\mathbf{x}) \in \mathcal{C}_1^*$ and 
	$K_R(\mathbf{x}) \in \mathcal{C}_2^*$ that solve (D), and also satisfy the 
	smoothness properties in (D1)--(D2). In other words, $W(\mathbf{x})$ may be 
	written as an optimal decomposition into the dual cones of 
	$\mathcal{C}_1$ and $\mathcal{C}_2$:
	\begin{align}\label{Optimal_Decomposition}
		W(\mathbf{x}) = W_R^+(\mathbf{x})+ K_R(\mathbf{x}) + 2\mathcal{E}_R,
	\end{align}
	where the optimum value of $\mathcal{E}_D$ in (D) is the same as 
	$\mathcal{E}_R$.
\end{assumption}
We refer to the optimal decomposition (\ref{Optimal_Decomposition}) 
of the interaction energy as the \emph{dual decomposition}, as it arises
from the dual formulation of (D) to (R).  At the level of numerical 
discretizations presented in Appendix B, the 
Assumption \ref{Reg_assumption} is justified by the following remark.
	
\begin{remark}(Numerical justification of Assumption \ref{Reg_assumption})
	Numerical discretizations of (R)
	presented in Appendix B result in a linear 
	program--which therefore has a duality gap of zero.  Hence, every numerical 
	discretization of (R) has the optimal value $\mathcal{E}_R$ equal 
	to the optimal value $\mathcal{E}_D$ in (D).
	Moreover, the finite 
	dimensional cones $\mathcal{C}_{1,h}$ and $\mathcal{C}_{2,h}$ 
	that arise as the discrete approximations to $\mathcal{C}_1$ and $\mathcal{C}_2$
	are closed, self-dual and polyhedral. The sum of two polyhedral cones is also 
	polyhedral and hence closed (Theorem 19.1 and Corollary 19.3.2 in 
	\cite{Rockafellar1970}). Therefore, for any finite discretization, 
	one has: 
	$(\mathcal{C}_{1,h}\cap\mathcal{C}_{2,h})^* 
	= \mathrm{cl}(\mathcal{C}_{1,h}^* + \mathcal{C}_{2,h}^*) 
	= \mathcal{C}_{1,h} + \mathcal{C}_{2,h}$, showing that the dual cone 
	$\mathcal{C}_h^*$ can be written as the sum of the cones 
	$\mathcal{C}_{1,h}$ and $\mathcal{C}_{2,h}$.  This justifies, for any finite
	dimensional discretization, the existence of an optimal dual decomposition 
	of the form (\ref{Optimal_Decomposition}).  
	Note that in general, the sum of two closed, but non-polyhedral cones, may 
	not be closed. For example, for two closed convex cones 
	$\mathcal{C}_1, \mathcal{C}_2 \subset \mathbbm{R}^n$, one may have the 
	pathological situation where a point 
	$\mathbf{x} \in \mathrm{cl}(\mathcal{C}_1 + \mathcal{C}_2)$, however 
	there are no values 
	$\mathbf{y}\in \mathcal{C}_1$, $\mathbf{z} \in \mathcal{C}_2$ such that 
	$\mathbf{x} = \mathbf{y} + \mathbf{z}$.
\end{remark}

With the Assumption \ref{Reg_assumption} on the existence of a dual decomposition, 
the dual problem (D) may be written as a conic optimization
problem with linear constraints:
\begin{align*}
		(D)\quad \quad &\textrm{Maximize  } \mathcal{E}_D, \\ \nonumber
		&\textrm{subject to } \big( W(\mathbf{x}) - 2\mathcal{E}_D - K(\mathbf{x}) \big)  \geq 0, \\
		&\phantom{\textrm{subject to }} \langle K, \cos(2\pi \mathbf{k}\cdot \mathbf{x}) \rangle \geq 0, 
		\quad \langle K, 1 \rangle = 0,\\
		&\phantom{\textrm{subject to }} \langle K, \sin(2\pi \mathbf{k}\cdot \mathbf{x}) \rangle = 0, \\
		&\phantom{\textrm{subject to }} \textrm{for all } \mathbf{x} \in \Omega 
		\textrm{ and } \mathbf{k} \in \mathbbm{Z}^d \setminus \mathbf{0}.
\end{align*}

\begin{remark}(Regularity observation)
	The regularity of the optimal decomposition to (D) is an interesting 
	problem: Numerical solutions in dimension one 
	(see Section~\ref{sec:NumericalResults_d1}) suggest that if $W(x)$ is 
	smooth at $x$, then $W_R^+(x)$ and $K_R(x)$ are not necessarily smooth 
	at $x$ (although continuity of $W_R^+(x)$ and $K_R(x)$ has been 
	observed).
\end{remark}
	
	
\begin{remark}(Examples of decompositions for $W(\mathbf{x})$)
	Two examples of feasible dual decompositions, i.e., of the form in 
	equation (\ref{ConicDecomposition}), are:
	\begin{enumerate}[leftmargin = 2cm]
		\item [Example 1:] Take 
		$\mathcal{E}_D = \frac{1}{2}\min_{\mathbf{x} \in \Omega} 
		W(\mathbf{x})$, $K(\mathbf{x}) = 0$ 
		and $W^+(\mathbf{x}) := W(\mathbf{x}) - 2\mathcal{E}_D \geq 0$. 
		\item [Example 2:] Write $W(\mathbf{x}) = 
		K_+(\mathbf{x}) + K_{-}(\mathbf{x}) +2\mathcal{E}_D$ 
		as the sum of two functions where $K_{\pm}(\mathbf{x})$ have only 
		$\pm$ cosine coefficients. 
		Take $K(\mathbf{x}) = K_+(\mathbf{x})$ to be the projection of 
		$W(\mathbf{x})$ onto cosine modes with positive coefficients, let 
		$\mathcal{E}_D := \frac{1}{2}\min_{\mathbf{x}\in\Omega} 
		(W(\mathbf{x}) - K(\mathbf{x}))$ 
		and take 
		$W^+(\mathbf{x}) = W(\mathbf{x}) - K(\mathbf{x}) - 2\mathcal{E}_D \geq 0$.
	\end{enumerate}
\end{remark}

\subsection{Properties of the optimal dual decomposition} 
The purpose of this subsection is to show that the support of $W_R^+(\mathbf{x})$
can be used to identify sets $S_*$ in which the functional $\mathcal{E}(\rho)$ is
convex whenever $\textrm{supp}(\rho) \subseteq S_*$.  Specifically, the conclusion of the
subsection will provide a sufficient condition for a candidate minimizer 
$\rho^*(\mathbf{x})$ to satisfy
the necessary condition given in Remark~\ref{E_Convex_on_Z}. 

We first discuss the support of $W_R^+(\mathbf{x})$ in relation to 
$F_R(\mathbf{x})$.  Revisiting the lower bound (R) and writing 
$W(\mathbf{x})$ using the optimal dual decomposition yields:
\begin{align} \label{DualEnergy}
	\mathcal{E}_R &= \frac{1}{2} \langle W, F_R\rangle 
	= \frac{1}{2} \langle W_R^+, F_R \rangle 
	+ \frac{1}{2}\langle K_R, F_R\rangle + \mathcal{E}_R.
\end{align}
Since both $W_R^+(\mathbf{x}), K_R(\mathbf{x})$ are in the appropriate 
dual cones, the pairings $\langle W_R^+, F_R\rangle \geq 0$ and 
$\langle K_R, F_R \rangle \geq 0$.  
Therefore, (\ref{DualEnergy}) holds only if the integrals vanish 
\begin{align} \label{DualConstraint}
	 \langle W_R^+, F_R\rangle = 0, \quad \quad \langle K_R, F_R\rangle = 0.
\end{align}
Here the constraint (\ref{DualConstraint}) can be used to infer that 
$F_R(\mathbf{x})$ must have a complementary support to 
$W_R^+(\mathbf{x})$ in real space, and $K_R(\mathbf{x})$ in $\mathbf{k}$ space. 
Specifically:
\begin{enumerate}[leftmargin = 1.4cm]
	\item [Case 1:] When $F_R(\mathbf{x}) \in C^0(\Omega)$ is continuous,
	the dual decomposition satisfies
	\begin{align} \label{ContinuousWplusSupport}
		W_R^+(\mathbf{x}) F_R(\mathbf{x}) &= 0, \hspace{10mm} 
		\textrm{for all } \mathbf{x} \in \Omega, \\ \label{ContinuousKSupport}
		\hat{K}_R(\mathbf{k}) \hat{F}_R(\mathbf{k}) 
		&= 0, \hspace{10mm} \textrm{for all } \mathbf{k} \in \mathbbm{Z}^d.
	\end{align}
	Here $\hat{F}_R(\mathbf{k})$, $\hat{K}_R(\mathbf{k})$ are the cosine
	coefficients defined in the proof of Proposition \ref{Prop_ConicDecomp}.
	\item [Case 2:] When $F_R(\mathbf{x}) 
	= \sum_{\mathbf{r}\in R} f_R(\mathbf{r}) \delta(\mathbf{x} -\mathbf{r})$, 
	is a collection of Dirac masses at the locations 
	$R = \{\mathbf{x}_1, \mathbf{x}_2, \ldots \mathbf{x}_m\}$, with amplitudes
	$f_R(\mathbf{r})$:
	\begin{align}\label{DiscreteWplusSupport}
		W_R^+(\mathbf{r}) &= 0, \hspace{10mm} 
		\textrm{for all } \mathbf{r} \in R, \\ \label{DiscreteKSupport}
		\hat{K}_R(\mathbf{k}) \hat{F}_R(\mathbf{k}) 
		&= 0,	\hspace{10mm} \textrm{for all } \mathbf{k} \in \mathbbm{Z}^d.
	\end{align}
	Again $\hat{K}_R(\mathbf{k})$ and $\hat{F}_R(\mathbf{k})$  are the cosine 
	coefficients of $K_R(\mathbf{x})$ and $F_R(\mathbf{x})$, where $F_R(\mathbf{k})$
	can be expressed in terms of $f_R(\mathbf{r})$:
	\[
	\hat{F}_R(\mathbf{k}) = \langle F_R, \cos(2\pi\mathbf{k}\cdot\mathbf{x})\rangle 
	= \sum_{\mathbf{r}\in R} f_R(\mathbf{r}) \cos(2\pi\mathbf{k}\cdot\mathbf{r}).
	\]	
\end{enumerate}
Equation (\ref{ContinuousWplusSupport}) (or the discrete version of the equation (\ref{DiscreteWplusSupport}) )  shows that $W_R^+(\mathbf{x}) = 0$ 
whenever $F_R(\mathbf{x}) \neq 0$, and vise versa.  We now combine this 
observation with the results from Proposition \ref{Prop_ConicDecomp}.  First set:
	\begin{align*}
		\mathcal{E}_R^+(\rho) &:= \frac{1}{2}\int_{\Omega}\int_{\Omega} 
		\rho(\mathbf{x}) \rho(\mathbf{y}) W_R^+(\mathbf{x}-\mathbf{y}) 
		\du \mathbf{x} \du \mathbf{y} = \frac{1}{2}\langle \rho\circ\rho, W_R^+\rangle, \\
		\mathcal{K}_R(\rho) &:= \frac{1}{2}\int_{\Omega}\int_{\Omega} 
		\rho(\mathbf{x}) \rho(\mathbf{y}) K_R(\mathbf{x}-\mathbf{y}) 
		\du \mathbf{x} \du \mathbf{y} + \mathcal{E}_R,
	\end{align*}
where $\mathcal{E}(\rho) = \mathcal{E}_R^+(\rho) + \mathcal{K}_R(\rho)$ is
the functional decomposition for $\mathcal{E}(\rho)$ that arises from
the optimal dual decomposition.  We now arrive at the main observation:
\begin{proposition}\label{Prop_supp}(Sets where $\mathcal{E}(\rho)$ is convex)
	Consider a candidate minimizer $\rho^*(\mathbf{x})$ with support $S_*$, 
	and suppose that the support 
	of $F_{\rho}(\mathbf{x}) = \rho^*\circ \rho^*$ lies in the support of 
	$F_R(\mathbf{x})$, i.e.,
	$\mathrm{supp}(\rho^*\circ\rho^*) \subseteq \mathrm{supp}(F_R)$.  Then 
	$\mathcal{E}(\rho)$ is convex on the space of probabilities having
	support $S_*$. In other words, $\mathcal{E}(\rho)$ is convex when restricted
	to the set 
	\[
	\mathcal{B}_* := \Big\{ \rho(\mathbf{x}) \in \mathcal{C}_1, 
	\int_{\Omega} \rho(\mathbf{x}) \du \mathbf{x} = 1, \; \mathrm{supp}(\rho) \subseteq S_*\Big\}.
	\]	
\end{proposition}
\begin{proof}
	The proof uses the dual decomposition and the complementary support equations 
	(\ref{ContinuousWplusSupport}) or (\ref{DiscreteWplusSupport}). 
	It will be sufficient to show that for any $\rho(\mathbf{x}) \in \mathcal{B}_*$,
	we have $\mathcal{E}_R^+(\rho) = 0$. This will imply that on the space
	$\mathcal{B}_*$, the functional $\mathcal{E}(\rho) = \mathcal{K}_R(\rho)$ is
	convex.	
	
	Suppose that $\rho(\mathbf{x}) \in \mathcal{B}_*$, then 
	$\mathrm{supp}(\rho\circ\rho) \subseteq \mathrm{supp}(\rho^*\circ\rho^*)$
	by a basic property of the auto-correlation of probabilities. Using the hypothesis
	in Proposition~\ref{Prop_supp}, one then has that 	
	$\mathrm{supp}(\rho\circ\rho)~\subseteq~\mathrm{supp}(F_R)$.  However 
	(\ref{ContinuousWplusSupport}) or (\ref{DiscreteWplusSupport}) guarantees that
	$W_R^+(\mathbf{x}) = 0$ for any $\mathbf{x} \in \mathrm{supp}(F_R)$, and hence 
	$W_R^+(\mathbf{x}) = 0$ for any $\mathbf{x} \in \mathrm{supp}(\rho\circ\rho)$.
	Therefore the integral
	$\mathcal{E}_R^+(\rho) = \frac{1}{2}\langle \rho\circ\rho, W_R^+\rangle = 0$.
\end{proof}

Proposition~\ref{Prop_ConicDecomp} shows that if a recovered minimizer 
$\rho^*(\mathbf{x})$ has an auto-correlation with support 
$\textrm{supp}(F_{\rho})\subseteq \textrm{supp}(F_R)$, then $\rho^*(\mathbf{x})$
satisfies the necessary condition for a candidate minimizer outlined in 
Remark~\ref{E_Convex_on_Z}.
In the subsequent numerical examples, we will observe that the recovery 
procedure outlined in Section~\ref{Sec_KLrecovery} will generate candidate minimizers
$\rho^*(\mathbf{x})$ that often satisfy the hypothesis in Proposition~\ref{Prop_supp}

Finally, we conclude this section with the observation that finding
analytic descriptions for sets $S_*$ in which the energy functional 
$\mathcal{E}(\rho)$, when restricted to $\rho(\mathbf{x})$ with 
$\textrm{supp}(\rho)\subseteq S_*$, is convex is not a simple problem.  The 
importance of the dual decomposition for $W(\mathbf{x})$ is that it is a 
constructive approach that allows one to find such sets $S_*$.  Specifically,
$W_R^+(\mathbf{x})$ and $K_R(\mathbf{x})$ are constructed analytically from 
$W(\mathbf{x})$; and if a set $S_*$ satisfies
the property that $\textrm{supp}(\rho) \subseteq S_* $ implies $\mathcal{E}_R^+(\rho) = 0$,
then $\mathcal{E}(\rho)$ is convex when restricted to probabilities with 
supports in $S_*$.

\section{Results: examples in one dimension} \label{sec:NumericalResults_d1}

\subsection{The Morse potential}\label{Subsec:Periodic_Morse_1d}

In this section we use the convex relaxation and recovery approach to 
generate candidate 
minimizers to the Morse potential on a periodic domain.  The Morse potential 
is a simple example of an attractive-repulsive potential
that has been used recently \cite{BernoffTopaz2013, LeverentzTopazBernoff2009, MogilnerKeshetBentSpiros2003} to model swarms and collective behavior in social 
phenomena.  On $\Omega = \mathbbm{R}$, we write the Morse potential as:
\begin{align} \label{MorsePotential}
	W_M(x) &= -G L e^{-|x|/l_1} + e^{-| x |/l_2}, \quad \quad G, L > 0, 
	\quad x \in \mathbbm{R},
\end{align}
where $L := l_1/l_2$ is a dimensionless quantity; $l_1$ and $l_2$ are 
the length scales associated with an attractive and repulsive 
force respectively; and $G$ denotes the strength of the attractive part of
the potential.
Mathematically, for different strengths of attraction and repulsion,
the Morse potential results in a non-convex energy functional 
$\mathcal{E}(\rho)$.  For computational purposes, we work on a periodic domain. 
To make $W_M(x)$ periodic, we introduce a box size $l_{box}$, and define the 
periodic Morse potential as:
\begin{align}\label{PeriodicMorseSum}
	W_{PM}(x) &= \sum_{n \in \mathbbm{Z}} W_M(x + n l_{box}).
\end{align}
Equation (\ref{PeriodicMorseSum}) can then be summed exactly 
by converting it into a geometric series.  To non-dimensionalize $W_{PM}(x)$, 
we use $l_{box}$ as the length scale, and replace $x \rightarrow x/l_{box}$.  We
also introduce the dimensionless parameter $\sigma := l_2/l_{box}$.  
After summation and non-dimensionalization, the function $W_{PM}(x)$, 
when restricted to one period $0 \leq x \leq 1$, takes the form:
\begin{align}\label{PeriodicMorse}
	W_{PM}(x) &= 
	\frac{-GL}{1 - e^{-1/(L\sigma)}} 
	\Big( e^{-x/(L\sigma)} +  e^{-(1-x)/(L\sigma)} \Big) + 
	\frac{1}{1 - e^{-1/\sigma}} \Big( e^{-x/\sigma} +  e^{-(1-x)/\sigma} \Big) 
	- \overline{W},
\end{align}
Here $\overline{W}$ is a constant\footnote{
$\overline{W} = \frac{-GL}{1 - e^{-1/(L\sigma)}} 
	A + 
	\frac{1}{1 - e^{-1/\sigma}}B$, where $A =  \int_{0}^1 e^{-x/(L\sigma)} +  e^{-(1-x)/(L\sigma)} \du x$, and $B = \int_{0}^{1} e^{-x/\sigma} +  e^{-(1-x)/\sigma} \du x$.}
added for numerical purposes to normalize
$W_{PM}(x)$ to have mean zero (see Property (W4)). 
When the box size $l_{box} \gg l_1, l_2$ is much larger than the 
interaction length scales, the periodic effects of $W_{PM}(x)$ are 
expected to be small, and minimizers of $\mathcal{E}(\rho)$ with
$W_{PM}(x)$ are expected to recover the results of minimizing $W_{M}(x)$
on the infinite line $\mathbbm{R}$.

In the following numerical examples we fix $\sigma = 0.1$, so that $l_{box}$ is 
several times larger than $l_1$ and $l_2$. 
To illustrate the utility of the new approach, we compute the phase
diagram for $W_{PM}(x)$ and characterize the results in the $(L, G)$ 
parameter plane.   This is done by systematically computing
the minimizer $F_R(x)$ and recovered $\rho^*(x)$ for 
every value of $(L, G)$.  We find that the qualitative properties, 
which are characterized by four different regions, A--D, 
in Figure \ref{PhaseDiagramMorse} are in agreement with the ones 
computed in \cite{LeverentzTopazBernoff2009}. In particular, 
the region D corresponds to the blow up
region observed in \cite{LeverentzTopazBernoff2009}. Within this region,
we observe a cascade where minimizers form lattices of Dirac masses--with 
progressively smaller lattice spacings, as $G$ decreases at a fixed value of $L$.

\begin{figure}[htb!]
	\centering
	\includegraphics[width = \textwidth]{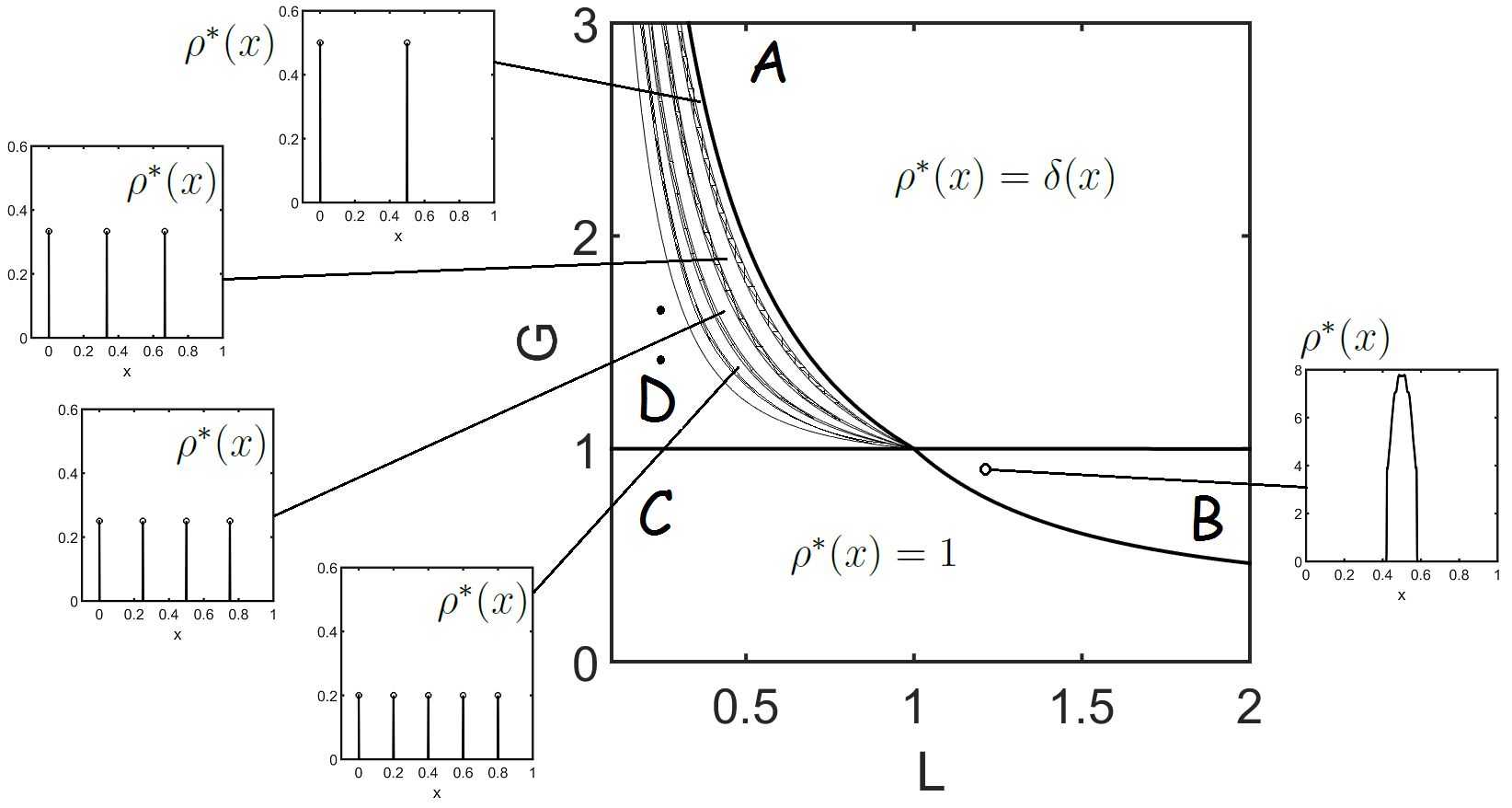}
	\caption{The phase diagram for the periodic Morse potential 
	(\ref{PeriodicMorse}) with $\sigma = 0.1$.  The $(L, G)$ 
	plane is partitioned by thick black lines into 4 regions, denoted by A--D, 
	which show qualitatively different minimizers $\rho^*(x)$ 
	(Note that
	these regions are similar to the ones observed in
	\cite{LeverentzTopazBernoff2009}).
	Region A: The global minimum
	is $\rho^*(x) = \delta(x)$. Region B: Densities $\rho^*(x)$ have a 
	non-zero width whose support is contained strictly inside $[0,1]$; and
	have a continuous auto-correlation $F_R(x)$. Here, one recovered 
	solution is shown for the parameter values $(L, G) = (1.2, 0.9)$. 
	Region C: $\rho^*(x) = 1$, corresponds to an evenly spread 
	probability distribution. Region D: Solutions are a collection of
	Dirac masses, and may form lattices. For instance, the white banded 
	regions show a cascade of lattice minimizers.
	Plotted are lattices with $2, 3, 4$, and $5$ evenly spaced 
	Dirac $\delta(x)$'s, and the number continues to increase
	as $G$ approaches 1.  
	The small transition regions (black shading), between 
	the lattice regions, may contain (possibly 
	infinitely) many different solutions.  Minimizers in regions A, C, 
	as well as the lattice solutions in D are exact global minimizers with 
	$\alpha = 1$ (See Remark \ref{Rmk:lattic_exact} and 
	Appendix A).}
	\label{PhaseDiagramMorse}
\end{figure}

Figures \ref{ResultsPeriodicMorse_Cvg}--\ref{ResultsPeriodicMorse} show explicit 
results for a fixed value of $(L, G) = (1.2, 0.9)$, that lies in the region 
where $F_R(x)$ is a continuous function.  Figure 
\ref{ResultsPeriodicMorse_Cvg} demonstrates the convergence of the 
Schulz-Snyder algorithm, while Figure \ref{ResultsPeriodicMorse_DualDecomposition}
shows the optimal dual decomposition for $W_{PM}(x)$.

Figure \ref{ResultsPeriodicMorse} together shows the recovered minimizer $\rho^*(x)$
(with a guarantee $\alpha = 0.99$),
along with the solutions to (R) and (D), i.e., $F_R(x), W_R^+(x), K_R(x)$.  The purpose of
showing both the solutions to (R) and (D) is to highlight the 
complimentarity conditions (\ref{ContinuousWplusSupport})--(\ref{ContinuousKSupport}).
Specifically, Subfigure \ref{ResultsPeriodicMorse}a shows $\rho^*(x)$
to have a support $S_*$ with length $|S_*| = 0.161$, which is consistent with
the histogram width observed in the particle simulations in Figure \ref{DiscreteParticles}. 
The auto-correlation
$F_{\rho}(x)$ also has a support $\textrm{supp}(F_{\rho}) = \textrm{supp}(F_R)$, and 
therefore is complementary to $W_R^+(x)$, i.e. $W_R^+(x) F_{\rho}(x) = 0$ for all 
$x\in \Omega$ (See Subfigure \ref{ResultsPeriodicMorse}b).  Hence, $\rho^*(x)$ 
satisfies the hypothesis in Proposition \ref{Prop_supp}, which implies that
$\mathcal{E}(\rho)$ is convex when restricted to probabilities having support
with a width of $\sim 0.161$.  Subfigure \ref{ResultsPeriodicMorse}c, shows
the complimentarity condition (\ref{ContinuousKSupport}).
Finally, we note that the size $|S_*|$ emerges as a new length 
scale for the particle density, and is 
exactly $1/2$ of the length where $W_R^+(x) = 0$. 


\begin{SCfigure}[5]
	\includegraphics[width = 0.4\textwidth]{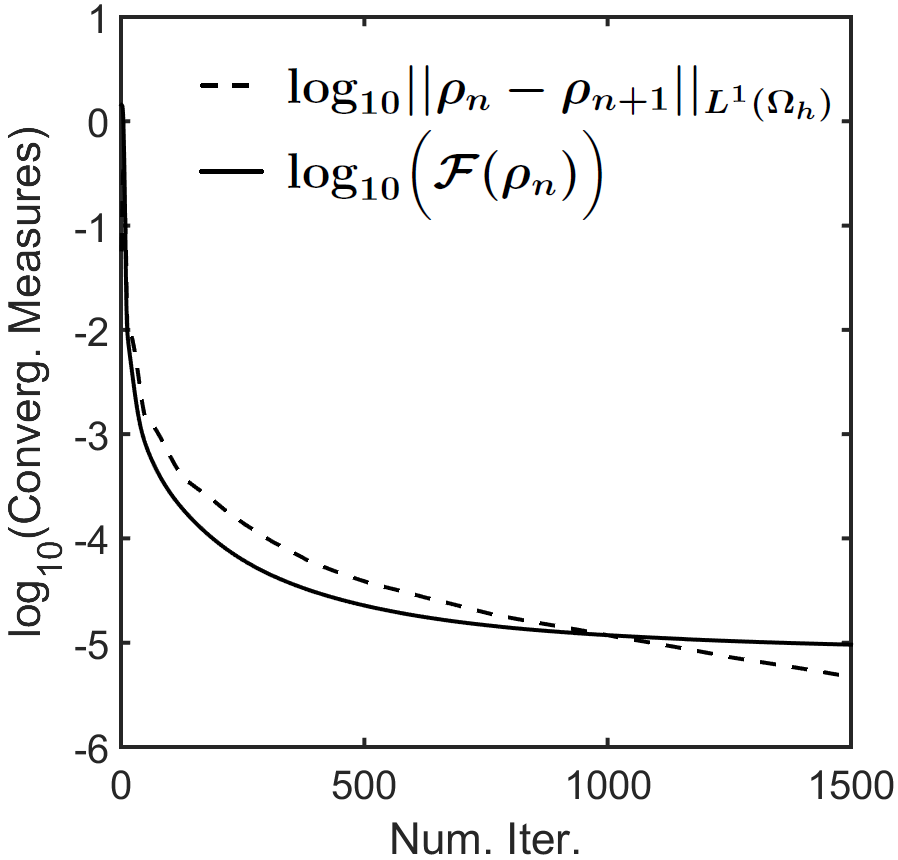} 
     \caption{
     Convergence versus iteration for \\
	 $\log_{10}||\rho_{n}-\rho_{n+1}||_{L^1(\Omega_h)}$ 
     (dashed curve), and \\
     $\log_{10}\big(\mathcal{F}(\rho_{n}) \big)\rightarrow  7.98 \times 10^{-6}$
     (solid line) in 
     the Schulz-Snyder algorithm. The quantities are
     for the periodic Morse potential (\ref{PeriodicMorse}) with $\sigma = 0.1$, 
     $(L, G)=(1.2, 0.9)$, and grid $n = 800$. 
 	 The non-zero value of $7.98 \times 10^{-6}$ is the result of a mismatch 
	 between the converged $F_{\rho^*}(x)$, and target $F_R(x)$,
	 and may be due to: (i) round-off or tolerance errors introduced
	 into the numerical discretizations; or (ii) a fundamental 
	 limitation that for the analytic solution $F_R(x)$ at hand,
	 there may not exist a $\rho^*(x)$ that 
	 exactly satisfies 
	 the sufficient conditions in Remark~\ref{Rmk:eqv_suff_cond}.  } 
     \label{ResultsPeriodicMorse_Cvg}
\end{SCfigure}

\begin{figure}[htb!]
 \includegraphics[width = \textwidth]{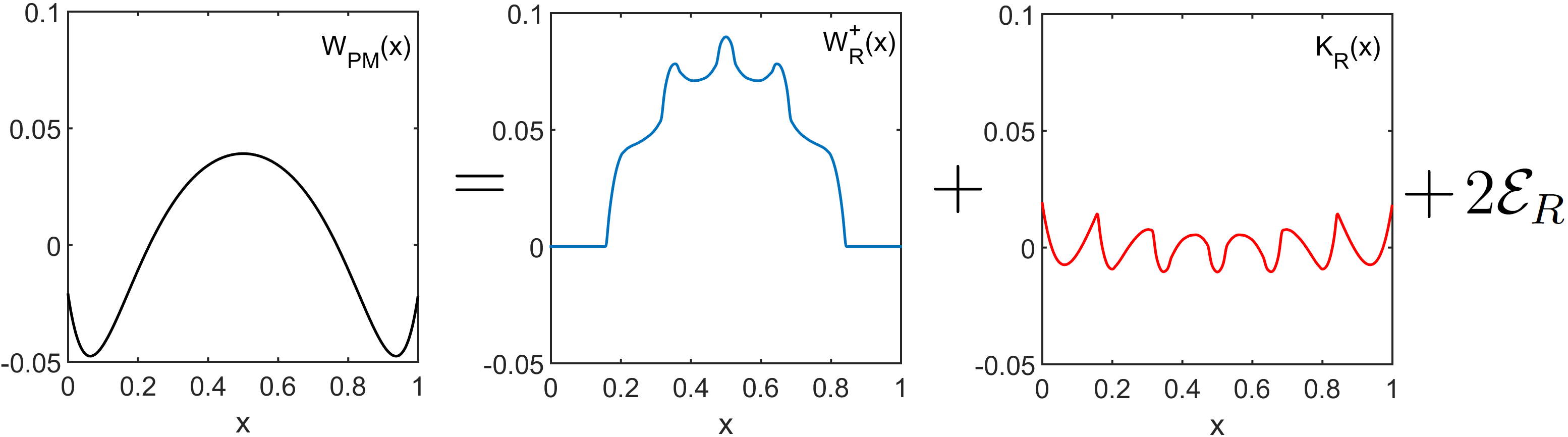}
  \caption{The figure shows the optimal
     dual decomposition (\ref{Optimal_Decomposition}) for the periodic
     Morse potential $W_{PM}(x)$ defined in (\ref{PeriodicMorse}).  Here the
     parameters are $\sigma = 0.1$, $(L, G)=(1.2, 0.9)$, and grid $n = 800$.  
     Note that the cosine coefficients
     of $K_R(x)$ are non-negative, while $\mathcal{E}_R$ is the largest possible
     constant as described by the solution to (D).  } 
     \label{ResultsPeriodicMorse_DualDecomposition}
\end{figure}

\begin{figure}[htb!]
	\centering
    \subfloat[(a) $\rho^*(x)$, $\Lambda(x)$.]
    { \includegraphics[width = 0.317\textwidth]{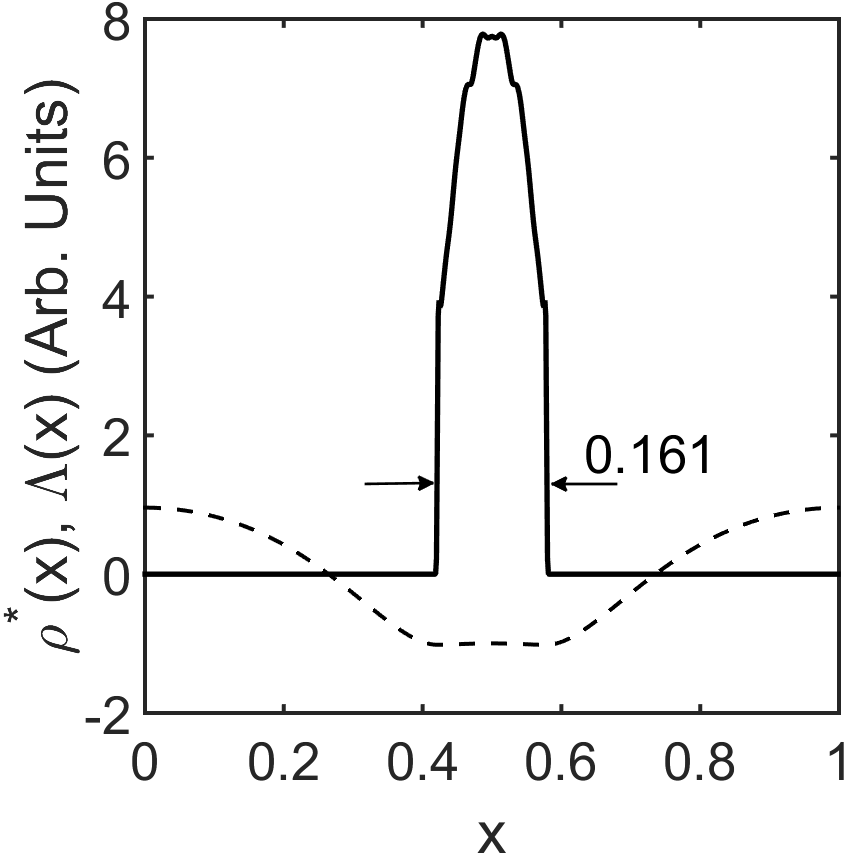} }    
	\subfloat[(b) $F_R(x)$, $F_{\rho^*}(x), W_R^+(x)$.]{ \includegraphics[width = 0.317\textwidth]{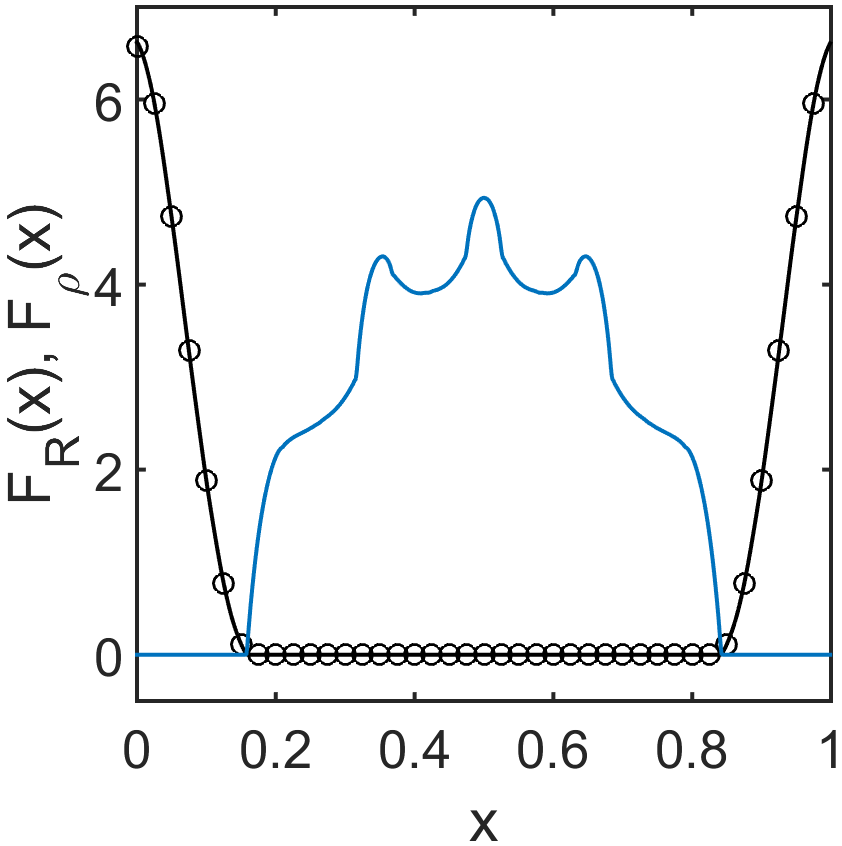} }
	\subfloat[(c) $\hat{\rho}^*(k), \hat{F}_R(k), \hat{K}_R(k)$.]{ \includegraphics[width = 0.32\textwidth]{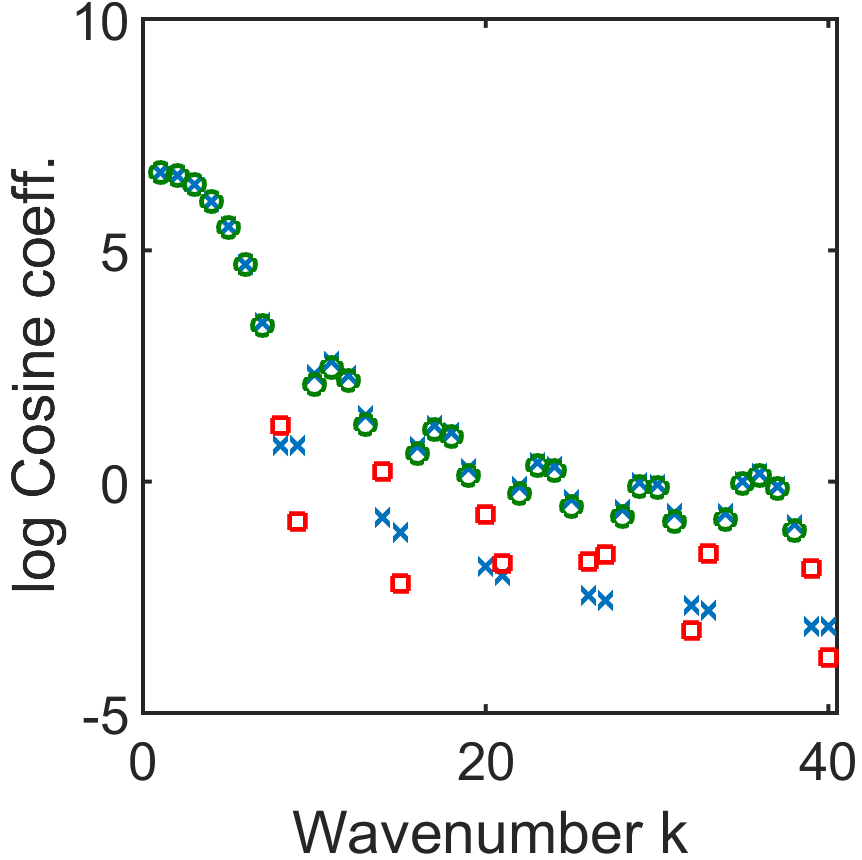} }\\
	 \caption{Example for the 
     Morse potential (\ref{PeriodicMorse}) with $\sigma = 0.1$, 
     $(L, G)=(1.2, 0.9)$, and grid $n = 800$. Figure
     (a) shows the recovered minimizer $\rho^*(x)$ (solid curve), which has a 
     guarantee of $\alpha = 0.99$, along with
     $\Lambda(x)$ (dashed curve, arbitrary units). The width of the support
     of $\rho^*(x)$ is $\sim 0.161$. Figure (b) shows the auto-correlation
     $F_{\rho^*}(x)$ (solid line), target auto-correlation 
     $F_R(x)$ (circles), and a re-scaled $W_R^+(x)$ (blue curve). Here 
     $W_R^+(x)$ is drawn to show that $F_{\rho}(x)$ and $W_R^+(x)$ have 
     complementary supports, i.e. $W_R^+(x)F_{\rho}(x) = 0$.
     This implies that $\rho^*(x)$ satisfies the hypothesis
     in Proposition \ref{Prop_supp}, and therefore $\mathcal{E}(\rho)$ is convex
     when restricted to probabilities with a width $\sim 0.161$. 
     Figure (c) shows that the cosine coefficients 	      
     (coefficients not plotted are numerically zero) 
     $\hat{F}_R(k)$ (green circles), and 
     $\hat{K}_R(k)$ (red squares), 
	 have complementary support for different values of $k$, i.e.,
	 $\hat{F}_R(k) \hat{K}_R(k) = 0$. Here the cosine coefficients
     $\hat{F}_{\rho^*}(k)$ (blue crosses) of the recovered solution
     are shown for reference.
     } 
     \label{ResultsPeriodicMorse}
\end{figure}

\subsection{A local potential}
In the context of social interactions, recent work \cite{MotschTadmor2014} has 
focused on a class of local interaction potentials where $W(x)$ has 
compact support.  In this section we examine the approximate global minimizers 
and dual decomposition for a continuous periodic version of the local potential 
examined in \cite{MotschTadmor2014}:
\begin{align} \label{localpotential}
	\psi(x) &= \left\{
\begin{array}{ll}
0.1, & |x| \leq \frac{1}{2}, \\
9|x|-4.4, & \frac{1}{2} < |x| \leq \frac{3}{5}, \\
1, & \frac{3}{5} < |x| \leq \frac{9}{10}, \\
10-10|x|, & \frac{9}{10} < |x| \leq 1, \\
0, &|x| > 1, 
\end{array} \right. \quad \quad \textrm{for } x \in \mathbbm{R}, \\ \nonumber
	W(x) &= \sum_{n \in \mathbbm{Z}}\Big( \psi\Big( \frac{x + n}{l_c}\Big) - \overline{\psi}\Big).
\end{align}
Figure \ref{LocalPotential}c shows the potential $\psi(x)$, which differs 
primarily from the one in \cite{MotschTadmor2014} by replacing the discontinuous 
jumps (at range values $0.1$, $1$ and $0$) by linear interpolation. The 
quantity $l_c > 0$ enters as the (dimensionless) ratio of the local interaction 
length to periodic domain length, with $W(x)$ entering in as a full periodic 
potential. One might expect in the limit $l_c \ll 1$ to recover the 
characteristics of the non-periodic model.  For $l_c = 0.1$, which is 
commensurate with the periodic domain length, and $n = 360$ grid points, one 
recovers the auto-correlation with 10 equispaced Dirac masses
\begin{align}
	F_R(x) &= \sum_{s \in S} f_R(s) \delta(x - s), \\ \nonumber
	f_R(s) &= \frac{1}{10}, \textrm{ where } S 
	= \Big\{0, \frac{1}{10}, \frac{2}{10}, \ldots, \frac{9}{10} \Big\}.
\end{align}
Since $F_R(x) = F_R\circ F_R$, letting $\rho^*(x) = F_R(x)$ recovers the 
exact auto-correlation and hence is a global minimizer with guarantee 
$\alpha = 1$.   Figure \ref{LocalPotential}a shows the dual decomposition of 
$W(x) = W_R^+(x) + K_R(x) + 2\mathcal{E}_R$. Numerically, it is observed that 
both $W_R^+(x)$ and $K_R(x)$ are constant in regions where the local potential
$W(x) = 0$, so that they too are 
effectively local potentials.  As a final remark, if $l_c$ is not taken as an 
integer fraction of the domain length, or the grid spacing $h = 1/n$ (see Appendix B), is not 
commensurate with the spacing of the Dirac masses, one may have non-lattice
minimizers that become sensitive to the number of grid points $n$, and tolerance
chosen in the numerical optimization routine. 


\subsection{A regularized power law potential}
Power law potentials are often used in models of social dynamics. Here we 
illustrate the approach for a regularized power law potential on a periodic 
domain. Set
\begin{align} \nonumber
	W_{p}(x) &= x^{-0.4} - \frac{1}{3.5}x^{-0.2} - \overline{W}, 
	\\ \label{RegPowerlaw}
	W(x) &= W_p( x + \epsilon) + W_p(1-x + \epsilon), \hspace{10mm} 
	\textrm{for } x \in [0, 1], \textrm{ extended periodically.}
\end{align}
The exponents $-0.4, -0.2$ and parameter $3.5$ are chosen arbitrarily. 
The parameter $\epsilon = 0.01$ is taken to regularize the 
discontinuity at $x = 0$.  Without the regularization, the value $W(0)$ becomes 
undefined and the optimization routine in Appendix B must be 
modified to obtain a convergent minimizer to (R). As a note, the shape of the 
potential is somewhat sensitive to the parameters $\epsilon$ and deviations 
from the constant $3.5$. The candidate minimizer $\rho^*(x)$ is shown in 
Figure~\ref{ResultsPwrlaw} and has a guarantee $\alpha = 0.988$. 


\begin{figure}[htb!] 
	\centering
	 \subfloat[(a) $W_R^+(x)$, $K_R(x)$.]
	 {\includegraphics[width = 0.33\textwidth]{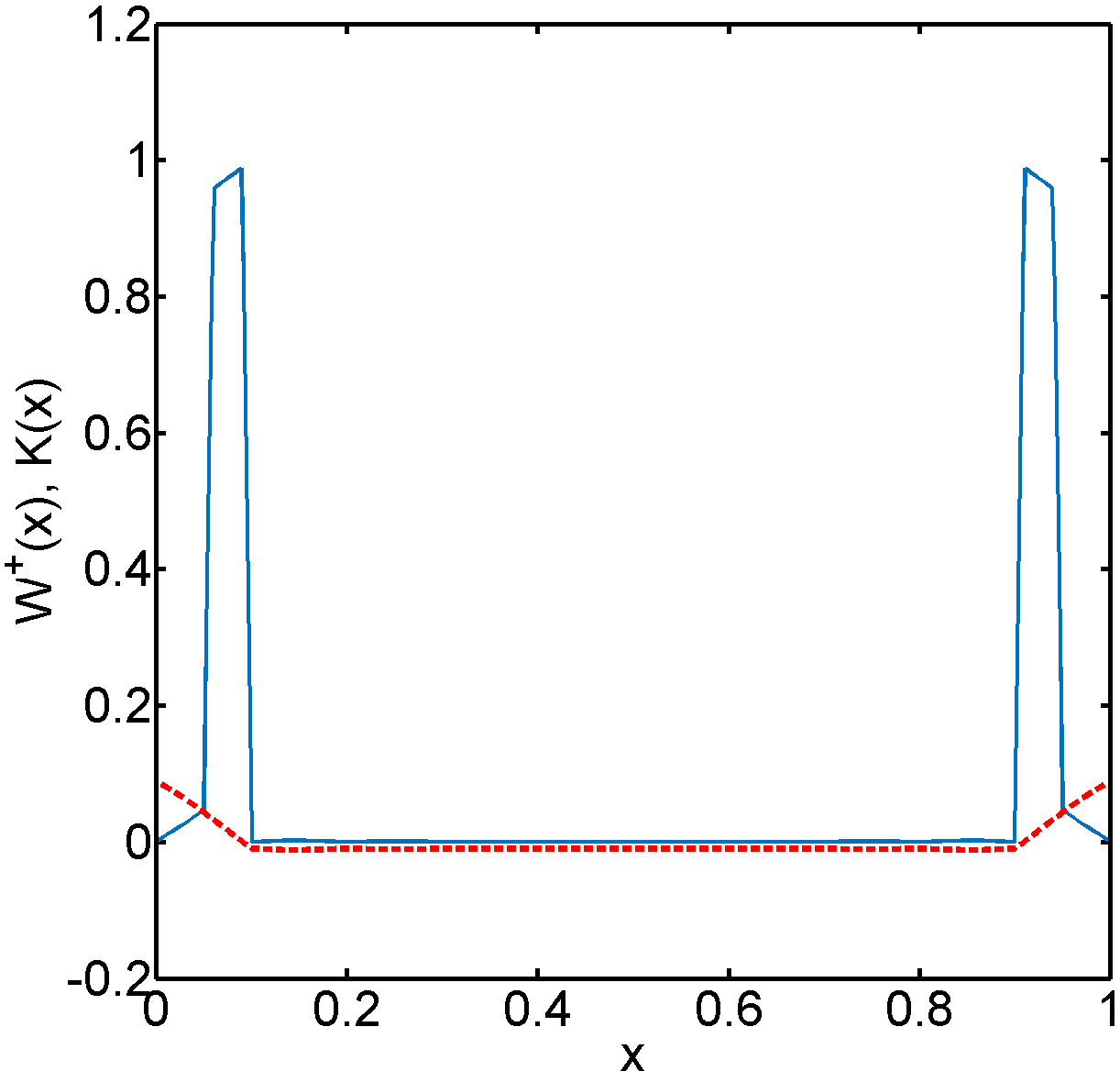}}
    \subfloat[(b) $F_R(x), F_{\rho^*}(x), \rho^*(x)$.]
    {\includegraphics[width = 0.338\textwidth]{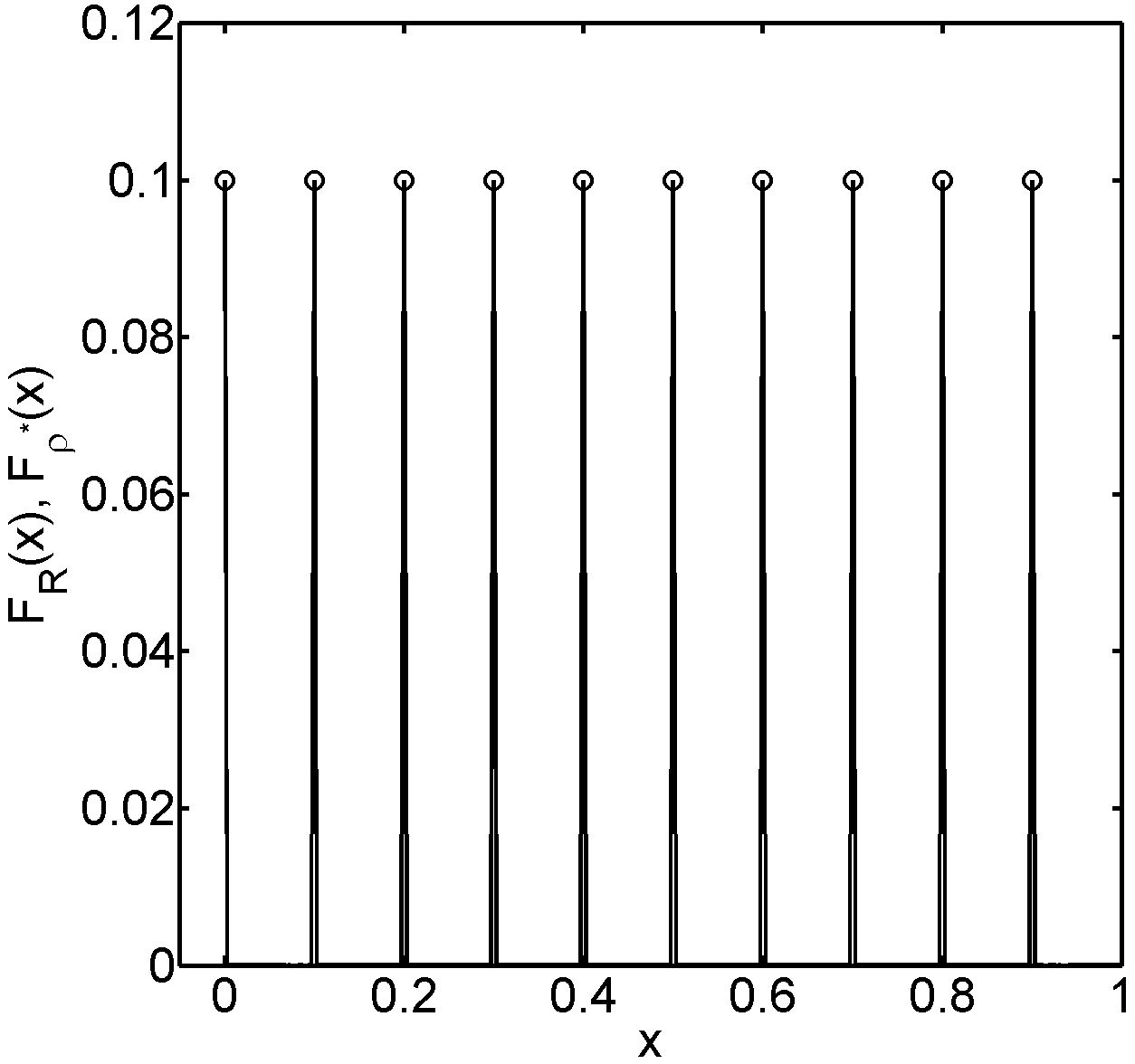}}
    \subfloat[(c) $\psi(x)$.]
    {\includegraphics[width = 0.32\textwidth]{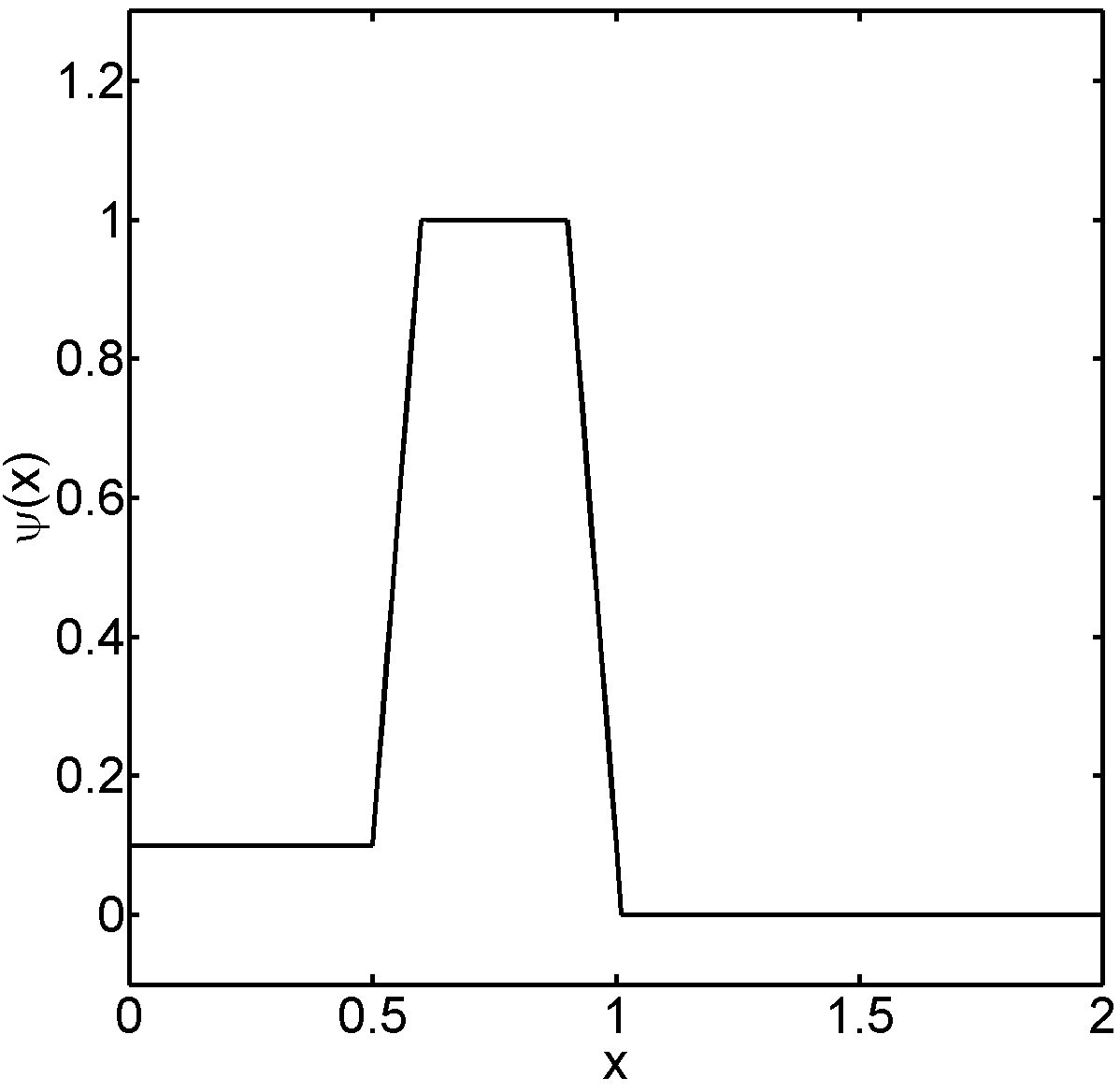}}
    \caption{Results for local potential (\ref{localpotential}), $n = 360$ grid.  
    (a) Optimal dual decomposition for $W(x)$ into $W_R^+(x)$ (blue curve), 
    and $K_R(x)$ (red curve). (b) $F_R(x)$ (circles), $F_{\rho^*}(x)=\rho_*(x)$ 
    (solid lines). The recovery is exact with $\alpha = 1$. 
    (c) Local interaction potential $\psi(x)$.}
 \label{LocalPotential}
\end{figure}

\begin{figure}[htb!] 
	\centering
	\subfloat[(a) $W_R^+(x)$, $K_R(x)$, $W(x)$.]
	{\includegraphics[width = 0.33\textwidth]{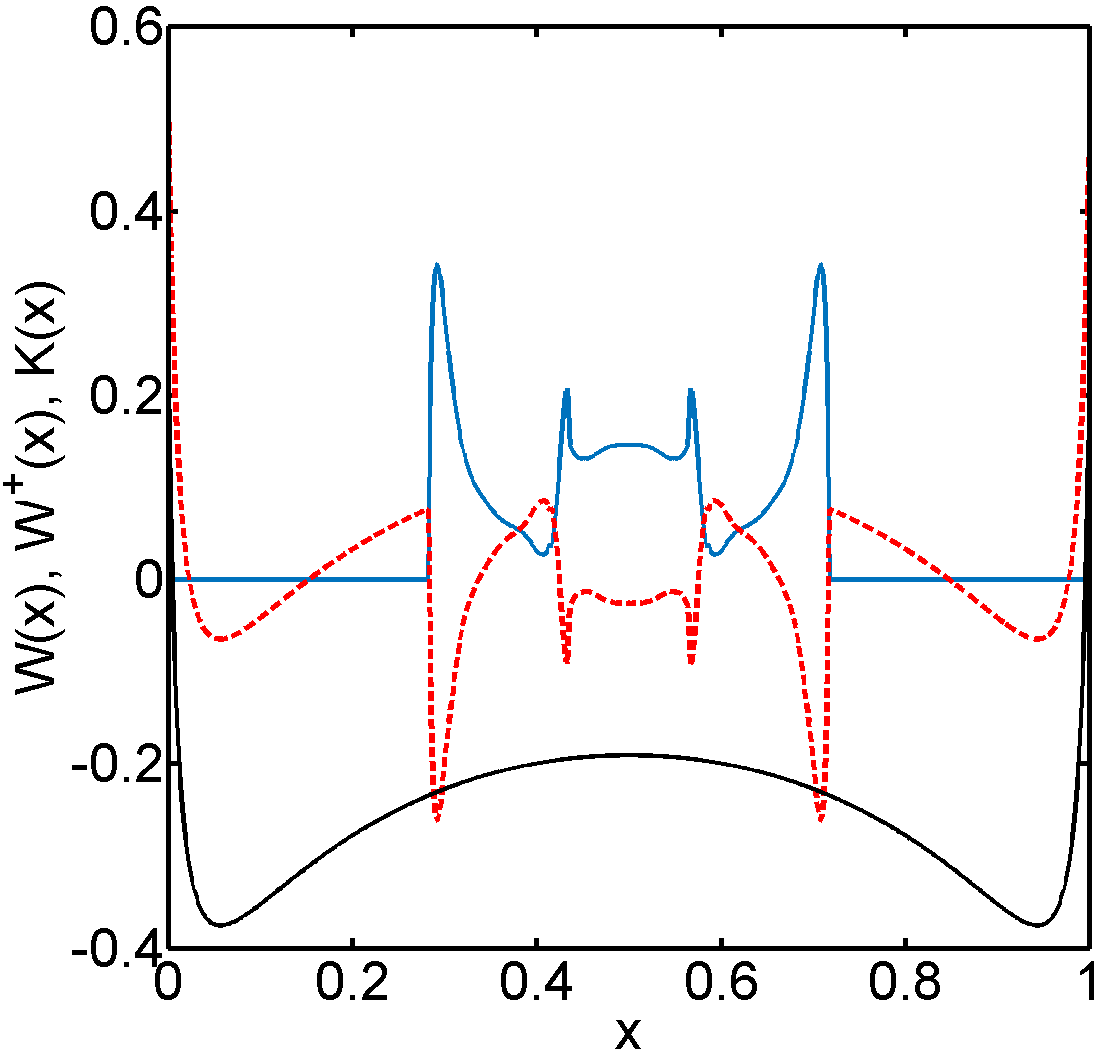}}
    \subfloat[(b) $\rho^*(x)$, $\Lambda(x)$.]
    {\includegraphics[width = 0.315\textwidth]{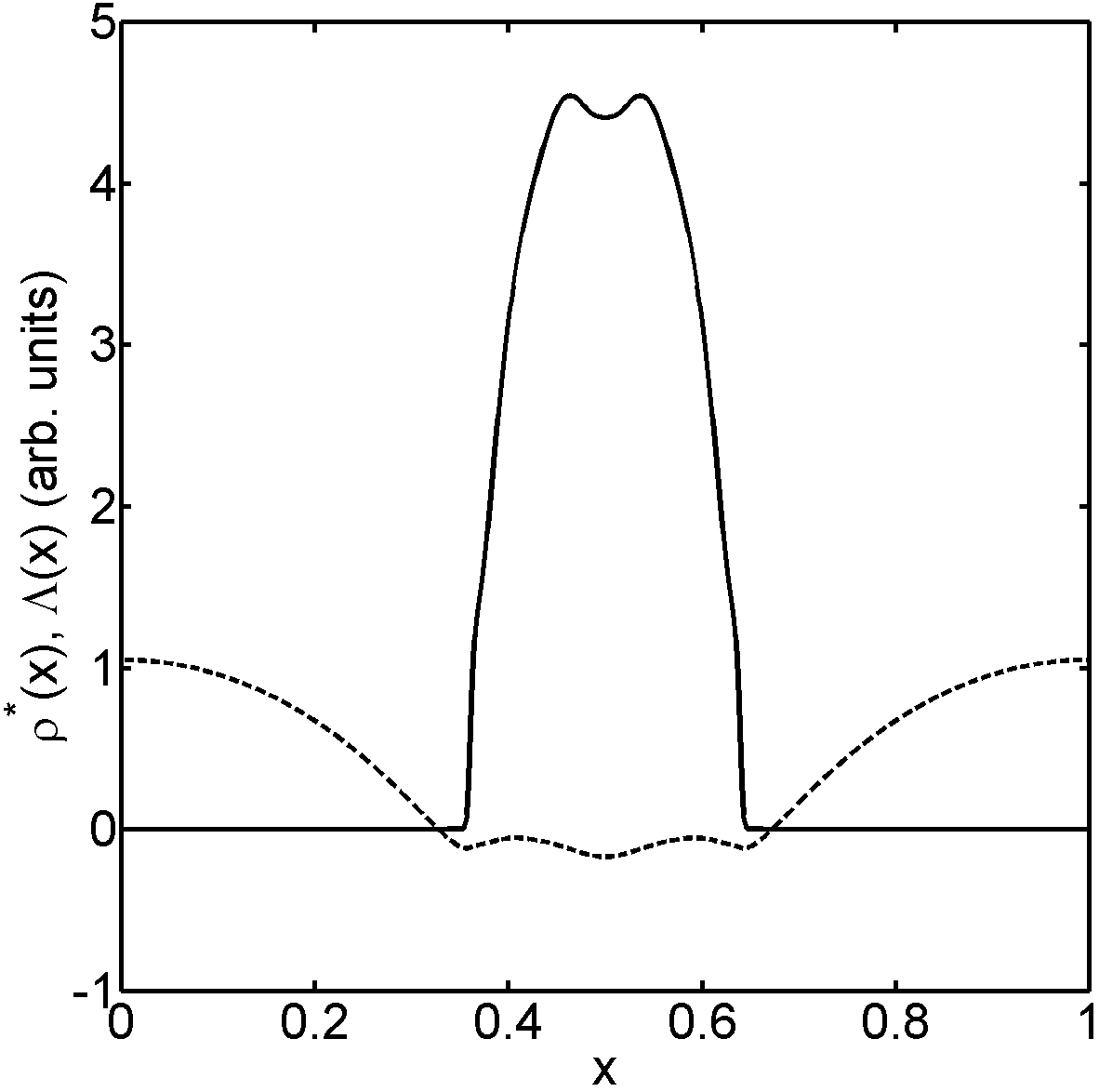}} 
    \subfloat[(c) $F_R(x)$, $F_{\rho^*}(x), W_R^+(x)$.]
    {\includegraphics[width = 0.333\textwidth]{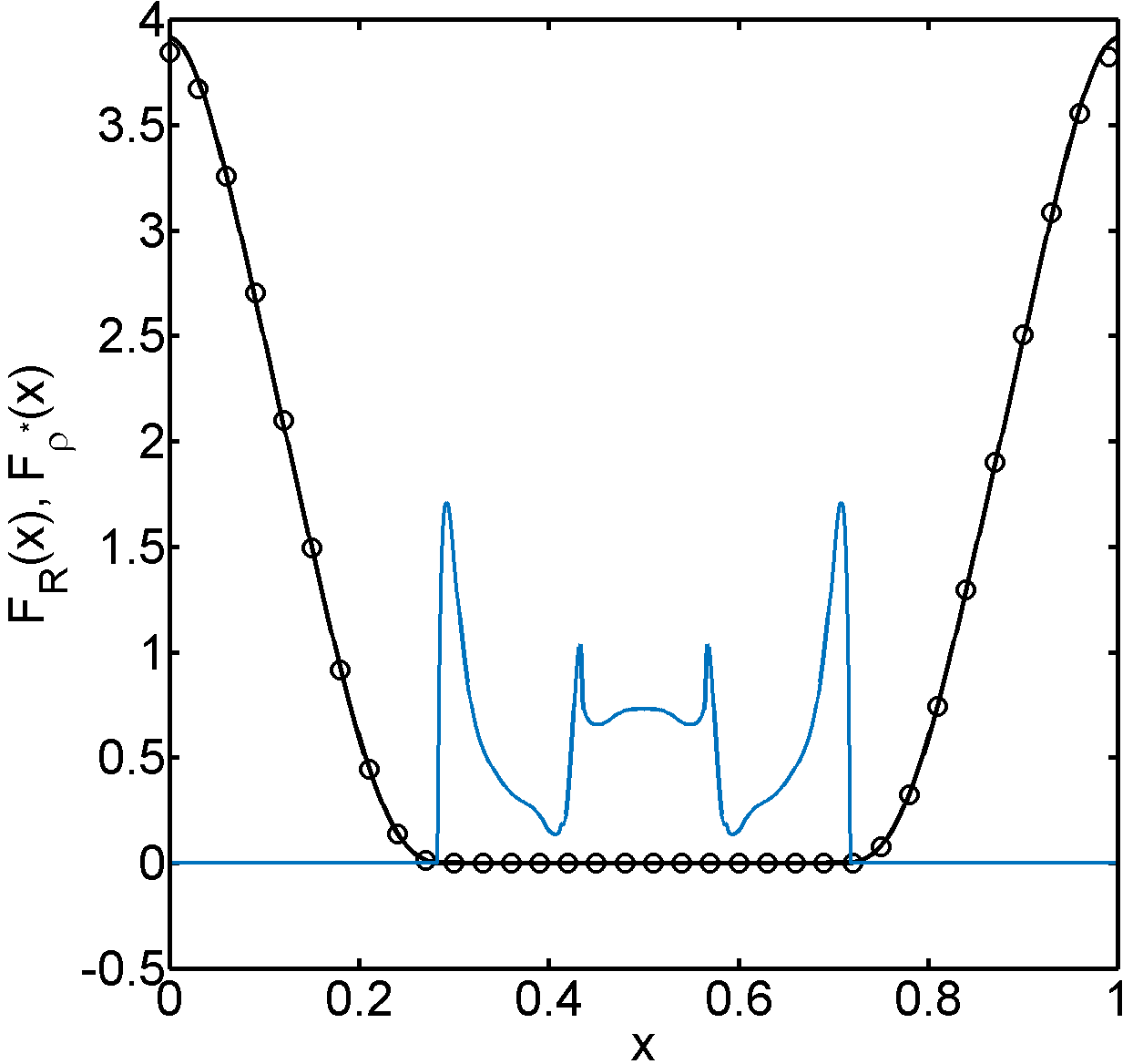}}
    \caption{Results for regularized power law potential (\ref{RegPowerlaw}), 
    $n = 1000$ grid.  (a) Optimal dual decomposition for $W(x)$ (black curve) 
    into $W_R^+(x)$ (blue curve), and $K_R(x)$ (red curve). (b) $\rho^*(x)$ 
    (solid) with guarantee $\alpha = 0.988$, re-scaled $\Lambda(x)$ (dashed) 
    with arbitrary units. (c) Auto-correlation $F_R(x)$ (dots), and 
    $F_{\rho^*}(x)$ (solid). Here $W_R^+(x)$ is plotted (blue curve, 
    arbitrary units) to
    show that $F_{\rho}(x)W_R^+(x) = 0$, thereby implying that $\rho^*(x)$ 
    satisfies the hypothesis in Proposition \ref{Prop_supp}. }
 \label{ResultsPwrlaw}
\end{figure}

\begin{figure}[htb!] 
	\centering
	\subfloat[(a) $W_R^+(x)$, $K_R(x)$, $W(x)$.]
	{\includegraphics[width = 0.33\textwidth]{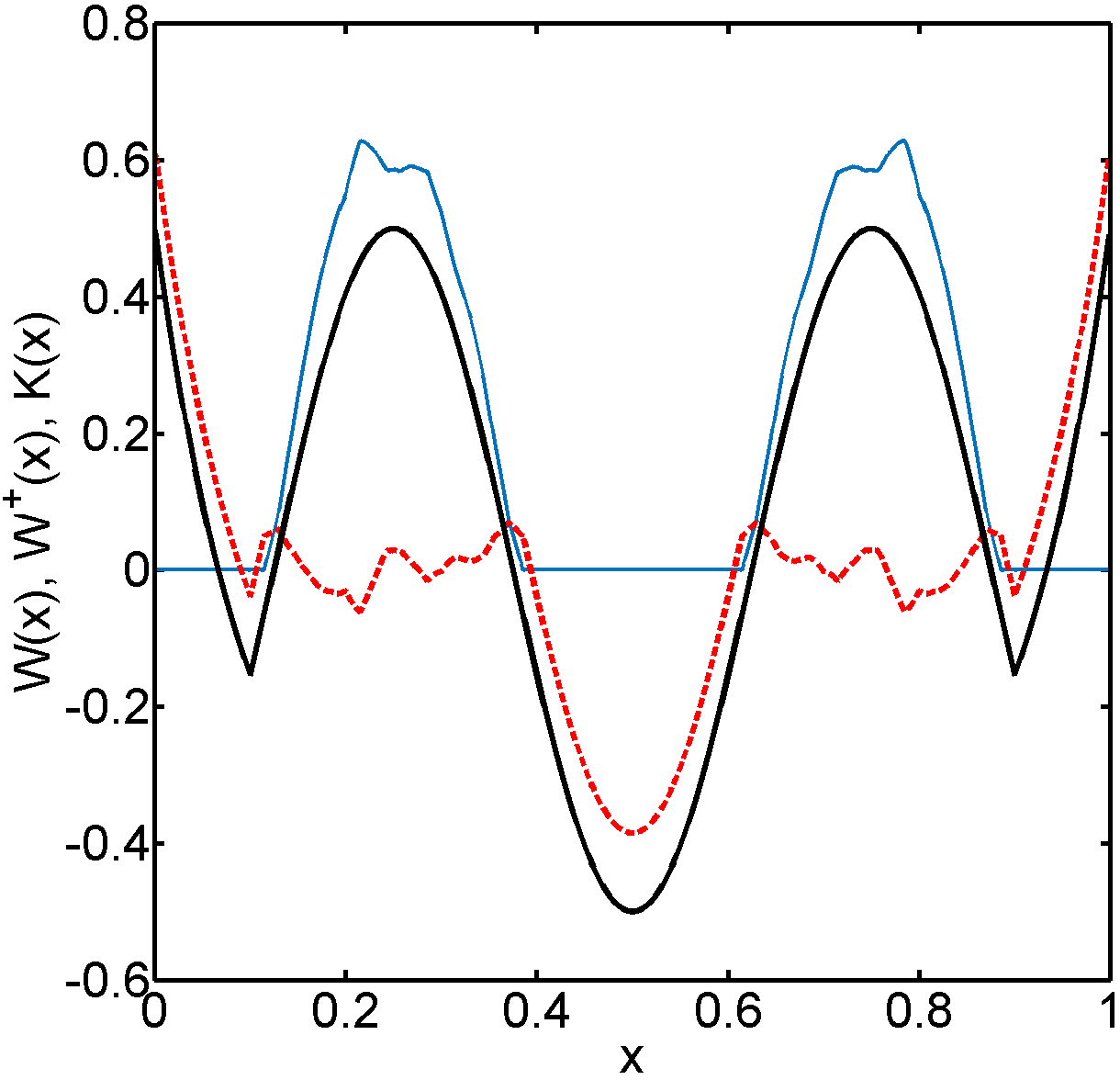}}
    \subfloat[(b) $\rho^*(x)$, $\Lambda(x)$.]
    {\includegraphics[width = 0.323\textwidth]{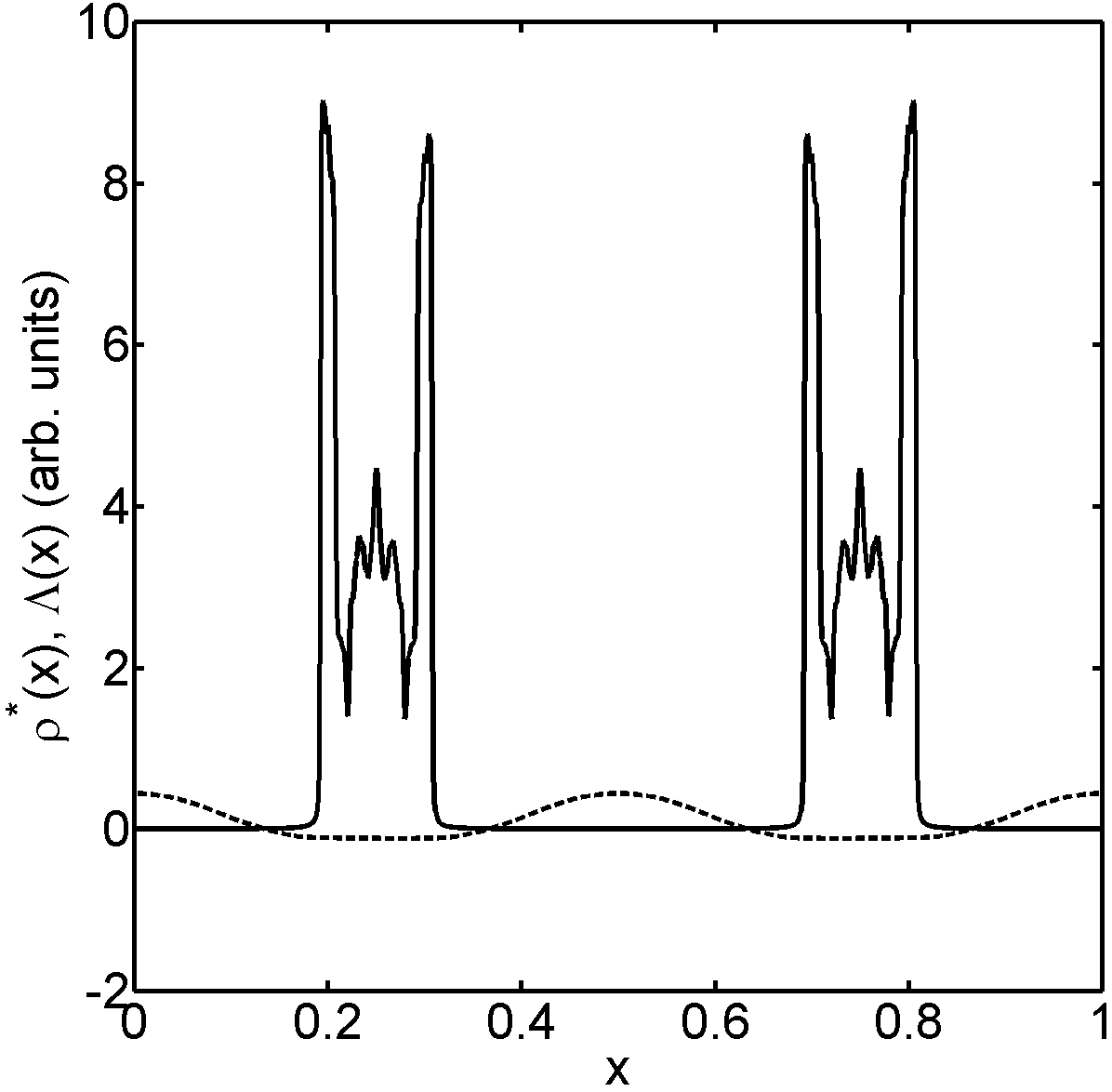} }
    \subfloat[(c) $F_R(x)$, $F_{\rho^*}(x), W_R^+(x)$.]
    {\includegraphics[width = 0.314\textwidth]{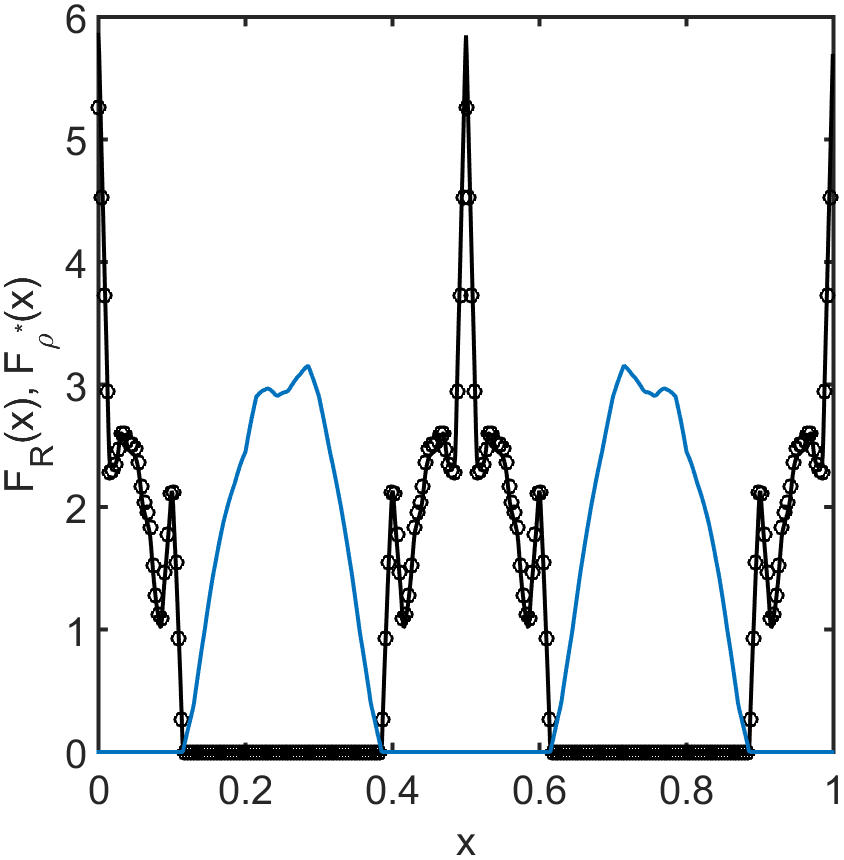} }
    \caption{Results for potential with multiple length scales 
    (\ref{MultiscalePot}), $n = 1024$ grid. (a) Optimal dual decomposition for 
    $W(x)$ (black curve) into $W_R^+(x)$ (blue curve), and $K_R(x)$ (red curve). 
    (b) $\rho^*(x)$ (solid) with guarantee $\alpha = 0.988$, rescaled $\Lambda(x)$ 
    (dashed) with arbitrary units. (c) Auto-correlation $F_R(x)$ (dots), and 
    $F_{\rho^*}(x)$ (solid). Here $W_R^+(x)$ is plotted (blue curve, 
    arbitrary units) to
    show that $F_{\rho}(x)W_R^+(x) = 0$, thereby implying that $\rho^*(x)$ 
    satisfies the hypothesis in Proposition \ref{Prop_supp}.  }
 \label{Results_Multilen_scale}
\end{figure}


\subsection{A potential with multiple length scales}
Another interesting example occurs for potentials that promote several length 
scales by having multiple local minima in $W(x)$. As an example, take
\begin{align} \nonumber
	W_t(x) &= \max\{ 1- x, 0 \} - \overline{W}, \hspace{31.3mm}  
	\textrm{for } x \in [0, 1],\\ \label{MultiscalePot}
	W(x) &= W_t\Big(\frac{x}{10} \Big) + W_t\Big( \frac{1-x}{10}\Big) 
	- \frac{1}{2}\cos(4\pi x),\hspace{3mm} \textrm{for } x \in [0, 1], 
	\textrm{ extended periodically.}
\end{align}
Here $W_t(x)$ is a repulsive triangle potential which has non-negative 
cosine modes. The $\cos(4\pi x)$ term is added to make $\mathcal{E}(\rho)$ non-convex.  
We find a candidate minimizer $\rho^*(x)$ with a guarantee $\alpha = 0.988$ 
(see Figure \ref{Results_Multilen_scale}). The dual decomposition 
solution found in Figure \ref{Results_Multilen_scale}a also highlights the fact that 
$K_R(x)$ and $W_R^+(x)$ are not in general smooth.

\begin{remark}(Minimizers with disconnected supports)
	Other works, such as \cite{BernoffTopaz2013} have been successful in 
	characterizing global minimizers under the assumption that $\rho^*(x)$ has 
	connected support.  The recovery process for potential (\ref{MultiscalePot}) 
	yields an $F_R(x)$ with disconnected support, thereby resulting in $\rho^*(x)$ 
	with multiply connected supports. 
\end{remark}


\section{Results: examples in two dimensions} \label{sec:NumericalResults_d2}

The purpose of this section is to solve the relaxation (R), and compute 
minimizers in some examples with two spatial dimensions. 
The examples will also highlight several difficulties and drawbacks that
become more significant in higher dimensions. 
Specifically, due to the enlarged set of constraints encountered in two
dimensions, the numerical 
solution using MATLAB's solver become slow, and motivate the need for 
more efficient numerical schemes.  

Here, we focus on an attractive-repulsive
potential that shares some similarity to the periodic Morse potential: 
\begin{align} \label{PeriodicMorse2d}
	W(x,y) &= -G L e^{-\frac{1}{L} (|\sin(\pi x) | + |\sin(\pi y) |)} 
	+ e^{-(|\sin(\pi x)| + |\sin(\pi y) |)} - \overline{W}, 
	\quad \quad G, L > 0, \quad x,y \in \mathbbm{R}.
\end{align}
As a result of the similar parameterization to $W_{PM}(x)$, 
we may expect minimizers with the potential (\ref{PeriodicMorse2d}) 
for different $G$ and $L$ values to have qualitatively similar 
behavior to those described in the phase diagram in 
Figure \ref{PhaseDiagramMorse}.
For different fixed values of $(L, G)$, we solve (R) for $F_R(\mathbf{x})$,
followed by performing the recovery procedure outline in Section \ref{Sec_KLrecovery}.

\subsection{Solutions $F_R(x)$ to (R) that are continuous}
Using values of $(L,G) =(1.5, 0.9)$ in (\ref{PeriodicMorse2d}), we obtained 
a solution $F_R(\mathbf{x})$ that is continuous, as seen in Figure \ref{Contour2D}a. 
In the numerical solution, we were limited to a coarse $40 \times 40$ grid due 
to the increased solution times required by MATLAB's solvers.   In future work
we plan to increase the efficiency of the solvers so that larger spatial discretizations 
may be used.  
Despite the relatively coarse mesh, we still resolved
a numerical solution to $F_R(\mathbf{x})$, which likely has an error to the 
true solution that is first order, i.e. $\mathcal{O}(1/n)$.
We also set the built in MATLAB tolerance to $10^{-8}$. 
The recovered candidate $\rho^*(\mathbf{x})$ (see Figure \ref{Contour2D})
was found to have a guarantee of $\alpha = 0.99$, and a relative entropy to $F_R(\mathbf{x})$
of
$\mathcal{F}(\rho^*) = 0.0011$.
We now make several remarks on the characteristics of $\rho^*(\mathbf{x})$:
\begin{enumerate}
	\item[(i)] The support of $F_{\rho}(\mathbf{x})$ is complementary to 
	$W_R^+(\mathbf{x})$, i.e., $F_{\rho}(\mathbf{x}) W_R^+(\mathbf{x}) = 0$.
	This implies that $\rho^*(\mathbf{x})$ satisfies Proposition \ref{Prop_supp},
	and hence $\mathcal{E}(\rho)$ is convex when restricted to densities having
	a support contained in the support of $\rho^*(\mathbf{x})$.
	\item[(ii)] The solution $\rho^*(\mathbf{x})$ with support $S_*$, 
	exhibits spikes at the four corners of the support.  
	To provide some explanation for the 
	spikes, note that the previous item (i) implies that 
	$\mathcal{E}(\rho^*) = \mathcal{K}_R(\rho^*)$ (see also Proposition
	\ref{Prop_supp}). 
	The spikes may then be attributed to the recovered
	solution $\rho^*(\mathbf{x})$ wanting to minimizing the convex 
	part of the energy $\mathcal{K}_R(\rho)$ that arises from an
	interaction potential $K_R(\mathbf{x})$. 
	For this example $W_R^+(\mathbf{x})=0$ in a diamond neighborhood 
	near the origin, so that within this region	
	$K_R(\mathbf{x}) = W(\mathbf{x}) - 2\mathcal{E}_R$ contains
	the attractive-repulsive behavior of $W(\mathbf{x})$. Hence, 
	$\rho^*(\mathbf{x})$ can be thought of as a density that arises 
	from locally repelling particles confined to the set $S_*$. 
	As a result, the majority of the density concentrates
	near the boundary and corners of $S_*$.			 	
	\item[(iii)] Figure \ref{Figure2d_ParticleExample} shows the 
	support of the recovered minimizer $\rho^*(\mathbf{x})$, and the
	steady state arrangement of $N = 1000$ particles obtained
	from the gradient flow of equation (\ref{DiscreteGradFlow}) (using 
	random initial data). 
	The recovered minimizer  
	identifies the emergent length scale, and pattern obtained by the 
	collective interaction of a large number of particles.
\end{enumerate}

\begin{figure}[htb!] 
	\centering
	\subfloat[(a) $F_R(\mathbf{x})$.]{\includegraphics[width = 0.4\textwidth]{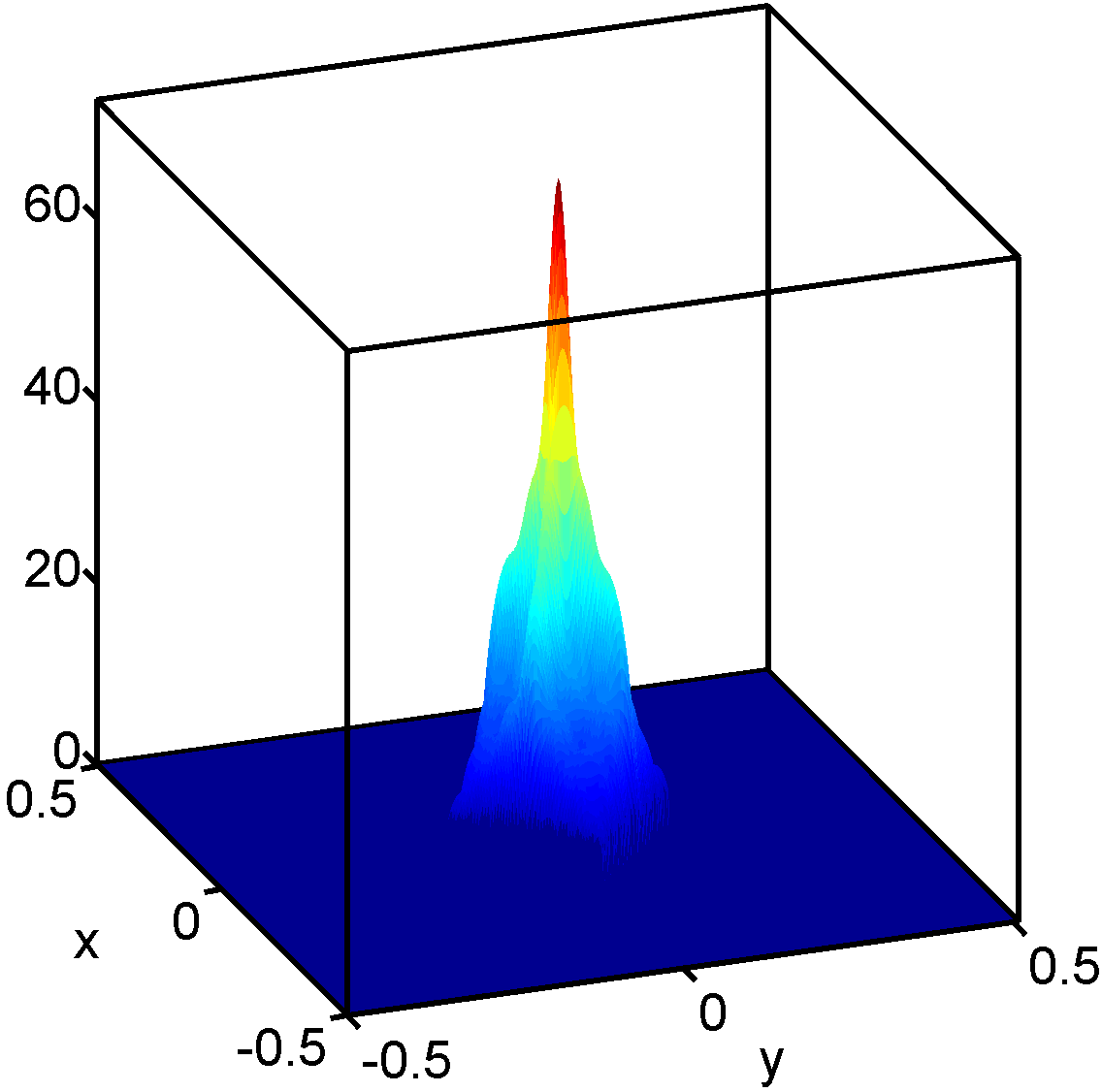} }    
   	\subfloat[(b) $\rho^*(\mathbf{x})$.]{\includegraphics[width = 0.38\textwidth]{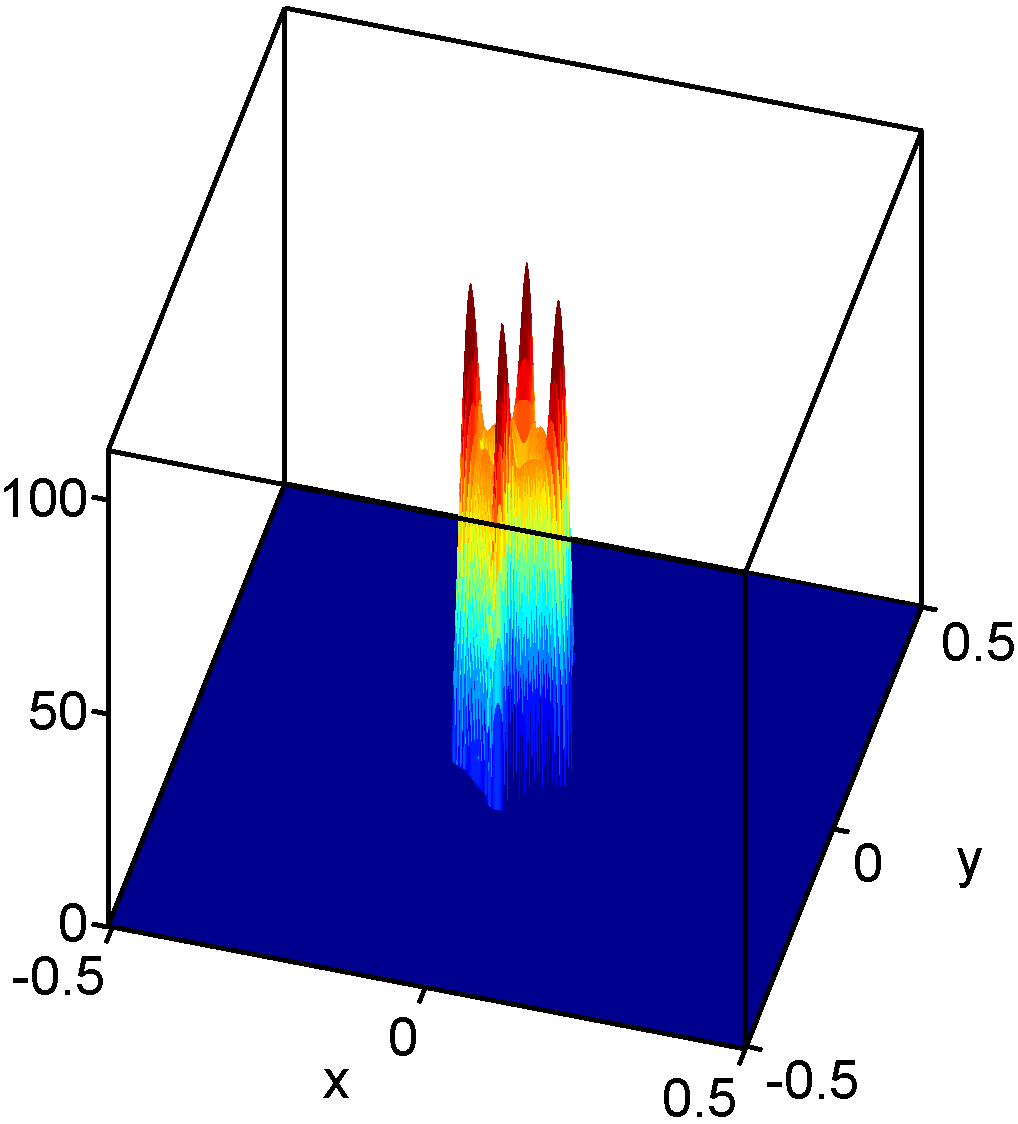}} \\
   	\subfloat[(c) $F_R(\mathbf{x})$.]{\includegraphics[width = 0.3\textwidth]{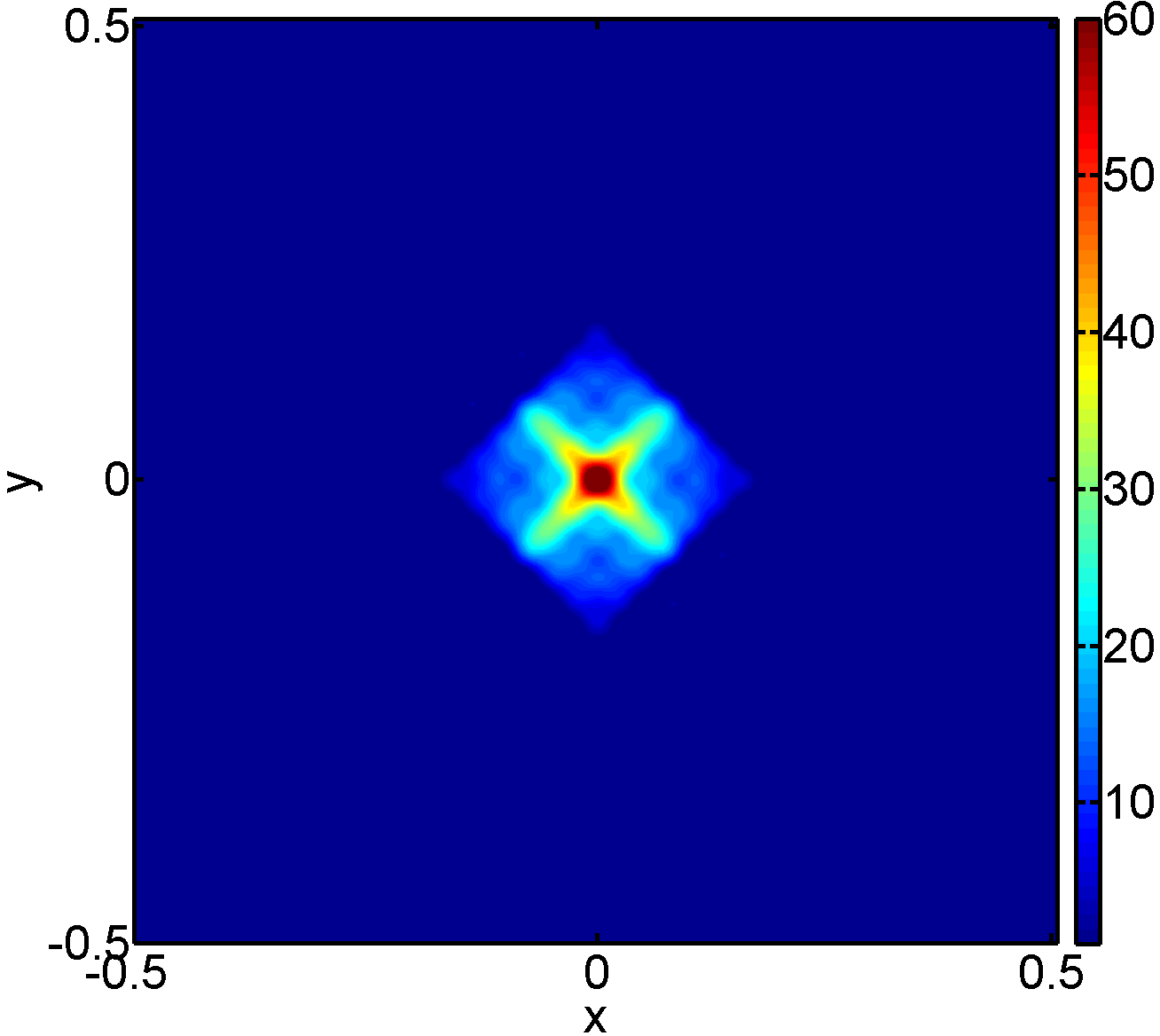}}
   	\subfloat[(d) $W_R^+(\mathbf{x})$.]{\includegraphics[width = 0.3\textwidth]{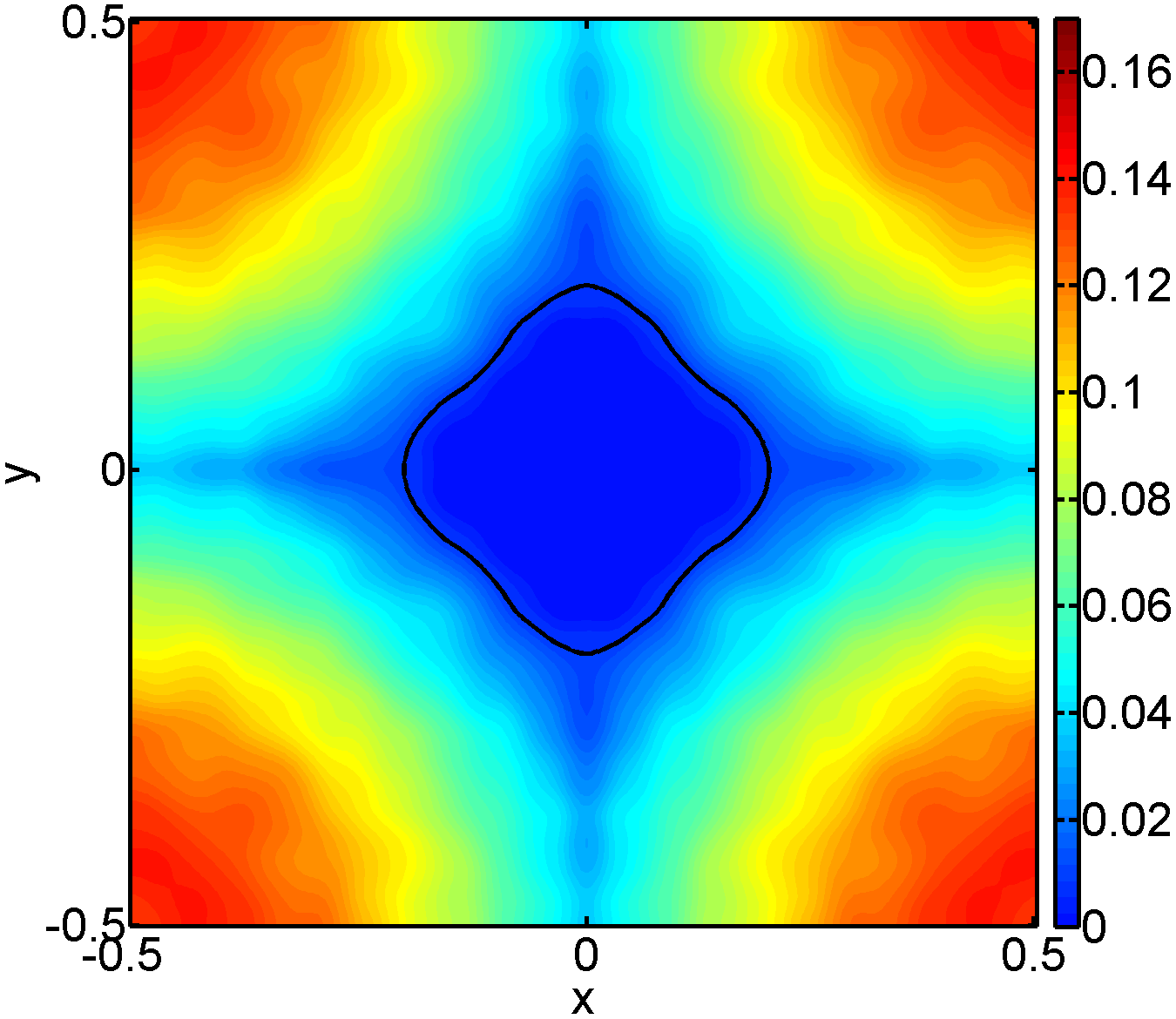}}
    \caption{
    Results for the two dimensional periodic potential in equation (\ref{PeriodicMorse2d}) 
    with $(L, G) = (1.5, 0.9)$, and $40\times 40$ grid. 
    (a) Target auto-correlation $F_R(\mathbf{x})$. 
    (b) $\rho^*(\mathbf{x})$ with guarantee $\alpha = 0.99$. 
    (c) Contour plot for $F_R(\mathbf{x})$ showing the support. 
    (d) Contour plot of $W_R^+(\mathbf{x})$ with a black line indicating the 
    region where $W_R^+(\mathbf{x}) = 0$. Note that 
    $W_R^+(\mathbf{x}) F_{\rho}(\mathbf{x}) = 0$, implying that 
    $\rho^*(\mathbf{x})$ satisfies Proposition \ref{Prop_supp}. }
 \label{Contour2D}
\end{figure}

\begin{figure}[htb!] 
	\centering
   	\subfloat[(a) $\rho^*(\mathbf{x})$.]{\includegraphics[width = 0.4\textwidth]
   	{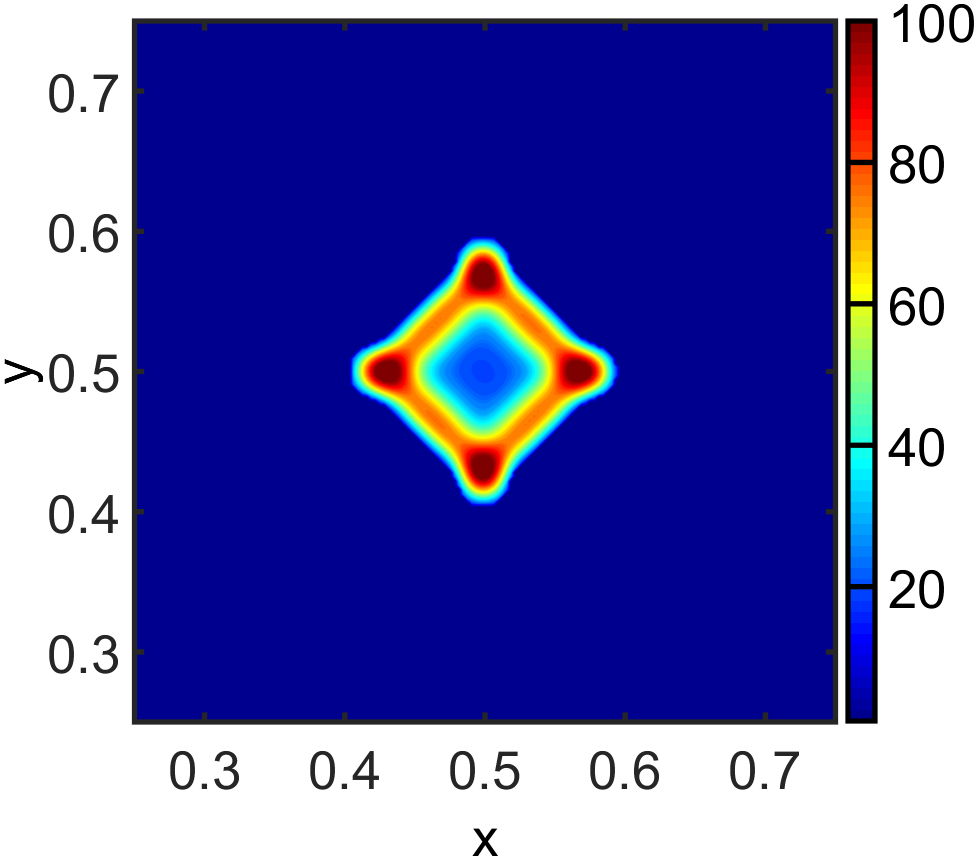}} 
	\subfloat[(b) $N = 1000$ particles.]
	{\includegraphics[width = 0.345\textwidth]{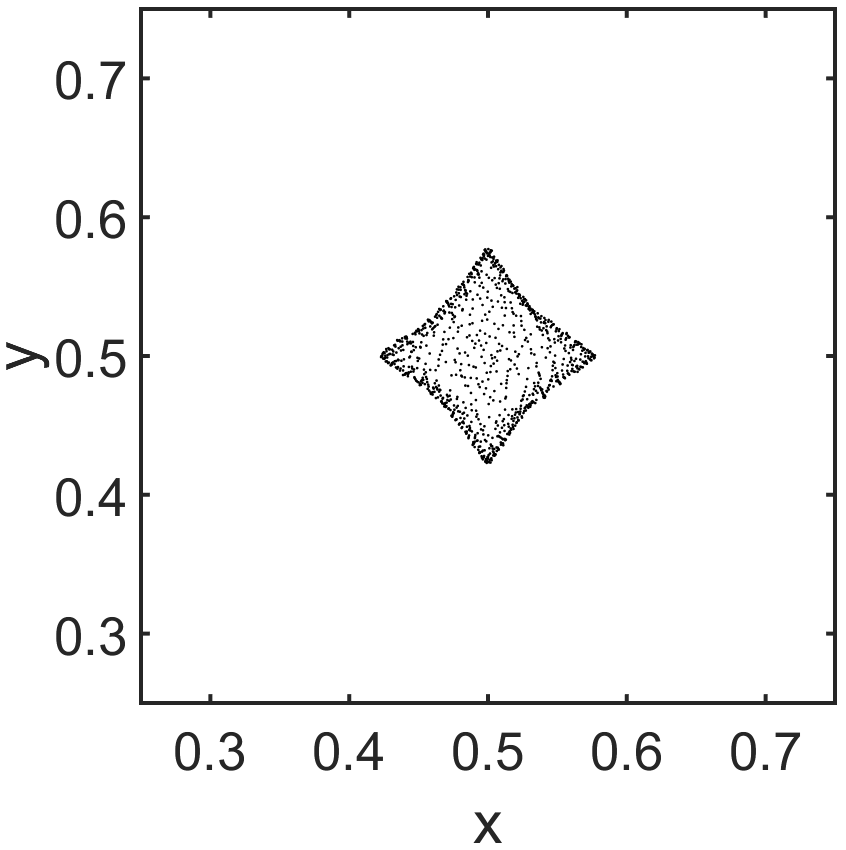}}
	\caption{A comparison of the recovered solution $\rho^*(\mathbf{x})$ 
	for the two dimensional periodic potential in equation (\ref{PeriodicMorse2d}) 
    (for parameters $(L, G) = (1.5, 0.9)$, and $40\times 40$ grid); with a 
    discrete steady state gradient flow. 
    Figure (a) shows the contour
    plot of $\rho^*(x)$, while Figure (b) shows the steady state solution of
    equation (\ref{DiscreteGradFlow}) with $N = 1000$ particles.  Note the 
	similarity in the support of $\rho^*(x)$ with the coalescence 
	of the individual particles.  }
	\label{Figure2d_ParticleExample}
\end{figure}

\subsection{Solutions $F_R(x)$ to (R) that are non-classical }
For values of $(L, G) = (0.5, 1.5)$, the solution $F_R(\mathbf{x})$ is a collection 
of discrete Dirac masses. Figure \ref{Discrete2d} shows the support of 
$F_R(\mathbf{x})$, $F_{\rho^*}(\mathbf{x})$ and $\rho^*(\mathbf{x})$.  
As evident by the small dots in Figure \ref{Discrete2d}c, 
this is a case where the recovered $\rho^*(\mathbf{x})$ has an auto-correlation
$F_{\rho}(\mathbf{x})$ that is not exactly inside $F_R(\mathbf{x})$, i.e., 
$\textrm{supp}(F_{\rho}) \nsubseteq \textrm{supp}(F_{R})$.  This implies that
Proposition \ref{Prop_supp} does not hold, and the recovered minimizer from
$F_R(\mathbf{x})$ is only an approximate one at best.  
One interesting observation, is that the recovery procedure successfully 
matches $93\%$ of the support of $F_{\rho^*}(\mathbf{x})$ 
with $F_R(\mathbf{x})$, so that $\rho^*(\mathbf{x})$ contains length scales
that try to optimize the overall energy $\mathcal{E}(\rho)$. However 
relative to the constant state $\rho(\mathbf{x}) = 1$, the guarantee 
is $\alpha = 0.54$, indicating there is a large gap between $\mathcal{E}(\rho^*)$ 
and the lower bound $\mathcal{E}_R$. For this example, it is possible that 
even the true global minimum $\rho_0(\mathbf{x})$ still has a large gap relative
to the bound $\mathcal{E}_R$. Figure \ref{Figure2d_ParticleExample_Discrete}
compares the support of $\rho^*(\mathbf{x})$ with the
	steady state arrangement of $N = 1000$ particles obtained
	from the gradient flow of equation (\ref{DiscreteGradFlow}) (using 
	random initial data). 	The figure shows that particles coalesce into 
	points that are not in a well defined pattern. 
	Finally, we remark that when the recovered minimizers $\rho^*(\mathbf{x})$ 
	have sharp spikes, the exact height and symmetry of the spikes obtained from 
	the Schultz-Snyder algorithm may become sensitive to small perturbations 
	in the target function $F_R(\mathbf{x})$. Developing alternative recovery 
	methods 
	with improved stability properties 
	may therefore be important in the future.

\begin{figure}[htb!] 
	\centering
    \subfloat[(a) $F_R(\mathbf{x})$, $F_{\rho^*}(\mathbf{x})$.]
    {\includegraphics[width = 0.32\textwidth]{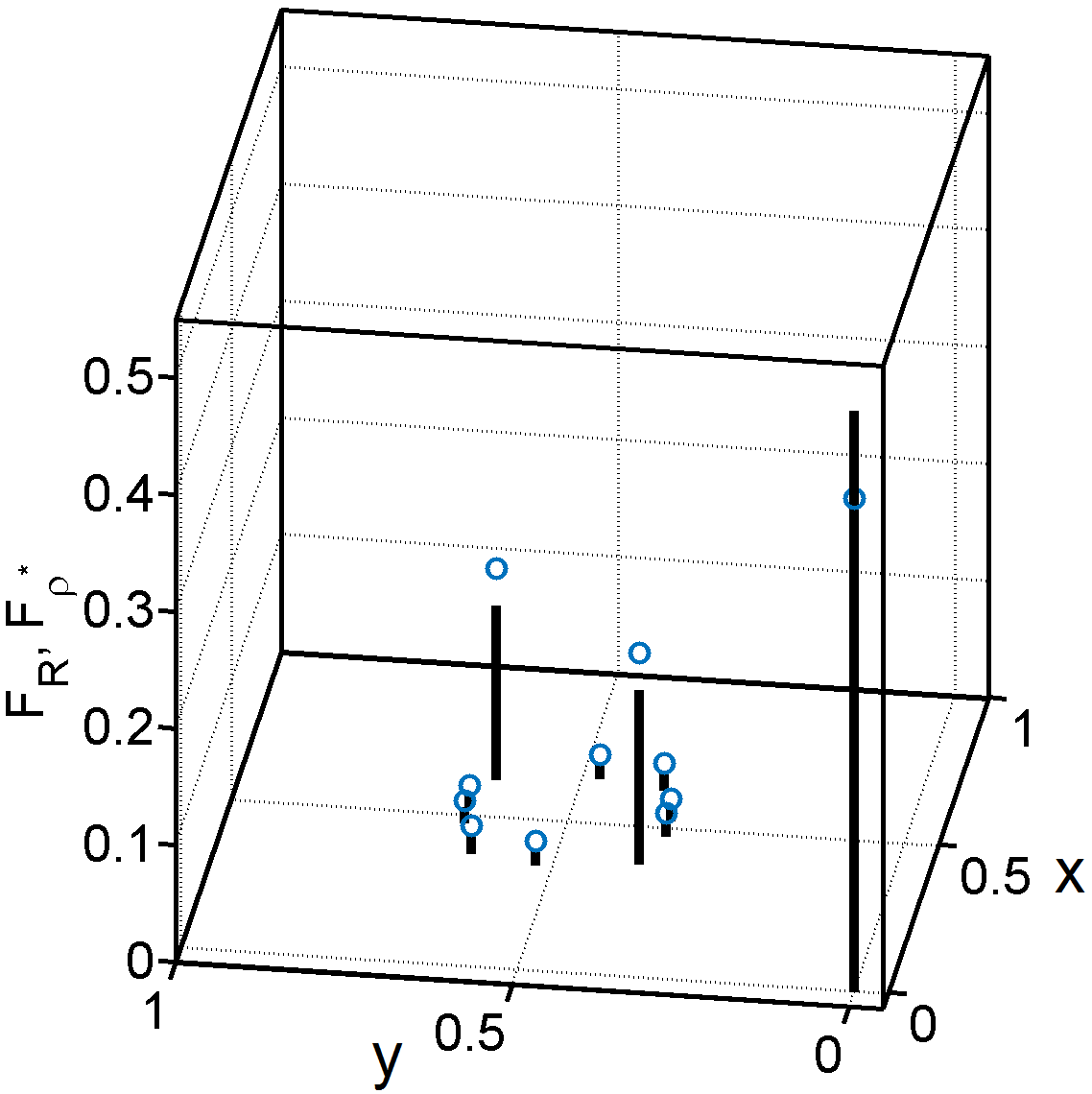} }
    \subfloat[(b) $F_R(\mathbf{x})$.]
    {\includegraphics[width = 0.3\textwidth]{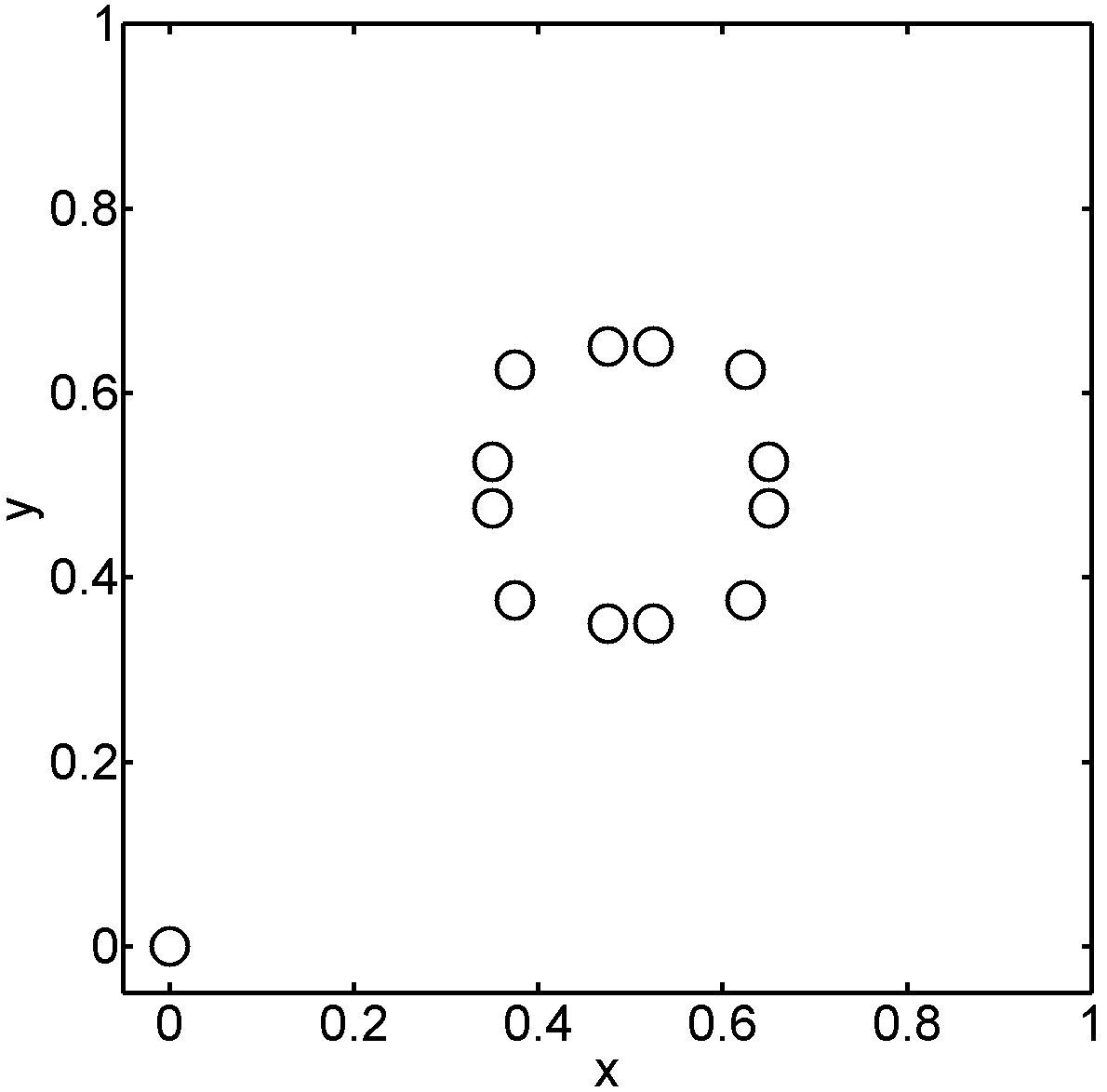} }
	\subfloat[(c) $F_{\rho^*}(\mathbf{x})$.]
	{\includegraphics[width = 0.3\textwidth]{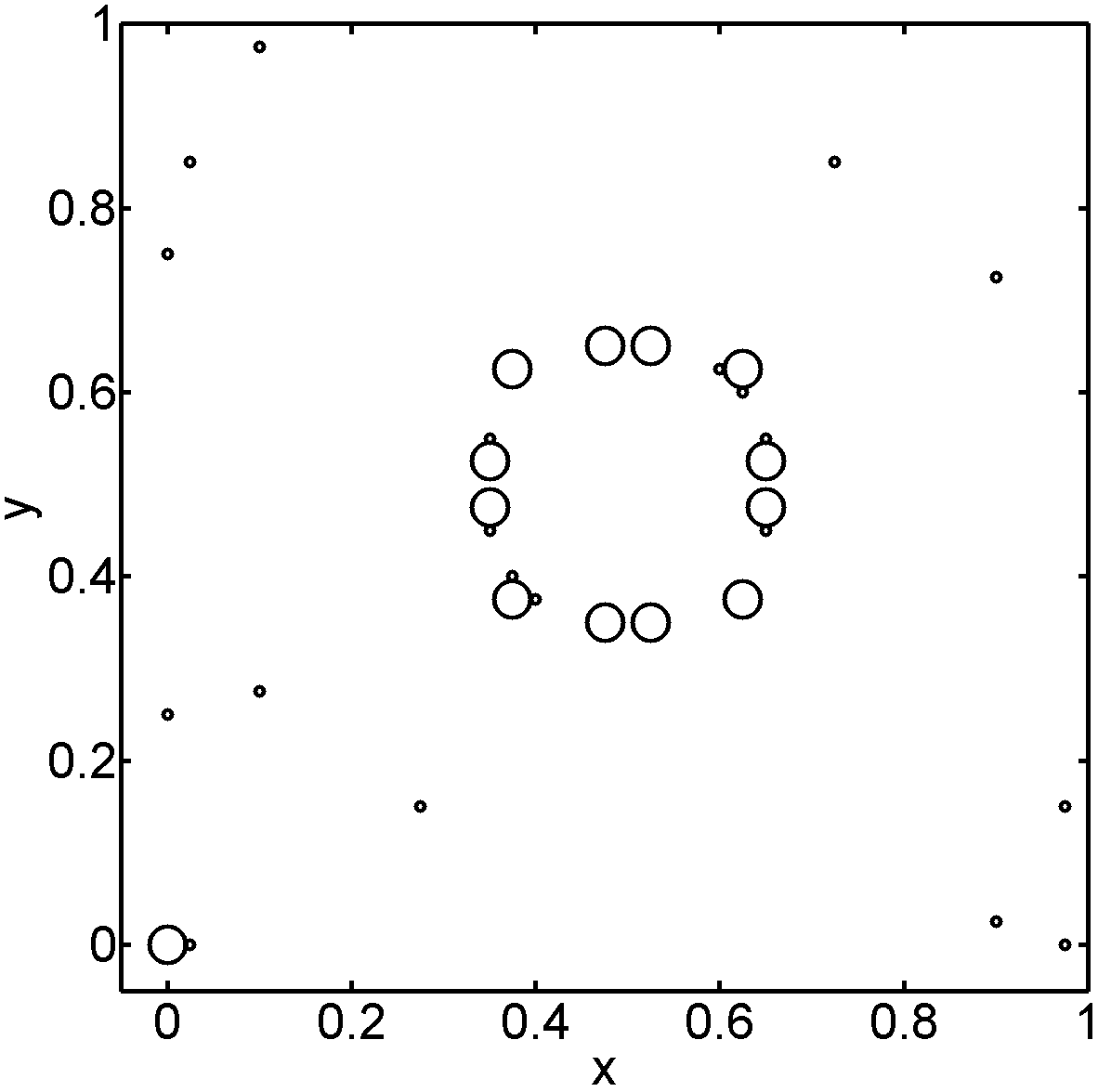}}
     \caption{Results for the two dimensional potential in equation (\ref{PeriodicMorse2d})
     with $(L, G) = (0.5, 1.5)$ and a $40\times 40$ grid. 
     Figure (a) shows $F_R(\mathbf{x})$ (black lines) and the
     recovered $F_{\rho^*}(\mathbf{x})$ (blue circles).
     Figure (b) shows the support of $F_R(\mathbf{x})$, while figure
     (c) shows the support of $F_{\rho^*}(\mathbf{x})$. The large circles account 
     for a total mass of $0.9267$, while the small circles 
     (each with mass $< 0.006$) account for the 
     remaining mass. 
     The value of the functional $\mathcal{F}(\rho^*) = 0.086$.}
     \label{Discrete2d}
\end{figure}

\begin{figure}[htb!] 
	\centering
   	\subfloat[(a) $\rho^*(\mathbf{x})$.]{\includegraphics[width = 0.32\textwidth]
   	{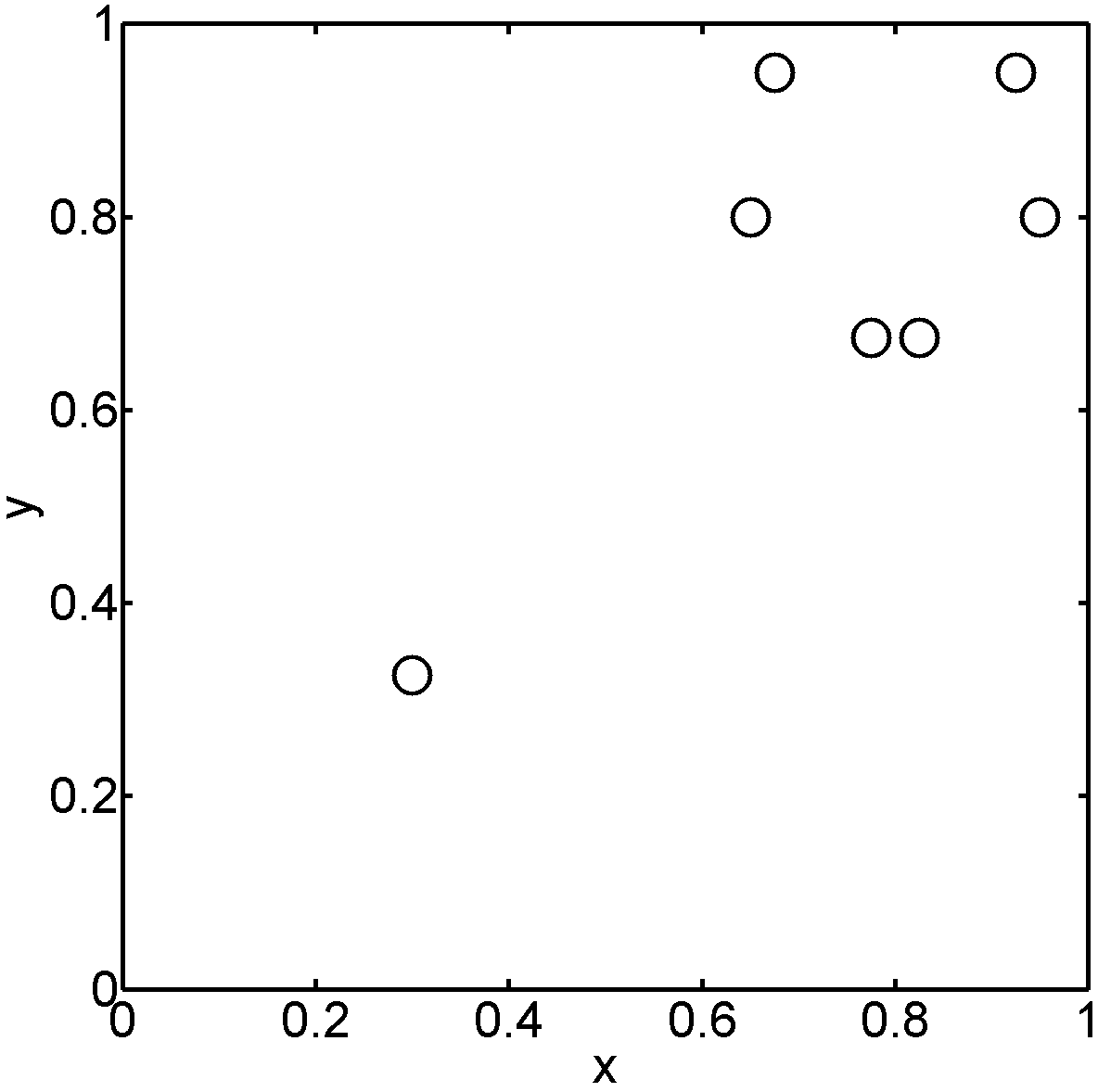}} 
	\subfloat[(b) $N = 1000$ particles.]
	{\includegraphics[width = 0.315\textwidth]{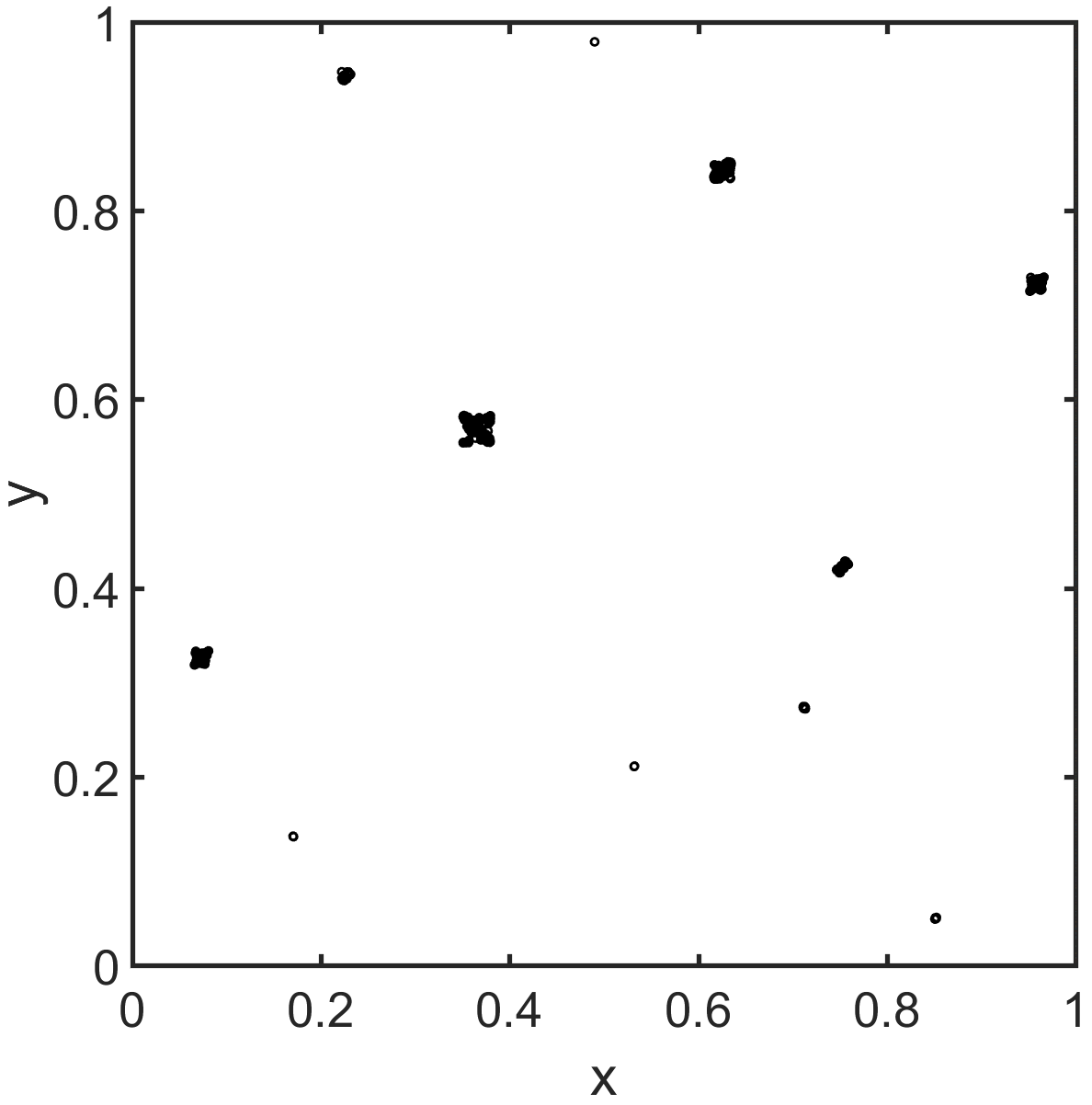}}
	\caption{A comparison of the recovered solution $\rho^*(\mathbf{x})$ 
	(with guarantee $\alpha = 0.54$) for the two dimensional periodic 
	potential in equation (\ref{PeriodicMorse2d}) and 
    parameters $(L, G)  = (0.5, 1.5)$, with a discrete steady state gradient flow. 
    Figure (a) shows the support of $\rho^*(\mathbf{x})$. The
    support contains $0.9894$ of the mass of $\rho^*(\mathbf{x})$. 
    Figure 
    (b) shows the steady state solution of
    equation (\ref{DiscreteGradFlow}) with $N = 1000$ particles. }
	\label{Figure2d_ParticleExample_Discrete}
\end{figure}


\section{Discussion and conclusions}

In this paper we provide a new approach for systematically computing 
approximate minimizers to an energy that models pairwise interactions.  
This is done by relaxing the non-convex optimization problem into a convex one 
to obtain a new sufficient condition for global minimizers.  A recovery 
procedure is then introduced as a way to find candidate minimizers that
satisfy the new sufficient condition (see Remark~\ref{Rmk:verification}). 
The advantage of the approach is that the resulting convex relaxation 
may be described analytically, which then leads to numerical 
descritizations of the new condition that may be solved using well-known methods. 

Analytically, the sufficient condition arises a lower bound to the 
minimum energy of the non-convex objective function.
The new lower bound then provides a way to quantify how optimal a candidate 
minimizer is.  The utility of the approach is demonstrated by the 
computation of a phase diagram for the periodic Morse potential, and also 
with the computation of minimizers for numerous interaction potentials 
in one and two dimensions.  For example, 
a lattice of Dirac masses is shown to be the global 
minimum, for specific parameter values, in the periodic Morse potential.  
Verifying that a lattice is a global minimizer to a non-local energy is a 
difficult problem in mathematical physics, with great practical interest (see
Remark \ref{Rmk:lattic_exact}).  Hence, new approaches that can show when a 
lattice minimizes a non-local energy are of theoretical interest.

Lastly, a fundamental problem in the minimization of pairwise energies over 
probabilities is to identify sets $S_*$ where the 
functional $\mathcal{E}(\rho)$ is convex, whenever the support of 
$\rho(\mathbf{x})$ is contained in $S_*$. To this 
end, our approach provides one way to identify such sets by exploiting
a dual optimization problem.  Specifically, the dual formulation results
in an optimal decomposition of the energy functional $\mathcal{E}(\rho)$
into the sum of a convex and non-convex functional.  The resulting 
convex/non-convex splitting can then be used to analytically identify 
supports in which $\mathcal{E}(\rho)$ is convex. 
From a physical perspective, this dual decomposition provides new insight into 
the natural length scales that many particle systems may self-assemble into; 
and may eventually help in designing and controlling pattern formation in many
particle systems.

\bigskip 
{\bf Acknowledgments:}
The author (DS) would like to thank Ihsan Topaloglu for many insightful 
comments regarding pairwise interaction problems.  The authors have also 
greatly benefited from conversations with Rustum Choksi, Robert Kohn, Cyrill 
Muratov and Jennifer Park.  The authors would also like to thank the anonymous
reviewers for many comments that helped to improve the presentation of
the paper. This work was supported by a 
grant from the Simons Foundation ($\#359610$, David Shirokoff); and partial 
support through the National Science Foundation, DMS--1719693 (Shirokoff).
\bigskip


\section*{Appendix A: Cases where the lower bound (R) is sharp}\label{AppendixA} 

There are several straightforward cases where the lower bound (R) is sharp, 
and the recovery $\rho^*(\mathbf{x})$ is guaranteed to be the exact global 
minimum--even when $\mathcal{E}(\rho)$ is non-convex.  In this section we 
outline the known cases where (R) is sharp. We also characterize the 
corresponding dual decomposition obtained from (D) in the known exact cases.

\begin{proposition} \label{Prop_DeltaMin}
	For a $W(\mathbf{x})$ satisfying properties (W1)--(W4), 
	$\rho_0(\mathbf{x}) = \delta(\mathbf{x})$ is a global minimizer to 
	(P) if and only if $W(\mathbf{0}) \leq W(\mathbf{x})$ for all 
	$\mathbf{x} \in \Omega$.  
\end{proposition}

\begin{proof}
	If $W(\mathbf{0}) \leq W(\mathbf{x})$ for all $\mathbf{x} \in \Omega$, 
	set $\rho_0(\mathbf{x}) = \delta(\mathbf{x})$. Then for any probability 
	distribution $\rho(\mathbf{x})$:
	\begin{align*}
		\mathcal{E}(\rho) &= 
		\frac{1}{2}\int_{\Omega}\int_{\Omega}W(\mathbf{x} - \mathbf{y}) 
		\rho(\mathbf{x})\rho(\mathbf{y}) \du \mathbf{x} \du \mathbf{y}, \\
		&\geq \frac{1}{2} W(\mathbf{0}) \int_{\Omega}\int_{\Omega} 
		\rho(\mathbf{x})\rho(\mathbf{y}) \du \mathbf{x} \du \mathbf{y} 
		= \mathcal{E}(\rho_0).
 	\end{align*} 
 	Hence $\rho_0(\mathbf{x})$ solves (P).
	To show the converse, take $\rho_0(\mathbf{x}) = \delta(\mathbf{x})$ as 
	a global minimizer to (P) and assume by contradiction there exists an 
	$\mathbf{s} \neq 0$ such that $W(\mathbf{s}) < W(\mathbf{0})$. Testing 
	the energy with a candidate 
	$\rho^*(\mathbf{x}) = \frac{1}{2}(\delta(\mathbf{x}) 
	+ \delta(\mathbf{x}-\mathbf{s}) )$ yields
	\begin{align}
		\mathcal{E}(\rho^*) &= 
		\frac{1}{4}\big( W(\mathbf{0}) + W(\mathbf{\mathbf{s}}) \big), \\
			&< \frac{1}{2} W(\mathbf{0}) = \mathcal{E}(\rho_0).
	\end{align}
	Hence, $\rho_0(\mathbf{x})$ cannot be a global minimizer and therefore 
	$W(\mathbf{0}) \leq W(\mathbf{s})$ for all $\mathbf{s}\in \Omega$. 
\end{proof}
\begin{remark}
	What is interesting about Proposition \ref{Prop_DeltaMin} is that the 
	condition on $W(\mathbf{x})$ does not at all imply that 
	$\mathcal{E}(\rho)$ is a convex functional.  As an example, 
	take $W(x) = -\cos(x) - \cos(2x) + 0.1 \; \cos(3x)$. Here $W(0)$ is the 
	minimum value of $W(x)$ yet $\mathcal{E}(\rho)$ is non-convex.
\end{remark}
%
The following simple proposition is known in the literature, however we 
repeat it here for completion.
\begin{proposition}\label{ConstantIsExact}
	Suppose $W(\mathbf{x})$ satisfies properties (W1)--(W4) and in addition
	satisfies property (\ref{Assumption_Abs_cont}), i.e.,
	\begin{align}\label{Assumption_Abs_cont_W}
		&\sum_{\mathbf{k}\in \mathbbm{Z}^d} \hat{W}(\mathbf{k}) < \infty, \quad
		\mathrm{where } \quad
		\hat{W}(\mathbf{k}) := \langle W, \cos(2\pi \mathbf{k}\cdot\mathbf{x}) \rangle. \\
		&\textrm{and } \quad W(\mathbf{x}) = 
		\sum_{\mathbf{k} \in \mathbbm{Z}^d } \hat{W}(\mathbf{k}) \cos(2\pi \mathbf{k}\cdot\mathbf{x}).
	\end{align}
	Then the function $\rho_0(\mathbf{x}) = 1$ is a global minimizer to (P) if 
	and only if 
	$\hat{W}(\mathbf{k}) \geq 0$ for all $\mathbf{k} \in \mathbbm{Z}^d$.
\end{proposition}
\begin{proof}
	If $\hat{W}(\mathbf{k}^*)  < 0$ 
	for some $\mathbf{k}^* \neq \mathbf{0}$, then 
	$\rho^*(\mathbf{x}) = 1 + \cos(2\pi\mathbf{x}\cdot\mathbf{k}^*) $ 
	has energy 
	\[
		\mathcal{E}(\rho^*) = \frac{1}{4} 
		\langle W, \cos(2\pi \mathbf{k}^*\cdot \mathbf{x})\rangle 
		< 0 = \mathcal{E}(1).
	\]
	Therefore the constant state is not the global minimum.  To show the converse, 
	substitute the cosine series
	expansion for $W(\mathbf{x})$ into $\mathcal{E}(\rho)$:
		\[ 
		\mathcal{E}(\rho) = \frac{1}{2}\sum_{\mathbf{k} \in 
		\mathbbm{Z}^d} \hat{W}(\mathbf{k})
		 \Big( \langle \rho, \cos(2\pi \mathbf{k}\cdot \mathbf{x})\rangle^2 
		+ \langle \rho, \sin(2\pi \mathbf{k}\cdot \mathbf{x})\rangle^2 \Big)
		 \geq 0. 
		\]
	This series is justified by the regularity assumption in 
	(\ref{Assumption_Abs_cont_W}).
	Since 
	$\hat{W}(\mathbf{k}) \geq 0$, the series expansion for $\mathcal{E}(\rho)$ over 
	$\mathbf{k}$ is always non-negative. Hence 
	$\mathcal{E}(\rho) \geq 0 = \mathcal{E}(1)$. 	
\end{proof}

\begin{proposition} \label{LowerBoundIsExact}
	Assume that $W(\mathbf{x})$ satisfies (W1)--(W4) and property 
	(\ref{Assumption_Abs_cont_W}).
	Then, the lower bound (R) is sharp when $\rho^*(\mathbf{x}) = 1$ or 
	$\rho^*(\mathbf{x}) = \delta(\mathbf{x})$ is a global minimum to (P).
\end{proposition}
\begin{proof}
	When $\rho^*(\mathbf{x}) = 1$, $W(\mathbf{x})$ has non-negative cosine 
	modes.  The lower bound functional in (R) may then be expanded
	in a cosine series (again which is justified by 
	(\ref{Assumption_Abs_cont_W})):
	\[ 
	\langle F, W\rangle = \sum_{\mathbf{k} \in \mathbbm{Z}^d} 
	\hat{F}(\mathbf{k}) \hat{W}(\mathbf{k}) \geq 0 
	= \langle 1, W\rangle = \mathcal{E}_0. 
	\]
	Hence $F(\mathbf{x}) = 1$ is the minimizer to (R) over continuous 
	functions, and $F_R(\mathbf{x}) = 1$ solves (R). 
	Alternatively, if $\rho^*(\mathbf{x}) = \delta(\mathbf{x})$ is a global 
	minimizer to (P), $W(\mathbf{0}) \leq W(\mathbf{x})$ for all 
	$\mathbf{x} \in \Omega$. Hence, for any probability distribution 
	$F(\mathbf{x})$:
	\[ 
		\mathcal{E}_0 = \frac{1}{2}W(\mathbf{0}) 
		= \frac{1}{2}\langle \delta(\mathbf{x}), W(\mathbf{x}) \rangle  
		\leq \langle F, W\rangle. 
	\]
	Therefore $F_R(\mathbf{x}) = \delta(\mathbf{x})$ solves (R) and is sharp.
	
	In both cases, when $F_R(\mathbf{x}) = 1$ and 
	$F_R(\mathbf{x}) = \delta(\mathbf{x})$, the solution $F_R(\mathbf{x})$ 
	satisfies $F_R \circ F_R = F_R$.  Hence, taking 
	$\rho^*(\mathbf{x}) = F_R(\mathbf{x})$, yields an exact recovery: 
	$F_R(\mathbf{x}) = \rho^* \circ \rho^*$.
\end{proof}

\begin{remark} 
	The cases discussed in Proposition \ref{LowerBoundIsExact} result in 
	simple optimal dual decompositions:
	\begin{itemize}	
		\item When $F_R(\mathbf{x}) = 1$ solves (R), the optimal dual 
		decomposition is 
	\[ 
		W_R^+(\mathbf{x}) = 0, 
		\quad K_R(\mathbf{x}) = W(\mathbf{x}), \quad \mathcal{E}_R = 0.
	\]
		\item When $F_R(\mathbf{x}) = \delta(\mathbf{x})$ solves (R) the 
		optimal dual decomposition is
	\[ 
	W_R^+(\mathbf{x}) = W(\mathbf{x}) - W(\mathbf{0}), 
	\quad K_R(\mathbf{x}) = 0, \quad \mathcal{E}_R = \frac{1}{2}W(\mathbf{0}). 
	\]
	\end{itemize}
\end{remark}



\section{Appendix B: Numerical solution of (R)} \label{sec:NumericalDetails}

In this section we present numerical details regarding the solution of (R) 
and dual decomposition (D). We discuss explicit details in dimension $d = 1$ 
and note that the extension to higher dimensions follows in a 
straightforward manner. To solve the relaxed problem, we use MATLAB's built 
in optimization routines, which require the construction of matrices 
representing the linear constraints in (R). 

Here we adopt the convention that vectors and matrices start with an index 
of $0$ (as opposed to MATLAB) so that row indices coincide with Fourier 
mode numbers. For the general problem (R) we discretize space with an
even number, $n > 0$, of points on an equispaced grid:
\[ 
	h = \frac{1}{n}, \quad  x_j = j h, \quad \textrm{ for } 0 \leq j \leq n-1.
\]

The functions $W(x)$ and $F(x)$ are then taken as $n$ dimensional vectors 
$\mathbf{w}, \mathbf{f} \in \mathbbm{R}^n$ so that:
\[ 
	\mathbf{w}_j \approx W(x_j), \quad \quad \mathbf{f}_j \approx F(x_j).
\]

There are two choices for imposing the mirror (or odd) symmetry of $\mathbf{f}$.
One can do it directly and set $\mathbf{f}_j = \mathbf{f}_{n-j}$, which 
will allow for a reduction in the number of variables to $n/2$; or one can 
build and enforce a sine constraint matrix. For efficiency reasons, we 
adopt the direct approach, however also describe how to construct the sine constraint
matrix.  

To build the matrices representing the sine and cosine constraints 
in (R), one may use the rows in the discrete Fourier transform matrix 
obtained via the fast Fourier transform.  Meanwhile, for the non-negativity 
constraint
in (R), one may either use the MATLAB's built in 
non-negativity constraint option, or directly enforce non-negativity 
by passing the MATLAB routine a constraint matrix.  Regardless of the 
option one uses, the three $n\times n$ constraint matrices can be constructed as 
follows:
\begin{align} \nonumber
	\textrm{Non-negative constraint matrix:}\quad\quad
		\mathbf{P}_{lj} &= -\delta_{lj},  \hspace{28.2mm} 
		0 \leq l, j \leq n-1, \\ \nonumber
	\textrm{Cosine mode matrix:} \quad\quad  
	\mathbf{C}_{l, :} &= -\textrm{{\fontfamily{qcr}\selectfont 
	real(fft(}}\mathbf{e}_l\textrm{{\fontfamily{qcr}\selectfont))}}, 
	\quad 0 \leq l \leq n-1, \\ \nonumber
	\textrm{Sine mode matrix:}\quad\quad  \mathbf{S}_{l, :} 
	&= \phantom{-} \textrm{{\fontfamily{qcr}\selectfont 
	imag(fft(}}\mathbf{e}_l\textrm{{\fontfamily{qcr}\selectfont))}}, 
	\quad 0 \leq l \leq n-1. 	
\end{align}

Here $\mathbf{C}_{l, :}$ and $\mathbf{S}_{l, :}$ are the entire $l^{th}$ 
matrix row, $\delta_{lj}$ is the Kronecker delta, and $\mathbf{e}_l$ is the 
$l^{th}$ row of the $n\times n$ identity matrix:
\begin{align} \nonumber
	\delta_{lj} = \left\{
\begin{array}{rl}
1 & \text{if } l = j,\\
0 & \text{if } l \neq j.\\
\end{array} \right. \hspace{10mm}
\mathbf{e}_l = \begin{bmatrix}
	 0,  0, \ldots, 0, 1, 0, \ldots, 0 
\end{bmatrix}.
\end{align}
By construction, the matrices have components 
$\mathbf{C}_{kj} = -\cos(2\pi k j h)$, $\mathbf{S}_{kj} = \sin(2\pi k j h)$ 
so that cosine and sine integrals are approximated via
\begin{align}
	-\langle \cos(2\pi k x), F(x) \rangle \approx h \sum_{j = 0}^{n-1} 
	\mathbf{C}_{k j} \mathbf{f}_j, \hspace{10mm} 	
	\langle \sin(2\pi k x), F(x) \rangle \approx h \sum_{j = 0}^{n-1} 
	\mathbf{S}_{k j} \mathbf{f}_j.
\end{align}
To write the mass constraint in (R) explicitly, we also introduce the unit 
vector
\begin{align} \nonumber
	\mathbf{1} = \begin{bmatrix}
	 1,  1,  1, \ldots, 1
\end{bmatrix}^T \in \mathbbm{R}^n.
\end{align}
Finally, note that by symmetry, the bottom half of the rows in matrices 
$\mathbf{C}$ and $\mathbf{S}$ are redundant. It is therefore sufficient to 
enforce constraints for only the rows of $l$ with 
$1 \leq l \leq \lfloor \frac{n}{2} \rfloor$ where 
\[  
	\Big\lfloor \frac{n}{2} \Big\rfloor =  \left\{
	\begin{array}{cl}
	\frac{n}{2} & \text{if } n \textrm{ is even},\\
	\frac{n-1}{2} & \text{if } n \textrm{ is odd}.\\
	\end{array} \right.
\]

The problem (R) then takes the discrete standard form:
\begin{align}  \nonumber
	(R_h) \quad \quad &\textrm{Minimize } 
	\quad \frac{1}{2}\mathbf{w}^T \mathbf{f} \\ \nonumber
	 &\textrm{subject to} \quad \mathbf{P} \; \mathbf{f} \leq 0, \\ \nonumber
	 &\phantom{\textrm{subject to}}\quad \mathbf{C}_{k,:} \; 
	 \mathbf{f} \leq 0, \hspace{10mm} 1 \leq k \leq \big\lfloor \frac{n}{2} 
	 \big\rfloor, \\ \nonumber
	 &\phantom{\textrm{subject to}}\quad \mathbf{S}_{k,:} \;\hspace{0.8mm} 
	 \mathbf{f} = 0, \hspace{10mm} 1 \leq k \leq \big\lfloor \frac{n}{2} 
	 \big\rfloor, \\ \nonumber
 	 &\phantom{\textrm{subject to}}\quad 	h \; \mathbf{1}^T \mathbf{f} = 1. 			  
\end{align}

Problem $(R_h)$ is then solved using a standard linear programming package 
with an interior-point algorithm.  We use MATLAB's 
{\fontfamily{qcr}\selectfont linprog} routine, with a tolerance set to 
$10^{-8}$.  In pseudo-code, the command takes the form:
\begin{align}\nonumber
	[\mathbf{f}_R, \; \mathcal{E}_R, \; \mathbf{W}^+, \; 
	\mathbf{K}] = \textrm{{\fontfamily{qcr}\selectfont linprog}}(\mathbf{w}, 
	\emph{constraint matrices} \; \mathbf{P}, \mathbf{C}, \mathbf{S}, \mathbf{1} ).
\end{align}
The output then consists of the optimal solution vector $\mathbf{f}_R$, 
the optimal solution value $\mathcal{E}_R$, as well as the dual decomposition 
vectors $\mathbf{W}^+$ and $\mathbf{K}$.  In other words, the dual decomposition
comes for free.

We identify two qualitatively different solutions $\mathbf{f}_R$ to problem 
$(R_h)$:
\begin{enumerate}[leftmargin = 1.4cm]
	\item[Case 1.] The solution $\mathbf{f}_R$ converges as $h\rightarrow 0$, 
	to a $C^0(\Omega)$ function with no Dirac mass singularities. In this case, 
	the procedure from Section \ref{Sec_KLrecovery} is used to recover a 
	discrete $\rho^*(\mathbf{x})$ from $\mathbf{f}_R$. The vector 
	$\rho^*(\mathbf{x})$ is discretized 
	using $n$ grid points on the same lattice as $\mathbf{f}_R$. The integrals 
	in the continuous Schulz-Snyder algorithm are also computed using 
	vectorized dot products (the standard midpoint rule is spectrally accurate 
	for smooth solutions on periodic domains and lower order for non-smooth 
	solutions $F_R(\mathbf{x})$).  The discrete $\rho^*(\mathbf{x})$ is computed to within steady-state 
	tolerances $\textrm{tol}_1$, $\textrm{tol}_2$ so that the discrete 
	quantities satisfy
	\begin{align}
		&\mathcal{F}(\rho_{n}) - \mathcal{F}(\rho_{n+1}) < \textrm{tol}_1, 
		\quad \quad ||\rho_{n+1} - \rho_n||_{L^1(\Omega_h)} < \textrm{tol}_2.\\
		&\textrm{where}\quad ||f||_{L^1(\Omega_h)} := h \sum_{j = 0}^{n-1} |f_j|.
	\end{align}

		
	\item[Case 2.] The solution $\mathbf{f} \rightarrow F_R(x)$ converges in 
	distribution to a set of delta distributions as $h \rightarrow 0$. Namely, 
	for any smooth function $u(x)$ and corresponding discrete vector 
	$\mathbf{u}$, the value $h(\mathbf{u}^T \mathbf{f}) \rightarrow 
	\langle u(x), F_R(x) \rangle$ converges as $h\rightarrow 0$.   
\end{enumerate}

\begin{remark}
	In case 2, one may obtain Delta masses in 
	$\mathbf{f}_R$ with a support of one mesh point each by modifying
	the grid size $h$ to naturally 
	accommodate the spacings between the Delta masses.  To do this: (i) Obtain a solution 
	$\mathbf{f}_R$ (that may have Delta masses smeared over a few grid points) 
	to $(R_h)$ with a suitably fine mesh $h$; (ii) Estimate the 
	distance between the Dirac masses in $\mathbf{f}_R$; (iii)
	Take a new grid spacing $h'$, such that the distance between Delta
	masses is an integer multiple of $h'$; (iv) Resolve the discrete problem $(R_{h'})$ using the new grid $h'$ to obtain improved convergence. 
	Improvements in the linear 
	programming time were also observed when choosing a grid spacing $h'$ 
	that is commensurate with the spacings of the Dirac deltas.
\end{remark}

\begin{remark}
	In two dimensions, we found that the solution $\mathbf{f}_R$ to the linear 
	program ($R_h$) can become sensitive to the exact number $n$, the 
	prescribed tolerance, and the allowable number of interior point iterations. 
\end{remark}

\begin{remark}
	We systematically ran hundreds of recovery tests and found that the 
	Schulz-Snyder algorithm often converged to the same value 
	$\mathcal{F}(\rho_{\infty})$ within numerical error.  We did, however 
	observe that when $F_R(\mathbf{x})$ was a discrete probability measure in two 
	dimensions, there were multiple $\rho^*(\mathbf{x})$ that minimized 
	$\mathcal{F}(\rho)$. The different $\rho^*(\mathbf{x})$ had almost 
	the same recovery guarantees $\alpha$ to within $\pm 0.02$.   
\end{remark}


\section*{Appendix C: Periodic effects for the solution to (R)} \label{DiscreteEffects1D}

The purpose of this section is to examine a simple sub-class 
of minimizers $F_R(x)$ to (R) in one dimension. 
We show that provided $W(x)$ satisifies a few
regularity properties, minimizers within this sub-class always have 
spacings that are commensurate with a discrete lattice.  This will turn out to be a direct result of the periodic domain $\Omega$. 

In this section, we consider the restricted set of probabilities 
\begin{align} \label{Subclass_3Delta}
	F(x) &= \alpha \delta(x) + \beta \delta(x - s) + \beta \delta(x + s), \\ \nonumber
	\alpha &+ 2\beta = 1, \quad \quad 0 \leq s \leq \frac{1}{2}.
\end{align}
The subclass (\ref{Subclass_3Delta}) is then completely characterized by two 
parameters $(s, \beta)$.  The lower bound problem (R), restricted to the 
probabilities (\ref{Subclass_3Delta}) with three delta masses, is:
\[	
	(R_3) \quad \textrm{ minimize } \quad 
	\frac{1}{2} \langle W(x), F(x)\rangle = \frac{1}{2}W(0) 
	+ \beta \big( W(s)-W(0)\big). 
\]

We now outline why unique minimizers of the form (\ref{Subclass_3Delta}) to 
($R_3$), characterized by values $(s^*, \beta^*)$, often have support 
commensurate with a lattice: that is $s^*\in \mathbbm{Q}$ is a rational number.

	First optimize the energy $(R_3)$ at a fixed $s$, over the weight $\beta$, 
	thereby yielding a function only of $s$:
	\begin{align}\nonumber
		E(s) &:= \frac{1}{2}W(0) +  \inf_{\beta} \Big[ \beta \; 
		\big(W(s) - W(0) \big) \Big]. 
	\end{align}
	If $W(0) \leq W(s)$
	then $E(s) = \frac{1}{2}W(0)$, which occurs when $\beta = 0$. 
	If $W(0) > W(s)$, then $E(s)$ takes the following form 
	\begin{align} \nonumber
		E(s) &= \frac{1}{2}W(0) + \theta(s) \big( W(s) - W(0) \big), \\
		\theta(s) &:= \sup \beta, \quad \textrm{ subject to }\quad 
		(1-2\beta) \delta(x) + \beta \delta(x - s) + \beta \delta(x + s) 
		\in \mathcal{C}. 
	\end{align}	
	The function $\theta(s)$ can be computed by examining the 
	convex cone constraint $F(x) \in \mathcal{C}$:
\[ 
	\langle F, \cos(2\pi k x) \rangle = 1-2\beta + 2 \beta \cos(2\pi k s) 
	\geq 0, \quad \textrm{ for all } k \in \mathbbm{Z} \setminus 0. 
\]
Hence,
\[ 
0 \leq \beta \leq \frac{1}{2(1- \cos(2\pi k s))} \quad \textrm{ for all } 
k \in \mathbbm{Z} \setminus 0. 
\]
It follows that, 
\begin{align} \nonumber
	\theta(s) &= \inf_{k \in \mathbbm{Z} \setminus 0} \frac{1}{2(1- \cos(2\pi k s))}.
\end{align}
When the value of $s$ is irrational (denote by $\overline{\mathbbm{Q}}$), 
$\cos(2\pi k s)$ can be made arbitrarily 
close to $-1$:
\[ 
	\theta(s) = \frac{1}{4} \text{  for } s \in \overline{\mathbbm{Q}}, 
	\quad \quad \theta(s) = \min_{0 \leq k \leq p}\frac{1}{2(1-\cos(2\pi k s))} 
	\textrm{  for } s = \frac{q}{p} \in \mathbbm{Q}.
\]
An immediate consequence is that $\theta(q/p) \geq \theta(s)$ for all 
$q/p \in \mathbbm{Q}$ (rational) and $s \in \overline{\mathbbm{Q}}$ 
(irrational).  The function $\theta(s)$ also has interesting continuity 
properties:
%
%
\begin{proposition} \label{Theta_Continuity}
	The function $\theta(s)$, for $0 < s < \frac{1}{2}$, is continuous at all
	$s \in \overline{\mathbbm{Q}} \cup \mathbb{Q}_e$, and discontinuous at 
	all $s \in \mathbbm{Q}_o$, where 
	$\mathbbm{Q} = \mathbbm{Q}_e \cup \mathbbm{Q}_o$,
	\[ 
		\mathbbm{Q}_e = \{ q/p \in \mathbbm{Q} : \textrm{gcd}(q, p) = 1, p 
		\textrm{ even}\}, \quad\quad \mathbbm{Q}_o = \{ q/p \in \mathbbm{Q} 
		: \textrm{gcd}(q, p) = 1, p \textrm{ odd}\}.
	\]
\end{proposition}
%
%
\begin{proof}
	For simplicity in the proof first introduce
	\[ 
	\tilde{\theta}(s) := -1 \textrm{ for } s \in \overline{\mathbbm{Q}}, 
	\quad \tilde{\theta}(s) = \min_{0 \leq k \leq p} \cos(2\pi k s) 
	\textrm{ for } s := \frac{q}{p} \in \mathbbm{Q}. 
	\]
	Since $\theta(s)$ is a composition of a continuous function with 
	$\tilde{\theta}(s)$, it is sufficient to prove Proposition~\ref{Theta_Continuity} 
	for the modified function $\tilde{\theta}(s)$ 
	instead of $\theta(s)$.
	
	First we remark on the value of $\tilde{\theta}(q/p)$ for integers $q, p$ 
	with $\textrm{gcd}(q,p) =1$: there exists an integer $k^* > 0$ such that
	\begin{align} \nonumber
		k^* q &\equiv \frac{p}{2}  \;(\textrm{mod } p) \textrm{ if } p 
		\textrm{ is even}, \\ \nonumber
		k^* q &\equiv \frac{p+1}{2} \;(\textrm{mod } p) \textrm{ if } p 
		\textrm{ is odd}. 
	\end{align}
	Hence the optimal value of $\tilde{\theta}(s)$ is given by
	\begin{align}\nonumber
		\tilde{\theta}(q/p) &= \cos\Big(\frac{2\pi k^* q}{p} \Big) 
		= \cos(\pi) = -1,  \hspace{26mm} \textrm{ if } p \textrm{ is even}, \\ \nonumber
		\tilde{\theta}(q/p) &= \cos\Big(\frac{2\pi k^* q}{p}\Big) 
		= \cos\Big(\pi + \frac{\pi}{p}\Big) = -\cos\Big(\frac{\pi}{p}\Big), 
		\quad \textrm{ if } p \textrm{ is odd}.
	\end{align}	
	If $s_0 \in \mathbbm{Q}_o$, then any sequence $s_j\rightarrow s_0$ with 
	$s_j \in \overline{\mathbbm{Q}}$ has 
	\[ 
	\tilde{\theta}(s_0) - \tilde{\theta}(s_j) = \tilde{\theta}(s_0)+1 
	> \delta > 0, \textrm{ for some } \delta.
	\]  
	Hence $\tilde{\theta}(s)$ is discontinuous at $\mathbbm{Q}_o$. 

	For continuity at a point $s_0 \in \overline{\mathbbm{Q}}\cup\mathbb{Q}_e$, 
	let $\epsilon > 0$. Clearly for any 
	$s \in \overline{\mathbbm{Q}}\cup\mathbbm{Q}_e$, 
	\[ 
	|\tilde{\theta}(s) - \tilde{\theta}(s_0)| = 0 < \epsilon.
	\]
	To examine the behavior of $s \in \mathbbm{Q}_o$, fix 
	$t = \lceil \epsilon^{-1}\rceil$ as the smallest integer larger than 
	$\epsilon^{-1}$. Note that there are only a finite number of rational 
	numbers $q/p$ with $p \leq t$ (and greatest common divisor 
	$\textrm{gcd}(q,p) = 1$) in the interval $s_0 - 1 < q/p < s_0 + 1$.  
	Hence, for any $\epsilon > 0$, one may take $\delta = \delta(\epsilon)$ small 
	enough so that rational value $q/p$ satisfying $|s_0 - q/p| < \delta$,
	must have $p > t$.
	Consequently for any rational $q/p \in \mathbbm{Q}_o$:
	\begin{align} \nonumber
		&\Longrightarrow \Big|\tilde{\theta}(q/p) - \tilde{\theta}(s_0) \Big| 
		= \Big|\cos\Big(\pi + \frac{\pi}{p}\Big) - \cos(\pi)\Big|\leq \frac{\pi}{p} 
		\leq \frac{\pi}{t} \leq \pi \epsilon.
	\end{align}
	In the last line we have used a Lipschitz constant of $1$ for cosine. 
	This concludes the proof. 	 		
\end{proof}

Proposition~\ref{Theta_Continuity} now leads to the following result:
if $W(s)$ is smooth enough, $s^*$ must be rational.
\begin{proposition} \label{Prop_RationalPoints}
	Suppose $W(s)$ satisfies (W1)--(W4) and has bounded second derivative on 
	$(0,1)$. Assume also that $W(0) > \min_{0 < x < 1} W(x)$ is not strict minimum
	value of $W(x)$.
	Fix $0 < s^* < \frac{1}{2}$ with 
	$s^* \in \overline{\mathbbm{Q}} \cup \mathbbm{Q}_e$. 
	Then $s^*$ does not minimize $E(s)$. 
\end{proposition}
%
\begin{proof}
	We assume that $s^* \in \overline{\mathbbm{Q}}\cup \mathbbm{Q}_e$ 
	minimizes $E(s)$, and then arrive at a contradiction.
	First observe that if $s^*\in \overline{\mathbbm{Q}}\cup \mathbbm{Q}_e$ 
	and minimizes $E(s)$, then $s^*$ must also minimize $W(s)$. This is
	because 
	$E(s^*) = \frac{1}{4}( W(s^*) + W(0) )$ whenever 
	$s^* \in \overline{\mathbbm{Q}}\cup \mathbbm{Q}_e$. Hence, by continuity
	of $W(s)$, $s^*$ must minimize $W(s)$ and therefore
	$W'(s^*) = 0$. Using Taylor's remainder theorem, 
	there exists a constant $C$ such that for any $s$ in the 
	neighborhood of $s^*$,
	\[ 
	|W(s) - W(s^*) | \leq C |s - s^*|^2. 
	\]
	We now argue that one can find a rational point close to $s^*$ that has a
	lower value of $E(s)$ than $E(s^*)$.  Using basic properties of cosine, as well as the result 
	from Proposition \ref{Theta_Continuity}, one has for any rational point 
	$q/p \in \mathbbm{Q}_o$ in the neighborhood of $s^*$, there exists a 
	$c_1 > 0$ such that
	\[ 
	\Big|\tilde{\theta}(q/p) - \tilde{\theta}(s^*)\Big| 
	= \Big|\cos\Big(\frac{\pi}{p} + \pi\Big)- \cos(\pi)\Big| \geq \frac{c_1}{p^2}.
	\]
	Now consider a sequence of approximating rational points 
	$s_j = q_j/p_j \rightarrow s^*$, for $j > 0$ with $s_j \in \mathbbm{Q}$ 
	that by a well-known theorem from continued fractions 
	\cite{JonesThron1980} satisfy
	\[ 
	\Big| \frac{q_j}{p_j} - s^* \Big| \leq \frac{c_2}{p_j^2}.
	\]
	An important remark, is that the sequence of $p_j$ generated via continued 
	fractions have $p_j$ odd infinitely often.  Therefore, without loss of 
	generality we may restrict the sequence $s_j$ to a sub-sequence on  
	$\mathbbm{Q}_o$ that has $p_j$ odd\footnote{An infinite continued 
	fraction can be represented as a unique sequence of positive integers 
	$(a_0, a_1, a_2, \ldots)$.  The rational $s_j = q_j/p_j$ approximations 
	satisfy the recursion relations $q_0 = a_0$, $p_0 = 1$, 
	$q_1 = a_1 a_0 + 1$, $p_1 = a_1$, and 
	$q_j = a_j q_{j-1}+q_{j-2}, p_j = a_j p_{j-1} + p_{j-2}$, for 
	$j \geq 2$. Therefore using an induction argument, one can show that if 
	$(p_{j-1}, p_{j})$ has at least one odd term, then $(p_{j+1}, p_{j+2})$ 
	also has one odd term. Since $(p_0, p_1) = (1, a_1)$, the result follows.}.
	
	Hence, combining the previous two inequalities, on this sequence 
	$s_j \in \mathbbm{Q}_o$
	\[ 
	\big| \theta(s_j) - \theta(s^*)\big| \geq c_3\big|s_j -s^* \big|.
	\]	
	By direct calculation, for $j$ sufficiently large, 
	\begin{align}
		E(s_j) - E(s^*) &= \theta(s_j) \Big( W(s_j) - W(s^*)\Big) 
		+ \Big( \theta(s_j) - \theta(s^*) \Big) \Big( W(s^*) - W(0) \Big), \\
		&\leq A_1 |s_j - s^*|^{2} - A_2 |s_j - s^*|,
	\end{align}
	where $A_1 > 0$ is an upper bound on $\theta(s_j)$ and the Taylor 
	constant, $A_2 = c_3 (W(0) - W(s^*)) > 0$. Finally, for sufficiently 
	large $j$ one has $A_2 |s_j - s^*| > A_1 |s_j-s^*|^{2}$ implying
	\[ 
		E(s_j) < E(s^*).
	\]
	Thus, $s^*$ cannot minimize $E(s)$.
\end{proof}
The purpose of Proposition \ref{Prop_RationalPoints} is to observe 
that if $W(s)$ is smooth enough on $(0,1)$, then minimizes of the form
(\ref{Subclass_3Delta}) must have rational spacings.

\bibliographystyle{siam}
\bibliography{bibliography}

\end{document}